\documentclass[12pt]{amsart}%

\usepackage{amsmath,amsthm,amssymb,enumitem,verbatim,stmaryrd,xcolor,microtype,graphicx,aliascnt,needspace}
\usepackage[T1]{fontenc}
\usepackage[utf8]{inputenc}
\usepackage[english]{babel} 
\usepackage[top=3.5cm,bottom=3.55cm,left=3.5cm,right=3.5cm]{geometry}
\usepackage[bookmarksdepth=2,linktoc=page,colorlinks,linkcolor={red!70!black},citecolor={red!80!black},urlcolor={blue!80!black},pdfpagemode=UseOutlines,pdfstartview={fitH}]{hyperref}
\usepackage{tikz}\usetikzlibrary{matrix,arrows,decorations.markings,shapes,backgrounds}
\usepackage{tikz-cd}
\usepackage[mathcal]{euscript}                               

\usepackage{mathptmx}                                        
\usepackage[colorinlistoftodos,bordercolor=orange,backgroundcolor=orange!20,linecolor=orange,textsize=footnotesize]{todonotes} \makeatletter \providecommand \@dotsep{5} \def\listtodoname{List of Todos} \def\listoftodos{\@starttoc{tdo}\listtodoname} \makeatother   

\allowdisplaybreaks 
\usepackage{etoolbox}
\makeatletter
\patchcmd{\@startsection}{\@afterindenttrue}{\@afterindentfalse}{}{}             
\patchcmd{\part}{\bfseries}{\bfseries\LARGE}{}{}
\patchcmd{\section}{\scshape}{\bfseries}{}{}\renewcommand{\@secnumfont}{\bfseries} 
\patchcmd{\@settitle}{\uppercasenonmath\@title}{\large}{}{}
\patchcmd{\@setauthors}{\MakeUppercase}{}{}{}
\addto{\captionsenglish}{} 
\makeatother

\usepackage{fancyhdr}

\addto\extrasenglish{}  
\theoremstyle{plain}
\newtheorem{thm}{Theorem}[section] 
\newaliascnt{theorem}{thm}\aliascntresetthe{theorem}
\newaliascnt{lemma}{thm}\newtheorem{lemma}[lemma]{Lemma}\aliascntresetthe{lemma}
\newaliascnt{cor}{thm}\newtheorem{cor}[cor]{Corollary}\aliascntresetthe{cor}
\newaliascnt{prop}{thm}\newtheorem{prop}[prop]{Proposition}\aliascntresetthe{prop}

\newtheorem*{thm*}{Theorem}
\newtheorem*{lem*}{Lemma}
\newtheorem*{cor*}{Corollary}

\theoremstyle{definition}
\newaliascnt{df}{thm}\newtheorem{df}[df]{Definition}\aliascntresetthe{df}
\newaliascnt{rem}{thm}\newtheorem{rem}[rem]{Remark}\aliascntresetthe{rem}
\newaliascnt{example}{thm}\aliascntresetthe{example}
\newaliascnt{ex}{thm}\newtheorem{ex}[ex]{Example}\aliascntresetthe{ex}

\newtheorem*{df*}{Definition}
\newtheorem*{ex*}{Example}
\newtheorem*{rem*}{Remark}

\numberwithin{equation}{section}

\DeclareFontFamily{OT1}{pzc}{}                                
\DeclareFontShape{OT1}{pzc}{m}{it}{<-> s * [1.10] pzcmi7t}{}
\DeclareMathAlphabet{\mathpzc}{OT1}{pzc}{m}{it}
\DeclareSymbolFont{sfoperators}{OT1}{bch}{m}{n} 
\DeclareSymbolFontAlphabet{\mathsf}{sfoperators} 
\makeatletter\def\operator@font{\mathgroup\symsfoperators}\makeatother 
\DeclareSymbolFont{cmletters}{OML}{cmm}{m}{it}              
\DeclareSymbolFont{cmsymbols}{OMS}{cmsy}{m}{n}
\DeclareSymbolFont{cmlargesymbols}{OMX}{cmex}{m}{n}
\DeclareMathSymbol{\myjmath}{\mathord}{cmletters}{"7C}     \let\jmath\myjmath 
\DeclareMathSymbol{\myamalg}{\mathbin}{cmsymbols}{"71}     
\DeclareMathSymbol{\mycoprod}{\mathop}{cmlargesymbols}{"60}\let\coprod\mycoprod
\DeclareMathSymbol{\myalpha}{\mathord}{cmletters}{"0B}     \let\alpha\myalpha 
\DeclareMathSymbol{\mybeta}{\mathord}{cmletters}{"0C}      \let\beta\mybeta
\DeclareMathSymbol{\mygamma}{\mathord}{cmletters}{"0D}     \let\gamma\mygamma
\DeclareMathSymbol{\mydelta}{\mathord}{cmletters}{"0E}     \let\delta\mydelta
\DeclareMathSymbol{\myepsilon}{\mathord}{cmletters}{"0F}   \let\epsilon\myepsilon
\DeclareMathSymbol{\myzeta}{\mathord}{cmletters}{"10}      \let\zeta\myzeta
\DeclareMathSymbol{\myeta}{\mathord}{cmletters}{"11}       \let\eta\myeta
\DeclareMathSymbol{\mytheta}{\mathord}{cmletters}{"12}     \let\theta\mytheta
\DeclareMathSymbol{\myiota}{\mathord}{cmletters}{"13}      \let\iota\myiota
\DeclareMathSymbol{\mykappa}{\mathord}{cmletters}{"14}     \let\kappa\mykappa
\DeclareMathSymbol{\mylambda}{\mathord}{cmletters}{"15}    \let\lambda\mylambda
\DeclareMathSymbol{\mymu}{\mathord}{cmletters}{"16}        \let\mu\mymu
\DeclareMathSymbol{\mynu}{\mathord}{cmletters}{"17}        \let\nu\mynu
\DeclareMathSymbol{\myxi}{\mathord}{cmletters}{"18}        \let\xi\myxi
\DeclareMathSymbol{\mypi}{\mathord}{cmletters}{"19}        \let\pi\mypi
\DeclareMathSymbol{\myrho}{\mathord}{cmletters}{"1A}       \let\rho\myrho
\DeclareMathSymbol{\mysigma}{\mathord}{cmletters}{"1B}     \let\sigma\mysigma
\DeclareMathSymbol{\mytau}{\mathord}{cmletters}{"1C}       \let\tau\mytau
\DeclareMathSymbol{\myupsilon}{\mathord}{cmletters}{"1D}   \let\upsilon\myupsilon
\DeclareMathSymbol{\myphi}{\mathord}{cmletters}{"1E}       \let\phi\myphi
\DeclareMathSymbol{\mychi}{\mathord}{cmletters}{"1F}       \let\chi\mychi
\DeclareMathSymbol{\mypsi}{\mathord}{cmletters}{"20}       \let\psi\mypsi
\DeclareMathSymbol{\myomega}{\mathord}{cmletters}{"21}     \let\omega\myomega
\DeclareMathSymbol{\myvarepsilon}{\mathord}{cmletters}{"22}\let\varepsilon\myvarepsilon
\DeclareMathSymbol{\myvartheta}{\mathord}{cmletters}{"23}  \let\vartheta\myvartheta
\DeclareMathSymbol{\myvarpi}{\mathord}{cmletters}{"24}     \let\varpi\myvarpi
\DeclareMathSymbol{\myvarrho}{\mathord}{cmletters}{"25}    \let\varrho\myvarrho
\DeclareMathSymbol{\myvarsigma}{\mathord}{cmletters}{"26}  \let\varsigma\myvarsigma
\DeclareMathSymbol{\myvarphi}{\mathord}{cmletters}{"27}    \let\varphi\myvarphi

\DeclareMathOperator{\Spec}{Spec}
\DeclareMathOperator{\Proj}{Proj}

\DeclareMathOperator{\eq}{eq}
\DeclareMathOperator{\coeq}{coeq}
\DeclareMathOperator{\Bands}{{Bands}}
\DeclareMathOperator{\BandSpaces}{{BandSpaces}}
\DeclareMathOperator{\FBands}{{FBands}}
\DeclareMathOperator{\Idylls}{Idylls}
\DeclareMathOperator{\Tracts}{Tracts}
\DeclareMathOperator{\Rings}{{Rings}}

\DeclareMathOperator{\Fields}{Fields}
\DeclareMathOperator{\Pastures}{Pastures}
\DeclareMathOperator{\PartFields}{PartFields}
\DeclareMathOperator{\Null}{{Null}}
\DeclareMathOperator{\OBlpr}{OBlpr}
\DeclareMathOperator{\Mon}{{Mon}}
\DeclareMathOperator{\HypRings}{HypRings}
\DeclareMathOperator{\HypFields}{HypFields}
\DeclareMathOperator{\FuzzRings}{FuzzRings}
\DeclareMathOperator{\Alg}{Alg}
\DeclareMathOperator{\Aff}{Aff}
\DeclareMathOperator{\BAff}{BAff}
\DeclareMathOperator{\Sch}{Sch}
\DeclareMathOperator{\BSch}{BSch}
\DeclareMathOperator{\MSch}{MSch}
\DeclareMathOperator{\OBSch}{OBSch}
\DeclareMathOperator{\Top}{Top}
\DeclareMathOperator{\Sh}{Sh}
\DeclareMathOperator{\Hom}{Hom}
\DeclareMathOperator{\Gr}{Gr}
\DeclareMathOperator{\Fl}{Fl}
\DeclareMathOperator{\SL}{SL}
\DeclareMathOperator{\MacPh}{MacPh}
\DeclareMathOperator{\colim}{colim}

\DeclareMathOperator{\nullker}{nullker}
\DeclareMathOperator{\im}{im}
\DeclareMathOperator{\sign}{sign}
\DeclareMathOperator{\hypersum}{\,\raisebox{-2.2pt}{\larger[2]{$\boxplus$}}\,}

\newcommand\A{{\mathbb A}}

\newcommand\C{{\mathbb C}}
\newcommand\D{{\mathbb D}}

\newcommand\F{{\mathbb F}}
\newcommand\G{{\mathbb G}}

\newcommand\K{{\mathbb K}}

\newcommand\N{{\mathbb N}}

\renewcommand\P{{\mathbb P}}
\newcommand\Q{{\mathbb Q}}
\newcommand\R{{\mathbb R}}
\renewcommand\S{{\mathbb S}}
\newcommand\T{{\mathbb T}}
\newcommand\U{{\mathbb U}}

\newcommand\Z{{\mathbb Z}}

\newcommand\cB{{\mathcal B}}
\newcommand\cC{{\mathcal C}}

\newcommand\cF{{\mathcal F}}
\newcommand\cG{{\mathcal G}}

\newcommand\cI{{\mathcal I}}

\newcommand\cO{{\mathcal O}}

\newcommand\cS{{\mathcal S}}
\newcommand\cT{{\mathcal T}}
\newcommand\cU{{\mathcal U}}
\newcommand\cV{{\mathcal V}}

\newcommand\cX{{\mathcal X}}

\newcommand\br{{\mathbf r}}

\newcommand\fm{{\mathfrak m}}
\newcommand\fn{{\mathfrak n}}
\newcommand\fp{{\mathfrak p}}
\newcommand\fq{{\mathfrak q}}

\renewcommand\geq{\geqslant}
\renewcommand\leq{\leqslant}
\newcommand{\0}{\mathbf{0}}
\newcommand{\id}{\textup{id}}
\newcommand{\pr}{\textup{pr}}
\newcommand{\res}{\textup{res}}
\newcommand{\rev}{\textup{rev}}
\newcommand{\fuse}{\textup{fuse}}
\newcommand{\oblpr}{\textup{oblpr}}
\newcommand{\band}{\textup{band}}
\newcommand{\fine}{\textup{fine}}
\newcommand{\weak}{\textup{weak}}
\newcommand{\str}{\textup{str}}
\newcommand{\an}{\textup{an}}
\newcommand{\trop}{\textup{trop}}

\newcommand{\nul}{\textup{null}}
\newcommand{\kernel}{\textup{ker}}
\newcommand{\Tits}{\textup{Tits}}
\newcommand{\Fun}{\F_1}
\newcommand{\Funpm}{{\F_1^\pm}}
\newcommand{\utop}{\underline{(-)}}
\newcommand{\gen}[1]{\langle #1 \rangle}
\newcommand{\genn}[1]{\langle\!\!\langle #1 \rangle\!\!\rangle}
\newcommand{\bandquot}[2]{#1\!\sslash\!#2}
\newcommand{\bandgenquot}[2]{#1\!\sslash\!\gen{#2}}
\newcommand{\bandgennquot}[2]{#1\!\sslash\!\genn{#2}}

\newcommand{\pastgennquot}[2]{#1\!\sslash\!\genn{#2}}

\newcommand{\onto}{\tikz{\draw[white] (-0.1,-0.1) -- (0.5,-0.1); \draw[->>] (0,0) -- (0.4,0);}}
\newcommand{\tinymatrix}[6]{\Big[\begin{smallmatrix} #1 & #4 \\ #2 & #5 \\ #3 & #6 \end{smallmatrix}\Big]}
\newcommand{\tinyvector}[3]{{}\Big[\begin{smallmatrix} #1 \\ #2 \\ #3 \end{smallmatrix}\Big]}

\title{New building blocks for $\Fun$-geometry: bands and band schemes}

\author{Matthew Baker}
\address{\rm Matthew Baker, School of Mathematics, Georgia Institute of Technology, Atlanta, USA}
\email{mbaker@math.gatech.edu}

\author{Tong Jin}
\address{\rm Tong Jin, School of Mathematics, Georgia Institute of Technology, Atlanta, USA}
\email{tongjin@gatech.edu}

\author{Oliver Lorscheid}
\address{\rm Oliver Lorscheid, University of Groningen, the Netherlands}
\email{o.lorscheid@rug.nl}

\hypersetup{pdftitle={Bands and band schemes},pdfauthor={Matthew Baker, Tong Jin and Oliver Lorscheid}}


\begin{document}

\begin{abstract}
We develop and study a generalization of commutative rings called \emph{bands}, along with the corresponding geometric theory of \emph{band schemes}.
Bands generalize both hyperrings, in the sense of Krasner, and partial fields in the sense of Semple and Whittle. They form a ring-like counterpart to the field-like category of \emph{idylls} introduced by the first and third author in previous work.

The first part of the paper is dedicated to establishing fundamental properties of bands analogous to basic facts in commutative algebra.
In particular, we introduce various kinds of ideals in a band and explore their properties, and we study localization, quotients, limits, and colimits.

The second part of the paper studies band schemes. After giving the definition, we present some examples of band schemes, along with basic properties of band schemes and morphisms thereof, and we describe functors into some other scheme theories.

In the third part, we discuss some ``visualizations'' of band schemes, which are different topological spaces that one can functorially associate to a band scheme $X$.
\end{abstract}

\maketitle

\begin{small} \tableofcontents \end{small}


\section*{Introduction}

\subsection*{Motivation and overview of the paper}
In \cite{Baker-Bowler19}, the first author and N. Bowler developed a new theory of matroids with coefficients, clarifying and extending earlier work of Dress and Wenzel \cite{Dress-Wenzel91,Dress-Wenzel92}. In this context, they introduced algebraic objects called {\em tracts}, which generalize not only fields but also partial fields and hyperfields.
In \cite{Baker-Lorscheid21b}, the authors introduced a slightly more restrictive class of algebraic objects called {\em idylls}, which are still general enough to contain both partial fields and hyperfields. They also realized the category of idylls as a full subcategory of the category of {\em ordered blueprints}, which was introduced by the third author in \cite{Lorscheid22} and \cite{Lorscheid23}. 
The category of ordered blueprints also contains the category of rings\footnote{All rings in this paper will be commutative rings with identity.} (and more generally hyperrings in the sense of Krasner).

\medskip

Since the publication of \cite{Baker-Lorscheid21b}, we have realized that it is unnecessarily cumbersome to use ordered blueprints in the context of matroids with coefficients, as there is a significantly simpler ring-like category -- the category of bands -- which contains idylls as a full subcategory. 
Loosely speaking, bands are to idylls as rings are to fields, and the definition of a band is closely modeled on the definition of an idyll. 

\medskip

The first part of the present paper is dedicated to developing some of the basic commutative algebra properties of bands. 
In particular, we introduce three different kinds of ideals in a band, which we call $m$-ideals, $k$-ideals, and null ideals, and we explore their properties. In each case, there are corresponding notions of maximal and prime ideals which we also explore. 
We also discuss localization by multiplicative sets, as well as free algebras, quotients, limits (e.g. products), and colimits (e.g. tensor products).
Although some of the results we prove are special cases of more general results about ordered blueprints, we believe it is pedagogically useful to give independent proofs using the simpler language of bands.

\medskip

In \cite{Lorscheid22,Lorscheid23,Lorscheid18}, the third author developed a geometric theory of ordered blue schemes, which are constructed by gluing together --- as local building blocks --- the spectra of ordered blueprints (suitably defined). 
In \cite{Baker-Lorscheid21b}, the authors defined a family of ordered blue schemes ${\rm Mat}(r,n)$ such that ${\rm Mat}(r,n)(K) = \G(r,n)(K)$ for all fields $K$, where $\G(r,n)$ is the Grassmannian parametrizing $r$-dimensional subspaces of $K^n$. Moreover, if $\K$ is the {\em Krasner hyperfield} --- the final object in the category of idylls --- ${\rm Mat}(r,n)(\K)$ is the set of {\em matroids} of rank $r$ on $\{ 1,\ldots,n \}$.
More generally, for any idyll $F$, ${\rm Mat}(r,n)(F)$ is the set of strong $F$-matroids of rank $r$ on $\{ 1,\ldots,n \}$ in the sense of Baker--Bowler~\cite{Baker-Bowler19}.

\medskip

Again, it is not necessary for this particular application to develop the more general theory of ordered blue schemes; it suffices to develop the somewhat simpler theory of {\em band schemes}. This is what we do in the second part of the present paper. Along the way, we clarify some of the subtleties which show up both in the theory of band schemes and the more general theory of ordered blue schemes; in particular, we explain how certain natural desiderata force us to use a particular definition of $\Spec(B)$ for a band $B$ which might at first seem counterintuitive. More precisely, we explain why one needs to define the spectrum --- as a topological space --- using prime {\em $m$-ideals} rather than prime $k$-ideals, prime null ideals, or some other notions of points. 

\medskip

We present some examples of band schemes, along with basic properties of band schemes and morphisms thereof, and we describe functors into some other scheme theories.
For example, given a band scheme $X$, the functor taking a band $B$ to the set $X(B) := {\rm Hom}(\Spec B,X)$ of $B$-points of $X$ can be restricted to rings, and the resulting functor turns out to be representable by a scheme $X^{\rm sch}$.
We also show that fibre products exist in the category of band schemes. In particular, given a band $B$, a band morphism $B' \to B$, and a band scheme $X$ over $B$, one can define the {\em base extension} $X_{B'}$.

\medskip

In the third and final section, we discuss various ``visualizations'' of band schemes. The basic idea is that there are various different topological spaces which one can associate to a band scheme $X$, such as its set $X(\K)$ of points over the Krasner hyperfield, which we call the {\em kernel space} of $X$ and denote by $X^{\rm ker}$, or the set of closed points of $X(\K)$, which we call the {\em Tits space} of $X$ and denote by $X^{\Tits}$. 
When $X=\Spec(B)$ is affine, $X^{\rm ker}$ is homeomorphic to the space of prime $k$-ideals of $B$ endowed with a suitable topology.
We also define the {\em null space} $X^{\rm null}$ of $X$, which when $X=\Spec(B)$ is affine corresponds to the space of prime null ideals, again endowed with a suitable topology. 
Each of these spaces has some uses; for example, while the map of topological spaces underlying a closed immersion between band schemes need not be closed, the induced map of null spaces is always closed.
We also discuss what we call the weak and strong Zariski topologies on a band scheme $X$.
In addition, we discuss the so-called ``fine topology'' on the point-set $X(B)$ when $X$ is a band scheme and $B$ is a topological band.

\subsubsection*{Additional motivation}
Matroids are far from the only motivation for introducing the theory of band schemes. Here are some additional sources of motivation, which --- taken together --- hopefully justify the careful development of band schemes from first principles which we undertake in the present paper:

\begin{enumerate}
\item {\bf Toric varieties.} Every toric variety $X$ defined over a field $k$ admits a band scheme model $\cX$, which is defined over the initial object $\Funpm$ in the category of bands. (By a \emph{model}, we mean that $\cX$ is a band scheme over $\Funpm$ whose base extension to $k$ is isomorphic to $X$). For more details, see \autoref{ex: toric band schemes}.
\item {\bf Models for algebraic groups and geometry over $\Fun$.} If $G$ is a reductive algebraic group over a field $k$, together with a ``sufficiently nice'' representation of $G$, there is an associated band scheme model $\cG$ defined over $\Funpm$ with the property that the Weyl group of $G$ is isomorphic to $\cG^{\Tits}$, the set of closed points of $\cG(\K)$. This provides a satisfying conceptual framework for studying ``Tits' dream'' in $\Fun$-geometry. For more details, see \cite{Lorscheid-Thas23} or 
\autoref{subsection: the Tits space} below.
\item {\bf Berkovich analytification.} An $\R$-valuation on a field $k$ is the same thing as a homomorphism from $k$ to $\T$, where $\T$ denotes the tropical hyperfield (identified with its associated topological band).
If $X=\Spec R$ is an affine $k$-scheme of finite type, considered as a band scheme over $k$, the set $X(\T)=\Hom_k(R,\T)$ with the fine topology is canonically homeomorphic to the Berkovich analytification of $X$. For more details, see \cite[Thm.~3.5]{Lorscheid22}.
\item {\bf Tropicalization.} Continuing with the notation from (3), choosing generators $a_1,\dotsc,a_n$ for $R$ as a $k$-algebra yields a presentation $R=k[a_1,\dotsc,a_n]/I$ for some ideal $I$, and from this presentation one can naturally define a band scheme $\cX$ over $k$ with the property that $\cX(\T)$ is canonically homeomorphic to the tropicalization $X^\trop$ of $X$ with respect to the embedding into $\A_k^n$ given by $a_1,\dotsc,a_n$. For more details, see \autoref{ex: tropicalization} below and \cite[Thm.~3.5]{Lorscheid22}.
\end{enumerate}

\subsection*{More detailed overview of the paper}
Here is a more detailed overview of the contents of the present paper.

\subsubsection*{Definition of a band}

A \emph{pointed monoid} $B$ is a set $B$ together with an associative and commutative multiplication $\cdot: B \times B \to B$, and two elements $0, 1 \in B$ such that $0 \cdot a = 0$ and $1 \cdot a = a$ for all $a \in B$.

\medskip

Let $B$ be a pointed monoid. Identifying $0 \in B$ with the additive identity element in the semiring $\N[B] = \ \big\{ \sum a_i \, \big| \, a_i\in B \big\}$ defines a semiring
\[\textstyle
 B^+ \ = \ \N[B]/\gen{0\sim 0_{\N[B]}} \ = \ \big\{ \sum a_i \, \big| \, a_i\in B-\{0\}\big\}.
\]

\medskip

 A \emph{band} is a pointed monoid $B$ together with an ideal $N_B \subset B^+$, called the \emph{null set of $B$}, such that for every $a \in B$, there exists a unique element $b \in B$ (denoted $-a$) such that $a+b\in N_B$.  

\subsubsection*{The categories of bands and idylls}

A \emph{band morphism} is a multiplicative map $f:B\to C$ with $f(0)=0$ and $f(1)=1$ such that $\sum a_i \in N_B$ implies $\sum f(a_i) \in N_C$. This defines the category $\Bands$.

\medskip

A band in which $0 \ne 1$ and every nonzero element has a multiplicative inverse is called an \emph{idyll}. We define $\Idylls$ as the full subcategory of $\Bands$ formed by idylls.

\medskip

The \emph{regular partial field} $\F_1^{\pm} = \{0, 1, -1\}$ is the idyll with the obvious multiplication and null set
\[
 N_{\F_1^{\pm}} =  \big\{0, \; 1 - 1, \; 1 - 1 + 1 -1, \; \dots\}.
\]
It is the initial object in both $\Bands$ and $\Idylls$.

\medskip

The \emph{Krasner hyperfield} is the pointed monoid $\K=\{0,1\}$ with the obvious multiplication, together with the null set $N_\K=\{0,\; 1+1,\; 1+1+1,\; \dotsc\}$. It is the terminal object in the category of idylls.

\medskip

The terminal object in $\Bands$ is the trivial band $\{0\}$ with $0=1$.

\subsubsection*{Examples}

Here are a few important examples of bands:

\begin{enumerate}
\item {\bf Rings.} Every ring $R$ is naturally a band, with the same underlying pointed monoid and with null set 
 \[ \textstyle
  N_R \ = \ \big\{\sum a_i \, \big| \, \sum a_i=0\text{ as elements of }R \big\}. 
 \]
This defines a fully faithful embedding of the category of rings into the category of bands.
\item {\bf Partial Fields.} Every partial field $P$ is naturally an idyll, with the same underlying pointed monoid and with null set $N_P=\{\sum a_i\mid \sum a_i=0\text{ in }P\}$. 
This defines a fully faithful embedding of the category of partial fields into the category of idylls.
\item {\bf Hyperrings.} A commutative hyperring $R$ in the sense of Krasner is a band with the same pointed monoid $R$ and with null set
\[\textstyle
 N_R \ = \ \big\{ \sum a_i \, \big| \, 0\in\hypersum a_i \big\}.
\]
This defines a fully faithful embedding of the category of hyperrings into the category of bands.
As a particularly important example of a hyperring, we mention the \emph{tropical hyperfield} $\T$, whose underlying pointed monoid is $\R_{\geq0}$ and whose null set is
\[\textstyle
 N_\T = \big\{ 0 \big\} \bigcup \big\{\sum a_i \, \big| \, \text{the maximum of }\{a_i\}\text{ appears twice}\big\}.
\]
\end{enumerate}

\subsubsection*{Quotients}

A \emph{quotient} of a band $B$ is an isomorphism class of surjective morphisms $\pi:B\to C$. 

\medskip

A \emph{null ideal} of $B$ is an ideal $I$ of $B^+$ that contains $N_B$ such that
if $a-c\in I$ and $c+\sum b_j\in I$, then $a+\sum b_j\in I$.

\medskip

Let $f:B\to C$ be a band morphism, the \emph{null kernel of $f$} is
 \[\textstyle
  \nullker f \ = \ \big\{ \sum a_i\in B^+ \, \big| \, \sum f(a_i)\in N_C \big\}.
 \]
One checks easily that the null kernel of a morphism is always a null ideal.

\medskip

Given a null ideal $I$ of $B$, we define an equivalence relation $\sim$ on $B$ by the rule $a\sim b$ if and only if $a-b\in I$.
Then the pointed monoid $\bandquot BI=B/\sim$ together with the null set
        \[\textstyle
         N_{\bandquot BI} \ = \ \big\{ \sum [a_i] \, \big| \, \sum a_i \in I \big\}
        \]
is a band, and the quotient map $\pi_I:B\to \bandquot BI$ is a band morphism with null kernel $I$.
 
\medskip
We will see in \autoref{cor: bijection between null ideals and quotients} that the association $I\mapsto \bandquot BI$ establishes a bijection
 \[
  \Phi: \ \big\{\text{null ideals of $B$} \big\} \ \longrightarrow \ \{ \text{quotients of $B$}\big\}.
 \]

\subsubsection*{Free algebras and presentations}

Let $k$ be a band. A \emph{$k$-algebra} is a band $B$ together with a morphism from $k$ to $B$.
If $B$ is a band and $\{x_i\mid i\in I\}$ is a set of indeterminates, there is a free $B$-algebra $B[x_i] := B[x_i\mid i\in I]$ with the universal property (\autoref{prop: universal property of the free algebra}) that to give a homomorphism from $B[x_i]$ to a $B$-algebra $C$ is the same thing as specifying a target value $c_i \in C$ for each $x_i$.

\medskip

For a subset $S$ of $B^+$, we denote by $\bandgenquot BS$ the quotient of $B$ by the null ideal $\gen S$ generated by $S$.

\medskip

Every band is a quotient of a free algebra over $\Funpm$ and can thus be written in the form $\bandgenquot{\Funpm[x_i]}{S}$. 
For example, we have $\F_2 = \bandgenquot{\Funpm}{1+1}$ and $\K = \bandgenquot{\Funpm}{1+1,\ 1+1+1}$.

\subsubsection*{Localizations}

If $B$ is a band and $S$ is a multiplicative subset of $B$ (i.e., $1 \in S$ and $S$ is closed under multiplication), we can define the \emph{localization} of $B$ at $S$, denoted $S^{-1}B$, as follows. 
As a pointed monoid, $S^{-1}B$ is $(S\times B)/\sim$, where $(s,a)\sim (s',a')$ if and only if there is a $t\in S$ such that $tsa'=ts'a$. 
Writing $\frac as$ for the equivalence class of $(s,a)$ in $S^{-1}B$, the null set of $S^{-1}B$ is defined as
  \[\textstyle
  N_{S^{-1}B} \ = \ \gen{ \sum \frac{a_i}1 \mid \sum a_i\in N_B }_{(S^{-1}B)^+}.
 \]
 
There is a natural map $\iota_S:B\to S^{-1}B$ defined by $\iota_S(a)=\frac a1$, and this construction has the expected universal property (\autoref{prop: universal property of the localization}): a band morphism from $S^{-1}B$ to $C$ is the same thing as a band morphism $f: B\to C$ with $f(S)\subset C^\times$. 

\subsubsection*{m-Ideals and k-ideals}

An \emph{$m$-ideal} of a band $B$ is a subset $I \subseteq B$ such that $0 \in I$ and $B \cdot I = I$. 

\medskip

A \emph{$k$-ideal} is an $m$-ideal $I$ with the additional property that if $a + \sum b_i \in N_B$ and $b_i \in I$, then $a \in I$. 

\medskip

A \emph{prime $m$-ideal} (resp. \emph{prime $k$-ideal}) is an $m$-ideal (resp. $k$-ideal) $\fp$ of $B$ for which $S=B-\fp$ is a multiplicative set. 
Given a prime $m$-ideal $\fp$ with complement $S$ in $B$, the \emph{localization of $B$ at $\fp$} is $B_\fp=S^{-1}B$. 

\medskip

As in commutative algebra, there is a bijection between prime $m$-ideals (resp. $k$-ideals) of $S^{-1}B$  and prime $m$-ideals (resp. $k$-ideals) of $B$ which are disjoint from $S$.

\medskip

Every band $B$ has a unique maximal proper $m$-ideal, namely $\fm=B-B^\times$, which is a prime $m$-ideal since $B^\times$ is a multiplicative set. The $m$-ideal $\fm$ is in general not a $k$-ideal.
Every maximal $k$-ideal of $B$ is prime (\autoref{proposition: maximal k-ideals are prime}).

\subsubsection*{Radicals}

If $I$ is an $m$-ideal of a band $B$, the \emph{radical of $I$} is 
 \[
  \sqrt{I} \ = \ \{a \in B \mid a^n \in I\text{ for some $n \geq1$}\}. 
 \]
As in commutative algebra, we have 
\[
 \sqrt{I} \ = \ \bigcap_{\substack{\text{prime $m$-ideals $\fp$}\\ \text{that contain $I$}}} \fp,
\]
and if $I$ is a $k$-ideal, then 
\[
 \sqrt{I} \ = \ \bigcap_{\substack{\text{prime $k$-ideals $\fp$}\\ \text{that contain $I$}}} \fp.
\]
In particular, every prime $m$-ideal (resp. $k$-ideal) is \emph{radical}, i.e., $\sqrt{I}=I$; see \autoref{prop: radical ideal}.

\subsubsection*{Limits and colimits}
The category of bands is complete and cocomplete.
In particular, it admits the following constructions:

\begin{enumerate}
\item {\bf Products.} The \emph{product} of a family of bands $\{B_i\}$ is defined as the Cartesian product $\prod B_i$ of the underlying monoids, together with the null set
\[\textstyle
 N_{\prod B_i} \ = \ \{ \sum_j (a_{ji})_{i\in I} \in (\prod B_i)^+ \mid \sum a_{ji}\in N_{B_i}\text{ for every }i\in I\}.
\]
This construction satisfies the universal property of products (\autoref{prop: universal property of the product}).
\item {\bf Equalizers.}
Given band morphisms $f:B\to C$ and $g:B\to C$, their \emph{equalizer} is the pointed submonoid $\eq(f,g)=\{a\in B\mid f(a)=g(a)\}$ of $B$, together with the null set
\[\textstyle
 N_{\eq(f,g)} \ = \ \{ \sum a_i\in\eq(f,g)^+\mid \sum a_i\in N_B\}.
\]
This construction satisfies the universal property of equalizers (\autoref{prop: universal property of the equalizer}).
\item{\bf Tensor products.}
Given a band $k$ and a non-empty family of $k$-algebras $\{B_i\}_{i\in I}$ with structure maps $\alpha_{B_i}:k\to B_i$, the \emph{tensor product of $\{B_i\}_{i \in I}$ over $k$} is the pointed monoid
\[\textstyle
 \bigotimes_k B_i \ = \ \{ a\in\prod B_i\mid a_i=1\text{ for all but finitely many }i\in I\} \ / \ \sim \; ,
\]
where $\sim$ is a suitable equivalence relation. 
This construction satisfies the universal property of tensor products (\autoref{prop: universal property of the tensor product}).
\end{enumerate}

\subsubsection*{Definition of a band scheme}

Let $B$ be a band. Its \emph{prime spectrum $\Spec B$} is the set of all prime $m$-ideals $\fp$ of $B$, together with the topology generated by the \emph{principal open subsets} 
\[
 U_h \ = \ \{ \fp\in\Spec B \mid h\notin\fp\}
\]
for $h\in B$.

\medskip

We endow $\Spec B$ with the structure sheaf $\cO_X$ on $X=\Spec B$ in $\Bands$ characterized by $\cO_X(U_h)=B[h^{-1}]$ for $h\in B$. 

\begin{ex*}
 The affine $n$-space over an idyll $F$ is $\A^n_F=\Spec F[T_1,\dotsc,T_n]$. The prime $m$-ideals of $F[T_1,\dotsc,T_n]$ are of the form $\fp_I=\gen{T_i\mid i\in I}_m$ for subsets $I$ of $\{1,\dotsc,n\}$. 
\end{ex*}

A \emph{band space} is a topological space $X$ together with a sheaf $\cO_X$ in $\Bands$. The \emph{stalk at $x\in X$} is the band
 \[
  \cO_{X,x} \ = \ \underset{x\in U\subset X\text{ open}}{\colim} \cO_X(U).
 \]
 
\medskip

A \emph{morphism of band spaces} is a continuous map $\varphi:X\to Y$ between band spaces together with a sheaf morphism $\varphi^\#:\cO_Y\to\varphi_\ast\cO_X$  such that for every $x\in X$ and $y=\varphi(x)$, the induced morphism of stalks $\varphi_x^\#:\cO_{Y,y}\to \cO_{X,x}$ sends non-units to non-units. This defines the category $\BandSpaces$ of band spaces.

\medskip

An \emph{affine band scheme} is a band space that is isomorphic to the spectrum of a band. A \emph{band scheme} is a band space in which every point has an affine open neighborhood. 
A \emph{morphism of band schemes} is a morphism of the associated band spaces. This defines the category $\BSch$ of band schemes.

Analogous to usual scheme theory, a band morphism $f:B\to C$ induces a morphism $f^\ast:\Spec C\to\Spec B$ of band schemes by taking inverse images of prime $m$-ideals of $C$. This enhances the construction of the spectrum to a functor $\Spec:\Bands\to\BSch$. 

\medskip

The fact that every band $B$ has a unique maximal $m$-ideal leads to various simplifications, when compared with usual scheme theory.
For example, every affine open subset $U$ of a band scheme $X$ is a principal open subset, and the image $\varphi(U)$ of an affine open $U$ of $X$ under a morphism $\varphi:X\to Y$ is contained in an affine open $V$ of $Y$. 

\subsubsection*{Motivation for the definition of a band scheme}

Since we have already seen various kinds of ideals in the theory of bands ($m$-ideals, $k$-ideals, and null ideals), the reader may wonder why we have used prime $m$-ideals to define the spectrum of a band.

To help motivate this choice, we show that $\Spec$ is the unique contravariant functor from $\Bands$ to $\BandSpaces$ with the following properties (cf.\ \autoref{subsection: covering families} for more details):

\begin{enumerate}
\item For all bands $B$ and all $h \in B$, the induced map of topological spaces $U_h := \Spec B[h^{-1}] \to \Spec B$ is an open embedding, and the collection $\{ U_h \}_{h \in B}$ forms a base for the topology of $\Spec B$.
 \item The functor $\Spec$ is fully faithful, i.e., for every pair of bands $B,B'$, the induced map $\Hom_{\Bands}(B,B') \to \Hom_{\BandSpaces}(\Spec B', \Spec B)$ is a bijection.
 \item The global section functor is left adjoint to $\Spec$, i.e., if $\Gamma X=\cO_X(X)$ denotes the band of global sections of $\cO_X$, then for all bands $B$ and band spaces $X$ there is a functorial bijection
\[
\Phi: \ \Hom_{\BandSpaces}(X,\ \Spec B) \ \longrightarrow \ \Hom_{\Bands}(B,\ \Gamma X).
\] 
 \item For every band $B$, the topological space $\Spec B$ is sober, and thus determined up to homeomorphism by its lattice of open subsets.
\end{enumerate}

\subsubsection*{Properties of band schemes}

In analogy with usual scheme theory, there is an inclusion-reversing bijective correspondence between radical $m$-ideals of $B$ and closed subsets of $X$. This restricts to a bijection between prime $m$-ideals and irreducible closed subsets. See \autoref{thm: Nullstellensatz}.

\medskip

Let $X$ be a band scheme and let $x$ be a point of $X$. Let $\cO_{X,x}$ be the stalk at $x$ and $\fm_x=\cO_{X,x}-\cO_{X,x}^\times$ its maximal $m$-ideal. 
The \emph{residue field at $x$} is the quotient $k(x)=\bandquot{\cO_{X,x}}{\gen{\fm_x}}$. 
There is a canonical morphism $\kappa_x:\Spec k(x)\to X$ which satisfies the expected universal property: for every idyll $F$, every morphism $\Spec F\to X$ with image $\{x\}$ factors uniquely through $\kappa_x$. See \autoref{prop: universal property of the residue field}.

\medskip

Note that the ``residue field'' of a point in a band scheme is in general not a field, but rather an idyll (or the $\0$ band).

\subsubsection*{Properties of morphisms of band schemes}

If $\psi_X:X\to Z$ and $\psi_Y:Y\to Z$ are morphisms of band schemes, the \emph{fibre product of $X$ and $Y$ over $Z$} is defined as the topological fibre product $X\times_ZY$, together with the structure sheaf that sends an affine open of the form $U\times_WV$ (with $U$, $V$ and $W$ affine) to $\Gamma U\otimes_{\Gamma W}\Gamma V$, together with the obvious restriction maps. 
This construction satisfies the usual universal property of fibre products. See \autoref{thm: limits and colimits for band schemes}.

\medskip

The algebraic definitions of open immersions, closed immersions, and separated morphisms in usual scheme theory all generalize in a meaningful way to band schemes, cf.~\autoref{subsection: open and closed immersions} and \autoref{subsection: separated morphisms}.
However, the naive topological characterization of closed immersions in terms of closed maps fails for the underlying topological space of a band scheme. For example, the diagonal embedding of $\A^1$ into $\A^2$ is a closed immersion, but the corresponding map of underlying topological spaces is not closed. This ``bug'' can be removed by passing to null spaces, cf.~\autoref{subsection: the null space}.

\subsubsection*{Functors into other scheme theories}

Band schemes map to several other types of schemes in a functorial way. 
For example:

\begin{enumerate}
\item {\bf Base extension to usual schemes.}
A band $B$ comes with the universal ring $B^+_\Z  \ := \ \Z[B]/\gen{N_B}.$
 This construction globalizes to a \emph{base extension functor}
\[
 (-)^+_\Z: \ \BSch \ \longrightarrow \ \Sch
\]
from the category $\BSch$ of band schemes into the category $\Sch$ of usual schemes. See \autoref{prop: base extension to schemes}.
\item  {\bf The underlying monoid schemes.} 
The forgetful functor $\cF:\Bands\to\Mon_0$ from bands to pointed monoids globalizes to a functor $\cF:\BSch\to\MSch$ from band schemes to monoid schemes. See \autoref{prop: functor to the underlying monoid scheme}.
\end{enumerate}

\subsubsection*{Visualizations}

Besides its underlying topological space, one can associate several other topological spaces to a band scheme in a functorial way. We call these associated spaces \emph{visualizations}, because they exhibit certain useful properties of band schemes. 
Here are some examples:

\begin{enumerate}
\item {\bf The Zariski topology.} 
In classical algebraic geometry, rational point sets $X(B)$ come equipped with a Zariski topology. In the context of band schemes, there are more subsets supporting closed subschemes than complements of open subsets. Consequently, rational point sets can be equipped with two different Zariski topologies, which we call the \emph{weak} and \emph{strong} Zariski topologies. See \autoref{section: the Zariski topology} for more details.
\item {\bf The fine topology.}
If $k$ is a band and $B$ is a $k$-algebra, the choice of a topology for $B$ endows the set $X(B):=\Hom_k(\Spec B,X)$ of $B$-rational points of a $k$-scheme $X$ with a natural topology which we call the \emph{fine topology}.
See \autoref{subsection: the fine topology} for more details, as well as \cite{Lorscheid23} and \cite{Lorscheid-Salgado16}.
\item {\bf The kernel space.} 
If $X=\Spec B$ is an affine band scheme, its \emph{kernel space} is the subspace of $X$ consisting of all prime $k$-ideals of $B$. 
This construction globalizes to give a functor $X \mapsto X^\kernel$ from band schemes to topological spaces.
With respect to the natural topology on the Krasner hyperfield $\K$, there is a natural homeomorphism between $X^\kernel$ and the fine topology on $X(\K)$.
A point $x \in X$ belongs to $X^\kernel$ if and only if the residue field $k(x)$ is nonzero (in which case it is automatically an idyll).
See \autoref{subsection: the kernel space} for more details.
\item {\bf The Tits space.} 
If $X=\Spec B$ is an affine band scheme, its \emph{Tits space} is the subspace of $X$ consisting of all maximal $k$-ideals of $B$. 
This construction globalizes to give a functor $X \mapsto X^\Tits$ from band schemes to topological spaces.
See \autoref{subsection: the Tits space} for more details.
\item {\bf The null space.}
A null ideal $I$ of a band $B$ is \emph{prime} if $S=B^+-I$ is a multiplicative set in $B^+$. The \emph{null space of $B$} is the set $\Null(B)$ of prime null ideals of $B$, together with the topology generated by \emph{basic opens} of the form $U_h \ = \ \big\{\fp\in\Null(B) \, \big| \, h\notin\fp \big\}$ for $h\in B^+$. 
This construction globalizes to give a functor $X \mapsto X^\nul$ from band schemes to topological spaces.
There is a canonical morphism $X^\nul\to X$ whose image is $X^\kernel$.
If $f:X\to Y$ is a closed immersion of band schemes, the corresponding morphism $f^\nul : X^\nul \to Y^\nul$ is a closed embedding of topological spaces (recall that this fails with the null space replaced by the usual underlying topological space of $X$).
For more details, see \autoref{subsection: the null space}.
\end{enumerate}

For a comparison of these different visualizations, see \autoref{subsection: comparison of visualizations}.

\subsection*{Acknowledgements}
We thank Manoel Jarra for useful feedback on a previous draft. We also thank the anonymous referees for their careful corrections to the paper. The first author was supported by NSF grant DMS-2154224 and a Simons Fellowship in Mathematics. The second author was supported by NSF grant DMS-2154224. 


\section{Bands}

\subsection{Definitions}

A \emph{pointed monoid} $B$ is a set $B$ together with an associative and commutative multiplication $\cdot: B \times B \to B$, and two elements $0, 1 \in B$ such that $0 \cdot a = 0$ and $1 \cdot a = a$ for all $a \in B$. We also write $ab$ for $a \cdot b$. 
A \emph{morphism of pointed monoids} is a monoid homomorphism which sends $0$ to $0$. This makes pointed monoids into a category.

Let $B$ be a pointed monoid. 
The \emph{free semiring over $B$}, denoted $\N[B]$, is characterized by the universal property that any morphism of monoids from $B$ to the underlying monoid of a semiring $S$ extends uniquely to a semiring homomorphism from $B^+$ to $S$. We think of elements of $\N[B]$ as formal sums of the form $\sum a_i$ with $a_i\in B$, with the empty sum serving as the additive identity element.

Identifying $0 \in B$ with the additive identity element in $\N[B]$\footnote{Formally, ``identifying $0 \in B$ with the additive identity element in $\N[B]$'' means we define $B^+$ as the quotient of $\N[B]$ by the equivalence relation generated by $0 + \sum a_i \sim \sum a_i$ for all $\sum a_i \in \N[B]$, together with its natural addition and multiplication operations.} defines a semiring $B^+$, which is characterized by the universal property that any morphism of \emph{pointed} monoids from $B$ to the underlying \emph{pointed} monoid of a semiring $S$ extends uniquely to a semiring homomorphism from $B^+$ to $S$.
We call $B^+$ the \emph{ambient semiring of $B$}, and write elements of $B^+$ as formal sums of the form $\sum a_i$ with $a_i\in B-\{0\}$.
Note that $B$ embeds as a submonoid of $B^+$. 

An \emph{ideal} of $B^+$ is a subset $I$ such that \ $0 \in I$, \ \ $I + I = I$ \ and \ $B \cdot I = I$. 

\begin{df}
 A \emph{band} is a pointed monoid $B$ together with an ideal $N_B \subset B^+$, called the \emph{null set of $B$}, such that for every $a \in B$, there exists a unique element $b \in B$ such that $a+b\in N_B$.  
\end{df}

We call the unique element $b$ with $a+b\in N_B$ the \emph{additive inverse of $a$} and denote it by $-a$. This means, in particular, that $B$ comes with a distinguished element $-1$, the additive inverse of $1$. We write $a-b$ for $a+(-b)$. 

\begin{lemma}\label{lemma: first properties of bands}
 Let $B$ be a band. Then
 \begin{enumerate}
  \item $B \cap N_B = \{0\}$. 
  \item $(-1)^2 = 1$ and $-a=(-1)\cdot a$ for all $a\in B$.
 \end{enumerate}
\end{lemma}

\begin{proof}
 By definition, $0\in B\cap N_B$. If $a \in B \cap N_B$, then since $0 - 0 \in N_B$ and $0 - a \in N_B$, we know $0 = a$, i.e.,\ $B\cap N_B=\{0\}$, as claimed.
 
 Since $1+(-1)\in N_B$ and $N_B$ is stable under multiplication by $B$, we have $a+(-1)\cdot a\in N_B$. Thus $-a=(-1)\cdot a$. For $a=-1$, this yields $-1+(-1)^2\in N_B$. Since $1$ is the unique additive inverse of $-1$, we conclude that $(-1)^2=1$, which completes the proof. 
\end{proof}

\begin{df}
 A \emph{band morphism} is a multiplicative map $f:B\to C$ with $f(0)=0$ and $f(1)=1$ such that $\sum a_i \in N_B$ implies $\sum f(a_i) \in N_C$. This defines the category $\Bands$ of bands.
\end{df}

Note that a band morphism $f:B\to C$ extends by linearity to a semiring homomorphism $f^+:B^+\to C^+$, via the rule $f^+(\sum a_i)=\sum f(a_i)$. We say that the null set $N_B$ of a band $B$ is \emph{generated by a subset $S$ of $B^+$} if $S\subset N_B$ and if every element of $N_B$ is a $B^+$-linear combination of elements in $S$. We write $N_B=\gen S_{B^+}$ in this case.

\begin{lemma}\label{lemma: morphisms can be tested on generators}
 Let $B$ be a band whose null set is generated by $S$, and let $f:B\to C$ be a multiplicative map between bands such that $f(1)=1$ and $\sum f(a_i)\in N_C$ for every $\sum a_i\in S$. Then $f$ is a band morphism.
\end{lemma}

\begin{proof}
 Since $N_B=\gen{S}_{B^+}$, every element of $N_B$ can be written as $\sum x_iy_i$ with $x_i \in B^+$ and $y_i \in S$. By assumption, $f^+(y_i)\in N_C$. Since $N_C$ is an ideal of $C^+$, also $f^+(\sum x_iy_i)=\sum f^+(x_i)f^+(y_i)\in N_C$, which shows that $f:B\to C$ is a morphism.
\end{proof}

\begin{df}
 Let $B$ be a band. A \emph{unit of $B$} is an element $a\in B$ for which there is a $b\in B$ (its \emph{inverse}) such that $ab=1$. The \emph{unit group of $B$} is the group $B^\times$ of units of $B$. An \emph{idyll} is a band $B$ with $B^\times=B-\{0\} \ne \emptyset$. We denote by $\Idylls$ the full subcategory of $\Bands$ that consists of all idylls. 

 A \emph{zero divisor of $B$} is an element $a\in B$ for which there is a nonzero element $b\in B$ such that $ab=0$. An element $a\in B$ is \emph{nilpotent} if $a^n=0$ for some $n\geq1$.

 A \emph{fusion band} is a band $B$ that satisfies the following \emph{fusion axiom} for all elements $c,a_1,\dotsc,a_n,b_1,\dotsc,b_m\in B$: 
 \begin{enumerate}[label=\rm(F)]    
  \item\label{F} If \ \ $-c + \sum_{i = 1}^n a_i$ \ \ and \ \ $c+\sum_{j = 1}^m b_j$ \ \ are in $N_B$, then \ \ $\sum_{i = 1}^n a_i + \sum_{j = 1}^m b_j \in N_B$. 
 \end{enumerate}
 We denote by $\FBands$ the full subcategory of $\Bands$ that consists of fusion bands. 
\end{df}

\begin{lemma}\label{lemma: multiplicative fusion}
 A fusion band $B$ satisfies the following \emph{multiplicative fusion rule} for all $c,d,a_1,\dotsc,a_n,b_1,\dotsc,b_m\in B$: 
 \begin{enumerate}[label=\rm(F*)]
  \item\label{F*} If \ \ $c-\sum_{i = 1}^n a_i$ \ \ and \ \ $d-\sum_{j = 1}^m b_j$ \ \ are in $N_B$, then \ \ $cd-\sum_{i=1}^n\sum_{j=1}^m a_ib_j \in N_B$. 
 \end{enumerate} 
\end{lemma}

\begin{proof}
 Since $N_B$ is stable under multiplication by $B$, we find that it contains also the elements $x_1=cd-\sum_{j=1}^m cb_j$ and $y_j=cb_j-\sum_{i=1}^n a_ib_j$. Applying the fusion axiom successively (for $k=1,\dotsc,m$) to $x_k$ and $y_k$ with respect to $cb_k$, which appears in both terms with opposite signs, defines a new element $x_{k+1}=cd-\sum_{j=k+1}^m cb_j-\sum_{i=1}^n \sum_{j=1}^k a_ib_j$ in $N_B$. Since $x_m=cd-\sum_{i=1}^n\sum_{j=1}^m a_ib_j$, the claim follows.
\end{proof}

\subsection{Examples}
In this section, we identify rings, partial fields, hyperrings, pastures and fuzzy rings as particular instances of bands. On the other hand, every band is an ordered blueprint, and every idyll is a tract. More precisely, all these descriptions yield fully faithful embeddings of categories, as summarized in \autoref{fig: diagram of subcategories}. 

Since the rest of this paper is independent from this section, we allow ourselves to be brief in our explanations and to omit details at times. The reader can find more details on the relation between certain categories mentioned below in \cite[Section 2]{Baker-Lorscheid21b}.
\

\subsubsection{Rings}
 Every ring $R$ is naturally a fusion band with null set 
 \[ \textstyle
  N_R \ = \ \big\{\sum a_i \, \big| \, \sum a_i=0\text{ as elements of }R \big\}. 
 \]
A map $f:R\to S$ between rings is a ring homomorphism if and only if it is a band morphism. In other words, this defines a fully faithful embedding $\Rings\to\FBands$. The band associated with a field is an idyll.

\subsubsection{Partial fields}
\label{ex:partfields}
The \emph{regular partial field} $\F_1^{\pm} = \{0, 1, -1\}$ is the idyll with the obvious multiplication and null set
\[
 N_{\F_1^{\pm}} =  \big\{0, \; 1 - 1, \; 1 - 1 + 1 -1, \; \dots\}.
\]
It is initial in $\Bands$: given a band $B$, there is a unique morphism $f:\Funpm\to B$, given by $f(0)=0$ and $f(\pm1)=\pm1$ (cf. Section~\autoref{sec:initial-terminal} below).
 
More generally, every partial field $P$ is a fusion idyll with the same underlying monoid and with null set $N_P=\{\sum a_i\mid \sum a_i=0\text{ in }P\}$. A map $f:P\to Q$ between partial fields is a partial field morphism if and only if it is an idyll morphism. This defines a fully faithful embedding $\PartFields\to\Idylls$.

\subsubsection{Hyperrings}
The \emph{Krasner hyperfield} is the band $\K=\{0,1\}$ with the obvious multiplication and null set $N_\K=\{0,\; 1+1,\; 1+1+1,\; \dotsc\}$. It is terminal in $\Idylls$: given an idyll $F$, there is a unique morphism $t_F:F\to\K$, given by $t_F(0)=0$ and $t_F(a)=1$ for $a\neq 0$.
 
More generally, a commutative hyperring $R$ in the sense of Krasner is a fusion band with the same pointed monoid $R$ and null set
\[\textstyle
 N_R \ = \ \big\{ \sum a_i \, \big| \, 0\in\hypersum a_i \big\}.
\]
A map $f:R\to S$ of hyperrings is a hyperring morphism if and only if it is band morphism. In other words, this defines a fully faithful embedding $\HypRings\to\FBands$. The band associated with a hyperfield is a fusion idyll, which yields a fully faithful embedding $\HypFields\to\Idylls$.
 
Other examples of hyperfields of interest, realized as idylls, are the \emph{sign hyperfield} $\S=\{0,1,-1\}$ with null set 
\[
 N_\S = \big\{n.1+m.(-1)\, \big| \, n=m=0\text{ or }nm\neq0 \big\} 
\]
and the \emph{tropical hyperfield} $\T=\R_{\geq0}$ with null set 
\[\textstyle
 N_\T = \big\{ 0 \big\} \bigcup \big\{\sum a_i \, \big| \, \text{the maximal element of }\{a_i\}\text{ appears twice}\big\}.
\]

\subsubsection{Pastures}
As a computational device to study matroid representations over (partial) fields and certain well-behaved hyperfields, such as $\S$ and $\T$, the first and third author introduce pastures in \cite{Baker-Lorscheid20}. To recall, a \emph{pasture} is a pointed monoid $P$ together with a subset $N_P$ of $P^3$ that is invariant under multiplication by $P$ and permutation of factors such that $P^\times=P-\{0\}$, and such that for every $a\in P$, there is a unique $b\in P$ with $a+b+0\in N_P$. Morphisms of pastures are structure preserving maps, which defines the category $\Pastures$. The category of pastures contains the categories $\PartFields$ and $\HypFields$ as full subcategories; cf.\ \cite[section 2.1.5]{Baker-Lorscheid20}.
 
A pasture $P$ defines a fusion band (which is in fact an idyll) $B=P$ with null set $N_B=\genn{a+b+c\mid (a,b,c)\in N_P}$ (cf.\ \autoref{def: fusion ideals} for the definition of $\genn{-}$). This extends to functorial inclusions $\Pastures\to\FBands$ and $\Pastures\to\Idylls$.

These observations allow us to consider pastures as bands. Indeed, we can characterize pastures as those idylls $P$ for which $N_P=\genn{a+b+c\mid (a,b,c)\in N_P}$ (i.e., for which the null set of $P$ is generated, as a fusion ideal, by the three-term relations in $P$).

\subsubsection{Hereditary fusion bands}
By analogy with the above characterization of pastures, it is natural to introduce the category of bands $B$ for which $N_B=\genn{a+b+c\mid (a,b,c)\in N_B}$. We call these \emph{hereditary fusion bands}, since the higher-order relations in the null set are ``inherited,'' via fusion, by the three-term relations.

Following \cite{BakerLorscheid21b}, univariate polynomials over hereditary fusion bands satisfy the following analogue of the division theorem (the proof follows from the same computation as that given in  \cite[Lemma A]{BakerLorscheid21b}):
\begin{prop} \label{prop:divides}
Let $B$ be a hereditary fusion band and $f(T) = c_0 + c_1 T + \cdots + c_n T^n$ a polynomial, considered as a formal expression in the indeterminate $T$ with all $c_i \in B$ and $c_n \neq 0$.
Then the following are equivalent for an element $a \in B$:
\begin{enumerate}
\item $f(a)=0$, in the sense that $c_0 + c_1 a + \cdots + c_n a^n \in N_B$. 
\item $T-a$ divides $f(T)$, in the sense that there exists a polynomial $g(T) = d_0 + d_1 T + \cdots + d_{n-1}T^{n-1}$ with $d_i \in B$ such that 
\begin{equation} \label{eq:divides}
c_0 = -a d_0, c_i = -a d_i +d_{i-1} \textrm{ for } i = 1,\ldots,n-1, \textrm{ and } c_n = d_{n-1}.
\end{equation}
\end{enumerate}
\end{prop}

Note that the equalities in \eqref{eq:divides} simply express the fact that every coefficient of the formal expansion of $f(T) - (T-a)g(T)$ belongs to the null set of $B$.

\autoref{prop:divides} forms the basis for the inductive definition of multiplicities of roots of polynomials over hyperfields given in \cite{BakerLorscheid21b}.
Clearly a condition such as ``hereditary'' is necessary in order to have such a result, since condition (1) involves the entire null set of $B$ while (2) only involves the three-term relations in $B$.

By definition, a pasture is the same thing as a hereditary fusion band which is also an idyll.

\subsubsection{Fuzzy rings}
Fuzzy rings were introduced by Dress in \cite{Dress86} as a device to generalize matroid theory; also cf.\ \cite[Appendix B]{Baker-Bowler19} and \cite{Giansiracusa-Jun-Lorscheid16} for definitions and further details. A fuzzy ring $(K,+,\cdot,\epsilon,K_0)$ defines the band $B=K^\times\cup\{0\}$ with null set 
\[ 
 N_B \ = \ \big\{ \sum a_i \in K_0 \, \big| \, a_i\in K^\times \big\}.
\]
This association extends naturally to a fully faithful inclusion of categories $\FuzzRings\to\Idylls$. Note that despite the terminology ``fuzzy ring'', $B$ is in fact an idyll. 

Note that axiom (FR5) (in \cite{Dress86}) of a fuzzy ring $K$ is a strong version of the multiplicative fusion rule \ref{F*} for the associated band $B$, but that $B$ is not necessarily a fusion band, i.e.,\ the \emph{additive} fusion axiom \ref{F} fails in general for $B$.

\subsubsection{Idylls as tracts}
\label{subsubsection: idylls as tracts}
Tracts were introduced by the first author and Bowler in \cite{Baker-Bowler19} as a natural algebraic setting for the theory of matroids with coefficients.
Roughly speaking, a tract is an idyll for which the null set is not required to be closed under addition; cf.\ \cite{Baker-Bowler19} for details. Consequently, the category of idylls is a full subcategory of the category of tracts.

\subsubsection{Bands as ordered blueprints}
\label{subsubsection: bands to oblpr}

Ordered blueprints were introduced by the third author as a natural algebraic setting for tropicalization, analytification, and $\F_1$-geometry; cf.\ \cite{Lorscheid22} and \cite{Lorscheid23}. In \cite{Baker-Lorscheid21b}, scheme theory over ordered blueprints is used to construct the moduli space of matroids. In this work, it becomes apparent that matroid theory is insensitive to a part of the structure of an ordered blueprint. From this perspective, a band can be seen as the part of an ordered blueprint that is essential to matroid theory. 
 
 The precise relation between bands and ordered blueprints is as follows (see \cite{Baker-Lorscheid21b} and \cite{Lorscheid23} for the definition of ordered blueprints).
 
 A band $B$ with null set $N_B$ is naturally a blueprint $B^\oblpr$ with monoid $B$, ambient semiring $B^+$
 and the partial order $\leq$ on $B^+$ that consists of all relations $0\leq \sum a_i$ for which $\sum a_i\in N_B$. This extends naturally to a fully faithful embedding $(-)^\oblpr:\Bands\to\OBlpr$. Consequently, the unique morphism $\Funpm\to B$ into a band $B$ yields a morphism $\Funpm\to B^\oblpr$ of ordered blueprints, where we denote the ordered blueprint associated with $\Funpm$ by the same symbol $\Funpm$ by abuse of notation. This shows that the image of $(-)^\oblpr$ is contained in the category $\OBlpr_\Funpm$ of $\Funpm$-algebras, which is a full subcategory of $\OBlpr$.
 
 The embedding $(-)^\oblpr:\Bands\to\OBlpr_\Funpm$ has a left-inverse and right-adjoint $(-)^\band:\OBlpr_\Funpm\to\Bands$, which sends an ordered blueprint $B$ to the band $B^\band=B$ with null set
 \[\textstyle
  N_{B} \ = \ \big\{ \sum a_i \, \big| \, 0\leq \sum a_i \big\}.
 \]

 There is a second natural functor $(-)^\rev:\Bands\to\OBlpr$, which sends a band $B$ to the ordered blueprint $(B,B^+,\leq')$ whose partial order $\leq'$ is generated by all relations $b\leq'\sum a_i$ for which $-b+\sum a_i\in N_B$. This ordered blueprint is \emph{reversible}\footnote{The term reversible stems from the analogy with the reversibility axiom of hyperrings.} in the sense that $b\leq'\sum a_i$ if and only if $0\leq' -b+\sum a_i$, cf.\ \cite[Def.\ 5.6.32]{Lorscheid18}. In other words, the image of $(-)^\rev$ is contained in the full subcategory $\OBlpr^\rev$ of $\OBlpr$ that consists of all reversible ordered blueprints. Note that $\OBlpr^\rev$ is a subcategory of $\OBlpr_\Funpm$.
 
 The restriction to $(-)^\rev:\FBands\to\OBlpr^\rev$ is a fully faithful embedding of categories. Note that if $B$ is a reversible ordered blueprint, then $B^\band$ is a fusion band. Therefore the restriction of $(-)^\band$ to $(-)^\band:\OBlpr^\rev\to\FBands$ is a left-inverse and right-adjoint to $(-)^\rev$.
 
 In conclusion, the composition $(-)^\band\circ(-)^\rev:\Bands\to\FBands$ is a left adjoint to the inclusion $\FBands\to\Bands$, which finds an alternative description in terms of \autoref{prop: band quotients}: it sends a band $B$ to the fusion band $B^\fuse=\bandgennquot{B}{N_B}$; cf.\ \autoref{rem: reflection onto fusion bands}. In particular, $\FBands$ is a reflective subcategory of $\Bands$.

\subsubsection{Conclusion}

We summarize the previous discussions in the commutative diagram of \autoref{fig: diagram of subcategories}, which extends the diagram from \cite[Thm.\ 2.21]{Baker-Lorscheid21b}. All arrows are fully faithful embeddings of categories, and the diagram commutes. 
Note that the top two rows of the diagram feature ``field-like'' objects, in the sense that all nonzero elements are invertible, and the lower two rows consider generalizations of rings.

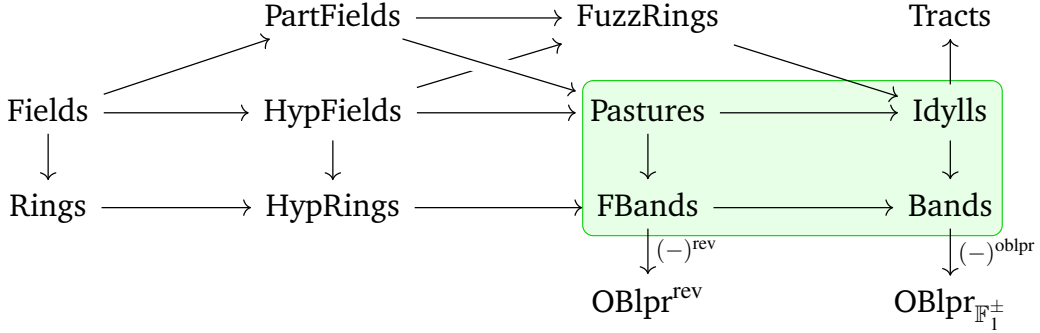
\begin{figure}[htb]
 \[
  \begin{tikzcd}[column sep=55pt, row sep=15pt, execute at end picture={\begin{scope}[on background layer]\draw[rounded corners,draw=green!80!black,fill=green!10] (1.0,-0.92) rectangle (6.55,1.13);\end{scope};}]
                                 & \PartFields \ar[r]               & \FuzzRings \ar[dr] & \Tracts\\
   \Fields \ar[ur] \ar[r] \ar[d] & \HypFields \ar[ur] \ar[r] \ar[d] & \Pastures \ar[r] \ar[d] \ar[<-,ul,crossing over] & \Idylls \ar[d] \ar[u] \\
   \Rings \ar[r]                 & \HypRings \ar[r]                & \FBands \ar[r] \ar[d,"(-)^\rev"]  & \Bands \ar[d,"(-)^\oblpr"] \\
                                 &                                  & \OBlpr^\rev                          & \OBlpr_\Funpm
  \end{tikzcd}
 \] 
 \caption{A commutative diagram of fully faithful embeddings of categories}
 \label{fig: diagram of subcategories}
\end{figure}

\subsection{Free algebras}
\label{subsection: free algebras}

Let $k$ be a band. A \emph{$k$-algebra} is a band $B$ together with a morphism $\alpha_B:k\to B$, which we call the \emph{structure map}. A ($k$-linear) morphism of $k$-algebras is a band morphism $f:B\to C$ between $k$-algebras that commutes with the respective structure maps, i.e.,\ $\alpha_C=f\circ\alpha_B$. This defines the category $\Alg_k$ of $k$-algebras. If $k$ is a ring, then we denote by $\Alg_k^+$ the full subcategory of $\Alg_k$ whose objects are $k$-algebras in the usual sense, i.e.,\ rings $R$ together with a ring homomorphism $k\to R$.

Let $B$ be a band, and let $I$ be an index set. We define the \emph{free $B$-algebra} in $\{x_i\mid i\in I\}$ as follows. Its underlying set is the free pointed monoid
\[ \textstyle
 B[x_i] \ = \ B[x_i \mid i\in I] \ = \ \big\{ a \cdot \prod_{i\in I} x_i^{n_i} \, \big| \, a\in B,\ (n_i)\in\bigoplus_{i\in I}\N \big\}
\]
generated by $\{x_i\}$ over $B$, where $0\cdot \prod_{i\in I} x_i^{n_k}$ is identified with $0$ for all $(n_i)\in\bigoplus_{i\in I}\N$. We write $a$ for $a\cdot\prod x_i^0$ and denote by $\iota:B\hookrightarrow B[x_i]$ the tautological inclusion, which extends to an inclusion $\iota^+:B^+\hookrightarrow B[x_i]^+$ of semirings. The null set of $B[x_i]$ is 
\[ \textstyle
 N_{B[x_i]} \ = \ \gen{\iota^+(N_B)}_{B[x_i]^+} \ = \ \gen{\sum a_i\in B[x_i]^+\mid \sum a_i\in N_B}_{B[x_i]^+}.
\]

Note that by definition, $\iota:B\to B[x_i]$ is a band morphism. We write $x_i$ for $1\cdot\prod x_j^{n_j}$ with $n_j=\delta_{i,j}$ and denote by $\iota_0:\{x_i\}\hookrightarrow B[x_i]$ the tautological inclusion. The free algebra $B[x_i]$ together with $\iota$ and $\iota_0$ satisfies the following universal property.

\begin{prop}\label{prop: universal property of the free algebra}
 For every band morphism $\alpha_C: B \to C$ and for every map $f_0:\{x_i\}\to C$, there is a unique band morphism $f: B[x_i] \to C$ such that $f_0=f\circ\iota_0$ and $\alpha_C=f\circ\iota$, i.e.,\ the diagrams 
 \[
  \begin{tikzcd}[column sep=60]
   \{x_i\} \arrow{dr}{f_0} \arrow[swap]{d}{\iota_0} \\
   B[x_i] \arrow[swap,dashed]{r}{f} & C 
  \end{tikzcd}
  \hspace{1.5cm}\text{and}\hspace{1.5cm}
  \begin{tikzcd}[column sep=60]
   B \arrow{dr}{\alpha_C} \arrow[swap]{d}{\iota} \\
   B[x_i] \arrow[swap,dashed]{r}{f} & C 
  \end{tikzcd}
 \]
commute. 
\end{prop}

\begin{proof}
 Since $f_0=f\circ\iota_0$ and $\alpha_C=f\circ\iota$, the band morphism $f: B[x_i] \to C$ satisfies $f(a\prod x_i^{n_i})=\alpha_C(a)\prod f_0(x_i)^{n_i}$ if it exists. This establishes uniqueness, and if we can show that $f$ is a well-defined band morphism, then it follows by definition that $f_0=f\circ\iota_0$ and $\alpha_C=f\circ\iota$.
 
 Since $B[x_i]$ is freely generated by $\{x_i\}$ as a pointed monoid over $B$, the map $f$ is a well-defined monoid morphism. Consider a generator $\sum\iota(a_i)$ of $N_{B[x_i]}$, i.e.,\ $\sum a_i\in N_B$. Then $\sum f(\iota(a_i))=\sum\alpha_C(a_i)\in N_C$ since $\alpha_C$ is a band morphism. By \autoref{lemma: morphisms can be tested on generators}, $f:B[x_i]\to C$ is a band morphism, as claimed.
\end{proof}

\subsubsection{Monoid algebras}\label{subsubsection: monoid algebras}
The construction of the free algebras generalizes to monoid algebras over $B$, which we describe briefly in the following. Given a band $B$ and a (multiplicatively written) commutative monoid $A$, we define the \emph{monoid algebra of $A$ over $B$} as the pointed monoid $B[A]=(B\times A)/\sim$ where the equivalence relation $\sim$ is generated by $(0,a)\sim (0,a')$ for all $a,a'\in A$. We write $ba=b\cdot a$ for the equivalence class of $(b,a)$ in $B[A]$. The product of $B[A]$ is given by $(ba)\cdot(b'a')=(bb')(aa')$, its unit is $1=1\cdot 1$, and its zero is $0=0\cdot 1$. The map $\iota:B\to B[A]$ with $\iota(b)=b\cdot 1$ is an inclusion of pointed monoids, and we write simply $b$ for $\iota(b)=b\cdot 1$. 

The pointed monoid $B[A]$ comes with the null set $N_{B[A]}=\gen{\iota^+(N_B)}_{B[A]^+}$, which turns $\iota:B\to B[A]$ into a band morphism and $B[A]$ into a $B$-algebra. The map $\iota_A:A\to B[A]$ with $\iota_A(a)=1\cdot a$ is a monoid morphism. 

The $B$-algebra $B[A]$ together with $\iota_A$ satisfies the following universal property: given a $B$-algebra $C$, with structure map $\alpha_C:B\to C$, and a monoid morphism $f_A:A\to C$, there is a unique  $B$-linear map $f:B[A]\to C$ such that $f_A=f\circ\iota_A$, i.e., the diagram 
 \[
  \begin{tikzcd}[column sep=60]
   A \arrow{dr}{f_A} \arrow[swap]{d}{\iota_A} \\
   B[A] \arrow[swap,dashed]{r}{f} & C 
  \end{tikzcd}
 \]
commutes. This can be proved in a similar way as \autoref{prop: universal property of the free algebra}.

We recover the free algebra $B[x_i]$ as the monoid algebra $B[A]$ for the monoid $A$ of finite products $x_{i_1}\dotsb x_{i_r}$ of the symbols $x_i$. More general monoid algebras appear in the theory of toric band schemes; cf.\ \autoref{ex: toric band schemes}.

\subsection{Quotients}

A quotient of a band $B$ is an isomorphism class of surjections $\pi:B\to C$. In this section, we study the quotients of $B$.

\begin{df}\label{def: fusion ideals}
 A \emph{null ideal} of $B$ is an ideal $I$ of $B^+$ that contains $N_B$ and satisfies the \emph{substitution rule}
 \begin{enumerate}[label=\rm(SR)] 
  \item\label{SR} if $a-c\in I$ and $c+\sum b_j\in I$, then $a+\sum b_j\in I$.
 \end{enumerate}
 A \emph{fusion ideal} or \emph{$f$-ideal} of $B$ is a null ideal that satisfies the \emph{fusion rule}
 \begin{enumerate}[label=\rm(FR)]    
  \item\label{FR} if $\sum a_i-c\in I$ and $c+\sum b_j\in I$, then $\sum a_i + \sum b_j\in I$. 
 \end{enumerate}
 Given a subset $S$ of $B^+$, we denote by $\gen{S}$ the null ideal generated by $S$ and by $\genn{S}$ the fusion ideal generated by $S$.
\end{df}
 
Note that the substitution rule \ref{SR} is a special case of the fusion rule \ref{FR}.

\begin{df}
 Let $f:B\to C$ be a band morphism. The \emph{null kernel of $f$} is the subset
 \[\textstyle
  \nullker f \ = \ \big\{ \sum a_i\in B^+ \, \big| \, \sum f(a_i)\in N_C \big\}
 \]
 of $B^+$.
\end{df}

\begin{lemma}\label{lemma:nullker}
 The null kernel $\nullker f$ of a band morphism $f:B\to C$ is a null ideal. If $C$ is a fusion band, then $\nullker f$ is a fusion ideal.
\end{lemma}

\begin{proof}
 The band morphism $f$ extends uniquely to a semiring homomorphism $f^+:B^+\to C^+$ via $f^+(\sum a_i)=\sum f(a_i)$. The null kernel of $f$ is nothing other than the inverse image of $N_C$  under $f^+$. Since the inverse image of an ideal is an ideal, $\nullker f$ is an ideal. In order to verify the substitution rule \ref{SR}, let $a-c$ and $c+\sum b_j$ be in $\nullker f$, i.e.,\ $f(a)-f(c)\in N_C$ and $f(c)+\sum f(b_j)\in N_C$. We conclude that $f(a)=f(c)$ in $C$ and thus $f(a)+\sum f(b_j)\in N_C$. This shows that $a+\sum b_j\in\nullker f$, as desired.
 
 Let $C$ be a fusion band and let $\sum a_i-c$ and $c+\sum b_j$ be in $\nullker f$. Then $\sum f(a_i)-f(c)$ and $f(c)+\sum f(b_j)$ in $N_C$. By the fusion axiom \ref{F}, which holds in $C$, we have $\sum f(a_i) + \sum f(b_j)\in N_C$. Thus $\sum a_i + \sum b_j\in \nullker f$, which shows that the null kernel of $f$ is a fusion ideal.
\end{proof}

Recall from \cite[Def.\ 2.1]{Lorscheid-Ray23} that a congruence on a pointed monoid $B$ is an equivalence relation $\sim$ on $B$ such that $a\sim b$ implies $ac\sim bc$ for all $a,b,c\in B$. The product $[a]\cdot[b]=[ab]$ of equivalence classes is independent of the choice of representatives, and turns $B/\sim$ into a pointed monoid.

\begin{prop}\label{prop: band quotients}
 Let $B$ be a band and $I$ a null ideal. 
 \begin{enumerate}
  \item\label{quot1} The relation $\sim$ on $B$ with $a\sim b$ if $a-b\in I$ is a congruence on $B$.
  \item\label{quot2} The monoid $\bandquot BI=B/\sim$ together with the null set
        \[\textstyle
         N_{\bandquot BI} \ = \ \big\{ \sum [a_i] \, \big| \, \sum a_i \in I \big\}
        \]
        is a band. It is a fusion band if and only if $I$ is a fusion ideal.
  \item\label{quot3} The quotient map $\pi_I:B\to \bandquot BI$ is a band morphism with null kernel $I$.
 \end{enumerate}
\end{prop}

\begin{proof}
 Let $a\sim b$ and $c\in B$. Then $a-b\in I$, hence $ac-bc\in I$, and thus $ac\sim bc$, which shows that $\sim$ is a congruence on $B$. This establishes \eqref{quot1}.
 
 Let $\sum a_i\in I$ and $a_i\sim b_i$ for all $i$, i.e.,\ $a_i-b_i\in I$. Applying the substitution rule \ref{SR} iteratively to all $a_i-b_i$ shows that $\sum b_i\in I$. This shows that $\sum [a_i]\in N_{\bandquot BI}$ if and only if $\sum a_i\in I$, independent of the choice of representatives $a_i$ for $[a_i]$.
 
 Choosing representatives allows us to conclude that $N_{\bandquot BI}$ is an ideal from the corresponding property of $I$. We have $[a] + [b]\in N_{\bandquot BI}$ if and only if $a + b\in I$, which is equivalent to $[b]=-[a]$ in $\bandquot BI$. This shows that $\bandquot BI$ is a band, which establishes the first claim of \eqref{quot2}. If $I$ is a fusion ideal, then the fusion axiom of $N_{\bandquot BI}$ can be deduced from the fusion axiom for $I$ by choosing representatives, thus the second claim of \eqref{quot2}. 
 
 The map $\pi_I$ is a band morphism by the definition of $\bandquot BI$. Since $\pi_I(a)=[a]$, its null kernel is $\{\sum a_i\mid\sum [a_i]\in N_{\bandquot BI}\}=I$, from which \eqref{quot3} follows.
\end{proof}

For a subset $S$ of $B^+$, we denote by $\bandgenquot BS$ the quotient of $B$ by the null ideal $\gen S$ generated by $S$, and by $\bandgennquot BS$ the quotient of $B$ by the fusion ideal $\genn S$ generated by $S$.

\begin{prop}\label{prop: universal property of the quotient}
 Let $B$ be a band and $S\subset B^+$. Every morphism $f:B\to C$ into a band $C$ with $\sum f(a_i)\in N_C$ for $\sum a_i\in S$ factors uniquely through the quotient map $\pi_{\gen S}:B\to \bandgenquot BS$. If $C$ is a fusion band, then $f$ factors uniquely through the quotient map $\pi_{\genn S}:B\to\bandgennquot BS$.
\end{prop}

\begin{proof}
 Since $\pi_{\gen S}:a\mapsto[a]$ is surjective, the only candidate for a morphism $\bar f:\bandgenquot BS\to C$ with $f=\bar f\circ\pi_{\gen S}$ is given by $\bar f([a])=f(a)$, which establishes the uniqueness claim. We continue with the verification that $\bar f$ is a well-defined band morphism.
 
 Let $[a]=[b]$ in $\bandgenquot BS$, i.e.,\ $a-b\in \gen S$. By assumptions $S\subset (f^+)^{-1}(N_C)$. By \autoref{lemma:nullker}, $(f^+)^{-1}(N_C)=\nullker f$ is a null ideal of $B$ and thus contains $\gen S$. This shows that $f(a)-f(b)\in N_C$ and thus $f(a)=f(b)$. Thus $\bar f([a])=f(a)=f(b)=\bar f([b])$ is independent of the choice of representative for $[a]=[b]$. We have $\bar f([1])=f(1)=1$ and $\bar f([a]\cdot [b])=f(ab)=f(a)f(b)=\bar f([a])\bar f([b])$ for $a,b\in B$, which shows that $\bar f$ is a monoid morphism. If 
 $\sum[a_i]\in N_{\bandgenquot BS}$, i.e.,\ $\sum a_i\in \gen S$, then $\sum\bar f([a_i])=\sum f(a_i)\in N_C$, which completes the proof that $\bar f$ is indeed a band morphism.
 
 Let $C$ be a fusion band. Then $\bar f:\bandgennquot BS\to C$ is a well-defined morphism of pointed monoids for the same reasons as above, but with \autoref{lemma:nullker} applied to the fusion kernel of $f$. In fact, the underlying monoids of $\bandgenquot BS$ and $\bandgennquot BS$ agree. Let $\sum[a_i]-[c],\ [c]+\sum[b_j]\in N_{\bandgennquot BS}$. Then $\sum\bar f([a_i])-\bar f([c])=\sum f(a_i)-f(c)$ and $\bar f([c])+\sum\bar f([b_j])=f(c)+\sum f(b_j)$ are in $N_C$, and therefore $\sum\bar f([a_i])+\sum\bar f([b_j])\in N_C$ by the fusion axiom \ref{F}. This shows that $f$ factors (necessarily uniquely) through $\pi_{\genn S}:B\to\bandgennquot BS$.
\end{proof}

\begin{cor}\label{cor: bijection between null ideals and quotients}
 The association $I\mapsto \bandquot BI$ establishes a bijection
 \[
  \Phi: \ \big\{\text{null ideals of $B$} \big\} \ \longrightarrow \ \{ \text{quotients of $B$}\big\}
 \]
 whose inverse $\Psi$ sends a quotient $[\pi:B\onto C]$ to the null kernel of $\pi$. 
\end{cor}

\begin{proof}
 By \autoref{prop: band quotients}, the null kernel of $\pi_I:B\to\bandquot BI$ is $I$, which shows that $\Psi\circ\Phi=\id$. Conversely, suppose we are given a surjective band morphism $\pi:B\to C$. Let $I$ be its null kernel and $\pi_I:B\to\bandquot BI$ the quotient map. Then $\sum\pi(a_i)\in N_C$ for $\sum a_i\in I$, which determines a unique band morphism $\bar\pi:\bandquot BI\to C$ with $\pi=\bar\pi\circ\pi_I$ by \autoref{prop: universal property of the quotient}. 
 
 Since $\pi$ is surjective, also $\bar\pi$ is surjective. Given $[a],[b]\in\bandquot BI$ with $\pi(a)=\bar\pi([a])=\bar\pi([b])=\pi(b)$, we find $a-b\in I$ and thus $[a]=[b]$ in $\bandquot BI$. This shows that $\bar\pi:\bandquot BI\to C$ is a bijection. If $\sum[a_i]\in N_{\bandquot BI}$, then $\sum\bar\pi([a_i])\in N_C$ since $\bar\pi$ is a band morphism. Conversely, assume that $\sum\pi(a_i)=\sum \bar\pi([a_i])\in N_C$. Then $\sum a_i\in I$ by the definition of $I=\nullker \pi$ and thus $\sum [a_i]\in N_{\bandquot BI}$ by the definition of $\bandquot BI$. This shows that $\bar\pi$ is an isomorphism of bands and thus $\Phi\circ\Psi=\id$.
\end{proof}

\begin{ex}
 As a consequence of the results in this section, every band $B$ can be written as $\bandquot{\Funpm[x_a\mid a\in B]}{I}$ where $I$ is the null kernel of the canonical map $\Funpm[x_a\mid a\in B]\to B$ that sends $x_a$ to $a$. Thus every band can be written in terms of generators and relations over $\Funpm$. Typically it requires less generators than elements in $B$ to describe the band $B$. Some examples are:
 \begin{align*}
   \F_2 &= \bandgenquot{\Funpm}{1+1},     
   & \quad
   \F_4 &= \bandgenquot{\Funpm[x]}{1+1,\ x^3+1,\ x^2+x+1}, 
   \\
   \F_3 &= \bandgenquot{\Funpm}{1+1+1},   
   &\quad
   \F_5 &= \bandgennquot{\Funpm[x]}{x^2+1,\ x-1-1},  
   \\
   \K   &= \bandgenquot{\Funpm}{1+1,\ 1+1+1},    
   & \quad
   \S   &= \bandgennquot{\Funpm}{1+1-1}.
 \end{align*}
 Note that the relation $1+1$, which appears in the null sets of $\F_2$, $\F_4$ and $\K$, implies that $-1=1$. Similarly, $x^2+1$ implies that $x^2=-1$ in $\F_5$.
 
 All of these examples are fusion bands, since both (partial) fields and hyperfields satisfy the fusion axiom. It can be verified directly that the null sets of $\F_2$, $\F_3$, $\F_4$, and $\K$ are generated by the mentioned elements as a null ideal. This is not the case for $\F_5$ and $\S$. The null set of $\F_5$ is finitely generated as a null ideal, but the mentioned elements generate a subideal that does not contain $1+1+1+1+1$, for instance. The null set of $\S$ is not finitely generated as a null ideal since it contains elements of the form $1+\dotsc+1-1$ with an arbitrary (positive) number of summands ``$+1$.''
\end{ex}

The \emph{Laurent algebra over $B$ in $\{x_i\}$} is the $B$-algebra
\[
 B[x_i^{\pm1}] \ = \ B[x_i^{\pm1}\mid i\in I] \ = \ \bandgenquot{B[x_i, y_i\mid i\in I]}{x_iy_i-1\mid i\in I}.
\]
If $B=P$ is a pasture, then $P[x_i^{\pm1}]$ is also a pasture, and we write $P(x_i)=P[x_i^{\pm1}]$ in analogy to the notation for rational function fields.

\begin{ex}
 This notation allows us to express many pastures as quotients of Laurent algebras, such as:
 \begin{align*}
  {} & \text{the \emph{near-regular partial field}}      & \qquad & \U \ = \ \pastgennquot{\Funpm(x,y)}{x+y-1}, \\
  {} & \text{the \emph{dyadic partial field}}       & \qquad & \D \ = \ \pastgennquot{\Funpm(x)}{x+x-1}, \\
  {} & \text{the \emph{golden ratio partial field}} & \qquad & \G \ = \ \pastgennquot{\Funpm(x)}{x^2+x-1}.
 \end{align*}
\end{ex}

\begin{rem}\label{rem: reflection onto fusion bands}
 Note that if $S$ in~\autoref{prop: universal property of the quotient} contains an element of $B^\times$, then $\bandgenquot BS=\{0\}$ is the trivial band. The quotient map $B\to\bandgenquot BS$ is a bijection if and only if every relation of the form $a+b\in S$ is contained in $N_B$, i.e.,\ $b=-a$.
 
 The fusion quotient yields an endofunctor on $\Bands$ that sends a band $B$ to an associated fusion band $\bandgennquot{B}{N_B}$, together with a quotient map $B\to\bandgennquot{B}{N_B}$.
\end{rem}

\subsection{The base extension to rings}
\label{subsection: universal ring}

A band comes with a universal ring, which we describe in the following.

\begin{df}
 Let $B$ be a band with null set $N_B$. The \emph{universal ring of $B$} is the quotient ring $B^+_\Z=\Z[B]/\gen{N_B}$ of the monoid ring $\Z[B]$ by the ideal $\gen{N_B}$. It comes with a multiplicative map $\rho:B\to B^+_\Z$.
\end{df}

\begin{ex}
 The universal ring of $\Funpm$ is $\Z$. The universal ring of $\K$ is the trivial ring $\{0\}$. If $B$ is a band associated with a ring $R$, then the universal ring of $B$ is $R$ itself. The universal ring of a partial field $P$ (considered as a band) is the universal ring in the sense of Pendavingh and van Zwam; cf. \cite{Pendavingh-vanZwam10a}.
\end{ex}

\begin{prop}\label{prop: universal property of the universal ring}
 Let $B$ be a band and $R$ a ring. Then for every band morphism $f:B\to R$, there is a unique ring homomorphism $f^+_\Z:B^+_\Z\to R$ such that $f=f^+_\Z\circ\rho$, i.e.,\ the diagram
 \[
\begin{tikzcd}[column sep=60, row sep=15]
  B \arrow{r}{f} \arrow[swap]{d}{\rho} & R \\
  B^+_\Z \arrow[swap,dashed]{ur}{f^+_\Z} 
  \end{tikzcd}
\]
commutes. In other words, $\Rings$ is a reflective subcategory of $\Bands$ with reflection $(-)^+_\Z:\Bands\to\Rings$.
\end{prop}

More generally, we define for a band $B$ and a ring $K$ the \emph{universal $K$-algebra of $B$} as the ring $B^+_K=B^+_\Z\otimes_\Z K$, which comes with a structure map $K\to B^+_K$. It satisfies the universal property that every band morphism $f:B\to R$ to a $K$-algebra $R$ (in the usual sense, i.e.,\ $R$ is a ring) extends uniquely to a $K$-linear ring homomorphism $B^+_K\to R$.

\subsection{Localizations}
\label{subsection: localizations}

\begin{df}
 Let $B$ be a band and $S$ a \emph{multiplicative subset} of $B$, which is a subset $S$ of $B$ that contains $1$ and is closed under multiplication. As a pointed monoid, the localization of $B$ in $S$ is 
 \[
  S^{-1}B \ = \ (S\times B)/\sim
 \]
 for the equivalence relation $\sim$ with $(s,a)\sim (s',a')$ if and only if there is a $t\in S$ such that $tsa'=ts'a$. We write $\frac as$ for the equivalence class of $(s,a)$ in $S^{-1}B$. 
 
 The pointed monoid $S^{-1}B$ is equipped with the null set 
 \[\textstyle
  N_{S^{-1}B} \ = \ \gen{ \sum \frac{a_i}1 \mid \sum a_i\in N_B }_{(S^{-1}B)^+},
 \]
 which turns $S^{-1}B$ into a band. It comes with the \emph{localization map} $\iota_S:B\to S^{-1}B$, which is the band morphism defined by $\iota_S(a)=\frac a1$. 
 
 A \emph{finite localization\footnote{The term ``finite localization'' is derived from the fact that $B[h_1^{-1},\dotsc,h_n^{-1}]=B[(h_1\dotsb h_n)^{-1}]$, where $B[h_1^{-1},\dotsc,h_n^{-1}]$ is, \emph{a priori}, defined as $S^{-1}B$ for $S=\{h_1^{e_1}\dotsb h_n^{e_n}\mid e_1,\dotsc,e_n\in\N\}$.}
 of $B$} is a localization of the form $B[h^{-1}]=S^{-1}B$ for $h\in B$ and $S=\{h^i\}_{i\in\N}$. In this case, we write $\iota_h:B\to B[h^{-1}]$ for the localization map.
\end{df}

\begin{prop}\label{prop: universal property of the localization}
Let $B$ be a band and $S$ a multiplicative set in $B$. Let $f: B\to C$ be a band morphism with $f(S)\subset C^\times$. Then there is a unique band morphism $f_S: S^{-1}B \to C$ with $f=f_S\circ\iota_S$, i.e.,\ the diagram
\[
\begin{tikzcd}[column sep=60, row sep=15]
B \arrow{r}{f} \arrow[swap]{d}{\iota_S} & C \\
S^{-1}B \arrow[swap,dashed]{ur}{f_S} 
  \end{tikzcd}
\]
commutes. 
\end{prop}

\begin{proof}
 The relation $f_S(\frac as)=f_S(\iota_S(s))^{-1}f_S(\iota_S(a))=f(s)^{-1}f(a)$ determines $f_S$ uniquely. We need to show that $f_S$ is a well-defined band morphism. 
 
 We omit the straight-forward verification that $f_S$ is well-defined and a morphism of pointed monoids (also cf.\ \cite[Exer.~3.6.3]{Lorscheid18}). By \autoref{lemma: morphisms can be tested on generators}, it suffices to show that $f_S$ preserves a generating set for the null set of $S^{-1}B$, which is given by terms of the form $\sum\frac{a_i}1$ for which $\sum a_i\in N_B$. Since $f$ is a band morphism, $\sum f(\frac{a_i}1)=\sum f_S(\frac11)^{-1}f_S(\frac{a_i}1)=\sum f(a_i)$ is in $N_C$, which shows that $f_S$ is a band morphism.
\end{proof}

\begin{rem}
 In analogy to the case of rings, the localization $S^{-1}B$ is isomorphic to $\bandgenquot{B[s^- \mid s \in S]}{ss^--1}$. 
\end{rem}

\subsection{\texorpdfstring{$m$}{m}-Ideals and \texorpdfstring{$k$}{k}-ideals}
\label{subsection: m-ideals and k-ideals}

\begin{df}
Let $B$ be a band. A subset $I \subset B$ is an \emph{$m$-ideal} of $B$ if 
\begin{enumerate}[label=\rm(I\arabic*)]
\item\label{I1} $0 \in I$. 
    
\item\label{I2} $B \cdot I = I$. 
\end{enumerate}
A \emph{$k$-ideal} is an $m$-ideal $I$ that satisfies:
\begin{enumerate}[label=\rm(I3)]
\item\label{I3} if $a + \sum b_i \in N_B$ and $b_i \in I$, then $a \in I$. 
\end{enumerate}
 Given a subset $S$ of $B$, we denote by $\gen{S}_m$ the smallest $m$-ideal and by $\gen{S}_k$ the smallest $k$-ideal of $B$ that contains $S$. 
\end{df}

In the case of pointed monoids, $m$-ideals are simply called ideals; cf.\ \cite{Chu-Lorscheid-Santhanam12} and \cite{Connes-Consani10a}. We use the term $m$-ideal in order to distinguish them from other types of ideals for bands.

\begin{ex}
 The trivial ideal $\gen\emptyset_m=\gen\emptyset_k=\{0\}$ is a $k$-ideal of $B$. 
\end{ex}

\begin{lemma}\label{lemma: m-ideal and k-ideal generated by a subset}
 Let $B$ be a band and $S$ a nonempty subset. Then $\gen{S}_m=\{as\mid a\in B,\ s\in S\}$. Define recursively $\gen{S}_k^{(0)}=S$ and for $n>0$,
 \[\textstyle
  \gen{S}_k^{(n)} \ = \ \{ a\in B\mid a-\sum b_is_i\in N_B\text{ for some }b_i\in B,\ s_i\in \gen{S}_k^{(n-1)}\}.
 \]
 Then $\gen{S}_k=\bigcup_{n\geq0}\gen{S}_k^{(n)}$. If $B$ is a fusion band, then $\gen{S}_k=\gen{S}_k^{(1)}$.
\end{lemma}

\begin{proof}
 To begin with, we note that $\{as\mid a\in B,\ s\in S\}$ is an $m$-ideal: it contains $0=0\cdot s$, which shows \ref{I1}, and $b(as)=(ba)s$ for $a,b\in B$ and $s\in S$, which shows \ref{I2}. It also contains every element $s=1\cdot s$ of $S$. Conversely, if an $m$-ideal $I$ contains $S$, then it contains every element of the form $as$ with $a\in B$ and $s\in S$. This establishes the description of $\gen{S}_m$.

 Next we show that $\bigcup_{n \geq 0} \gen{S}_k^{(n)}$ is the smallest $k$-ideal that contains $S$. The axioms~\ref{I1} and~\ref{I2} are straightforward and follow from a similar argument as in the previous paragraph. We notice further that $\gen{S}_k^{(n)} \subseteq \gen{S}_k^{(n+1)}$, since every element $a \in \gen{S}_k^{(n)}$ satisfies $a - a \in N_B$ for $a \in \gen{S}_k^{(n)}$ and hence $a \in \gen{S}_k^{(n+1)}$. To verify~\ref{I3}, pick $a + \sum b_i \in N_B$ and assume without loss of generality that $b_i \in \gen{S}_k^{(n)}$. Then evidently $a \in \gen{S}_k^{(n+1)} \subseteq \bigcup_{n \geq 0} \gen{S}_k^{(n)}$. Conversely, if a $k$-ideal $I$ contains $\gen{S}_k^{(n)}$, with $\gen{S}_k^{(0)}=S$, then it contains every element $a$ such that $a - \sum b_is_i\in N_B$ for some $b_i \in B$ and $s_i \in \gen{S}_k^{(n-1)}$ by axioms~\ref{I2} and~\ref{I3}. This completes the description of $\gen{S}_k$. 

 If $B$ is a fusion band, all we need to do is to verify that $\gen{S}_k^{(1)}$ is a $k$-ideal. Since $0\in N_B$, it contains $0$; thus \ref{I1} holds for $\gen{S}_k^{(1)}$. If $a+\sum b_is_i\in N_B$, then $ca+\sum cb_is_i\in N_B$ and thus $ca\in \gen{S}_k^{(1)}$, which verifies \ref{I2} for $\gen{S}_k^{(1)}$. Let $a-\sum_{i=1}^n b_i\in N_B$ with $b_i\in \gen{S}_k^{(1)}$, i.e.,\ $b_i-\sum_{j=1}^{m_i} b_{ij}s_{ij}\in N_B$ for some $m_i\in\N$,\ $b_{ij}\in B$, and $s_{ij}\in S$. Applying the fusion axiom \ref{F} successively to all $b_i$ yields $a-\sum_{i=1}^n\sum_{j=1}^{m_i} b_{ij}s_{ij}\in N_B$. Thus $a\in \gen{S}_k^{(1)}$, which establishes \ref{I3} for $\gen{S}_k^{(1)}$. Therefore $\gen{S}_k^{(1)}$ is a $k$-ideal.
\end{proof}

The ``$k$'' in ``$k$-ideal'' stands for ``kernel,'' which refers to the fact that the class of $k$-ideals of a band $B$ agrees with the class of kernels of morphisms from $B$ to other bands, a notion that is defined as follows. 

\begin{df}
 Let $f:B\to C$ be a band morphism. The \emph{kernel of $f$} is the subset $\ker f=\{a\in B\mid f(a)=0\}$ of $B$.
\end{df}

\begin{lemma}\label{lemma: null ideal generated by a k-ideal}
Let $I$ be a $k$-ideal of a fusion band $B$. Then the null ideal $\gen{I}$ generated by $I$ is given by 
\[\textstyle
 \gen{I} = \{\sum s_i + \sum a_j \mid s_i \in I,\ \sum a_j +\sum t_k\in N_B\text{ for some }t_k\in I\}. 
\]
In particular, $\gen{I} \cap B = I$. 
\end{lemma}

\begin{proof}
 Let $T$ denote the right-hand side of the equation. Since the null ideal $\gen{I}$ contains both $I$ and $N_B$, and since the substitution rule \ref{SR} allows us to exchange elements $a_i\in I$ that appear in a sum $\sum a_i\in N_B$ with other elements $t_k\in I$, the null ideal $\gen{I}$ contains $T$. Thus the equality follows if we can show that $T$ is a null ideal. 
 
 Evidently $T$ contains $0$ and is closed under sums and and multiplication by elements in $B$. We are left with verifying the substitution rule \ref{SR}. Let $a-c \in T$ and $c + \sum b_i + \sum d_j \in T$, where $b_i \in I$ and $d_j \notin I$. If $a \in I$, then from $a - c \in T$ we see that $a - c + \sum t_k \in N_B$ for some $t_k \in I$. Since $I$ is a $k$-ideal, $c$ must be in $I$. In this case we have $a + \sum b_i + \sum d_j \in T$, and thus \ref{SR} holds. By symmetry, the same argument works if $c\in I$. 
 
 If $a, c \notin I$, then there are $t_k \in I$ with $a - c + \sum t_k \in N_B$. Since $c + \sum b_i + \sum d_j \in T$ with $b_i \in I$ and $d_j \notin I$, there are $r_{\ell} \in I$ with $c + \sum d_j + \sum r_{\ell} \in N_B$. Since $B$ is a fusion band, $a + \sum t_k + \sum d_j + \sum r_{\ell} \in N_B$, and hence $a + \sum b_i + \sum d_j \in T$. This completes the proof that $T$ is a null ideal. 
\end{proof}

\begin{lemma}\label{lemma: inverse images of ideals and k-ideals as kernels}
 Let $f:B\to C$ be a band morphism and $I$ an $m$-ideal of $C$. Then $f^{-1}(I)$ is an $m$-ideal of $B$. If $I$ is a $k$-ideal, then so is $f^{-1}(I)$. In particular, $\ker f$ is a $k$-ideal. More precisely, a subset $S$ of $B$ is a $k$-ideal if and only if it is the kernel of the quotient map $\pi_{\gen S}:B\to\bandgenquot BS$.
\end{lemma}

\begin{proof}
 Since $f(0) = 0 \in I$, we have $0 \in f^{-1}(I)$. If $a \in f^{-1}(I)$ and $b \in B$, then $f(b a) = f(b)f(a) \in I$. Hence $ba \in I$. This shows $f^{-1}(I)$ is an $m$-ideal of $B$ and establishes the first claim.

 Assume that $I$ is a $k$-ideal of $C$, and let $a + \sum b_i \in N_B$ with $b_i \in f^{-1}(I)$. Then $f(a) + \sum f(b_i) \in N_C$ and $f(b_i) \in I$. This forces $f(a) \in I$, and therefore $f^{-1}(I)$ is a $k$-ideal of $B$, which establishes the second claim. 

 This implies, in particular that the kernel $\ker f=f^{-1}(0)$ of a band morphism $f$ is a $k$-ideal, since $\{0\}$ is a $k$-ideal. Conversely, assume that $I$ is a $k$-ideal of $B$. We will show that $I$ is the kernel of the quotient map $\pi:B\to\bandgenquot BI$ by constructing $\bandgenquot BI$ explicitly. 
 
 As a first step, we define the quotient monoid $\overline B=B/I=\{\bar 0\}\sqcup(B-I)$, where $\bar 0=\bar a$ is the class of all elements $a\in I$ and $\bar a=\{a\}$ for $a\in B-I$. We define the subset $\overline N_B=\{\sum \bar a_i\mid \sum a_i\in N_B\}$ of the semiring $\overline B^+$, which is an ideal, but not a null set in general since a class $\bar a\in \overline B$ might possess multiple additive inverses, i.e., we might have $\bar a-\bar b\in \overline N_B$ for $\bar b\neq-\bar a$. To remedy this, we define the multiplicative relation $\sim$ on $\overline B$ generated by $\bar a\sim \bar b$ if $\bar a-\bar b\in\overline N_B$.
 
 We claim that $\widetilde{B}=\overline B/\sim$, together with $\widetilde N_B=\{\sum [\bar a_i]\in\widetilde B^+\mid \sum\bar a_i\in \overline N_B\}$, is a band. Since $\sim$ is a multiplicative relation, $\widetilde B$ is a pointed monoid. That $\widetilde N_B$ is an ideal of $\widetilde B^+$ can be verified by choosing suitable representatives. By the very definition of $\sim$, additive inverses are unique, i.e.,\ $[\bar a]-[\bar b]\in \widetilde N_B$ if and only if $[\bar b]=-[\bar a]$. This shows that $\widetilde B$ is a band with null set $\widetilde N_B$.
 
 Let $\pi:B\to \widetilde B$ be the quotient map. By construction, it is clear that $I\subset\ker\pi$ and that $\widetilde B$ is the minimal quotient for which $\pi(I)=0$. In conclusion, this shows that $\widetilde B=\bandgenquot BI$. It is also clear by construction that $I$ is the kernel of the monoid morphism $B\to \overline B$. The claim that $I=\ker \pi$ follows if we can show that the quotient map $\bar\pi:\overline B\to\widetilde B=\bandgenquot BI$ has trivial kernel.
 
 Let $\bar d\in\ker\bar\pi$, i.e., $\bar d\sim\bar 0$ in $\overline B$. Since $\sim$ is the transitive closure of relations of the form $\bar c\bar a\sim\bar c\bar b$ with $\bar a-\bar b\in\overline N_B$, we can assume that $\bar d=\bar c\bar a$ and that $\bar c\bar b=\bar 0$. Since $\overline N_B$ is closed under multiplication by $\overline B$, we have $\bar d-\bar 0=\bar c\bar a-\bar c\bar b\in\overline N_B$, which means that $d'+\sum e_i\in N_B$ for some $c_k\in I$ and $d'\in B$ with $[d']=[d]$, i.e., either $d,d'\in I$ or $d'=d$. In the latter case, we conclude that also $d=d'$ is in $I$, since $I$ is a $k$-ideal. In either case, we have shown that $\bar d=\bar 0$, which establishes our claim that $\bar\pi=\{\bar0\}$. This completes the proof.
\end{proof}

\begin{prop}\label{prop: ideals in localizations}
 Let $B$ be a band, $S$ a multiplicative subset and $\iota_S:B\to S^{-1}B$ the localization map. Let $I$ be an $m$-ideal of $B$ and $J$ an $m$-ideal of $S^{-1}B$. Then
 \[\textstyle
  S^{-1}I \ = \ \{ \frac as \mid a\in I,\ s\in S\}
 \]
 is an $m$-ideal of $S^{-1}B$,
 \[
  \iota_S^{-1}(S^{-1}I) \ = \ \{ a\in B\mid sa\in I \text{ for some }s\in S\},
 \]
 and $S^{-1}(\iota_S^{-1}(J))=J$. If $I$ and $J$ are $k$-ideals, then so are $S^{-1}I$ and $\iota_S^{-1}(J)$.
\end{prop}

\begin{proof}
Since $\frac{0}{1} \in S^{-1}I$, and for every $\frac{b}{t} \in S^{-1}B$ and $\frac{a}{s} \in S^{-1}I$, $\frac{b}{t} \cdot \frac{a}{s} = \frac{ab}{st} \in S^{-1}I$, we see that $S^{-1}I$ is an $m$-ideal of $S^{-1}B$, which establishes the first claim. 

We have $a \in \iota_S^{-1}(S^{-1}I)$ if and only if $\frac{a}{1} = \frac{b}{t} \in S^{-1}B$ for some $b \in I$ and $t \in S$, if and only if $sa \in I$ for some $s \in S$, which establishes the second claim. 

Let $\frac{a}{s} \in S^{-1}(\iota_S^{-1}(J))$, i.e., $\iota_S(a) = \frac{a}{1} \in J$ and $s \in S$. Then $\frac{a}{s} = \frac{a}{1} \cdot \frac{1}{s} \in J$ so $S^{-1}(\iota_S^{-1}(J)) \subseteq J$. Conversely, if $\frac{a}{s}$ is an element in $J$, then $\frac{a}{1} = \frac{a}{s} \cdot s \in J$, and thus $a \in \iota_S^{-1}(J)$. This shows that $\frac{a}{s} \in S^{-1}(\iota_S^{-1}(J))$, and hence $S^{-1}(\iota_S^{-1}(J)) \supseteq J$. This establishes the third claim.

Finally assume that $I$ and $J$ are $k$-ideals. Then $\iota_S^{-1}(J)$ is a $k$-ideal by \autoref{lemma: inverse images of ideals and k-ideals as kernels}. In order to show that $S^{-1}I$ is a $k$-ideal, consider $x = \frac{a}{s} + \sum \frac{b_i}{t_i} \in N_{S^{-1}B}$, and $\frac{b_i}{t_i} \in S^{-1}I$. Multiplying $x$ by $s \prod t_i$, it follows that $a \cdot \prod t_i + \sum c_i \in N_B$ for some $c_i \in I$. Therefore, $a \cdot \prod t_i \in I$, and hence $\frac{a}{s} \in S^{-1}I$. 
\end{proof}

\subsection{Prime ideals}
\label{subsection: prime ideals}

The notion of primality is the same for $m$-ideal and $k$-ideals; we therefore focus on $m$-ideals in this section. The analogous results for the subclass of $k$-ideals can be derived \textit{a posteriori} using the results from \autoref{subsection: m-ideals and k-ideals}.

\begin{df}
 Let $B$ be a band. A \emph{prime $m$-ideal} is an $m$-ideal $\fp$ of $B$ for which $S=B-\fp$ is a multiplicative subset. 
 
 Given a prime $m$-ideal $\fp$ with complement $S$ in $B$, the \emph{localization of $B$ at $\fp$} is $B_\fp=S^{-1}B$. In this case, we write $\iota_\fp:B\to B_\fp$ for the localization map.
\end{df}

\begin{lemma}\label{lemma: inverse images of prime ideals}
 Let $f:B\to C$ be a band morphism and $\fp$ a prime $m$-ideal of $C$. Then $f^{-1}(\fp)$ is a prime $m$-ideal of $C$.
\end{lemma}

\begin{proof}
 Let $S=C-\fp$. Then $f^{-1}(S)$ is a multiplicative set in $B$ and is equal to $B-f^{-1}(\fp)$, which shows that $f^{-1}(\fp)$ is prime.
\end{proof}

\begin{prop}\label{prop: prime ideals in localizations}
 Let $B$ be a band and $S$ a multiplicative set in $B$. Then
 \[
  \begin{tikzcd}[column sep=60]
   \Big\{\begin{array}{c} \text{prime $m$-ideals $\fp$ of $B$}\\ \text{with $\fp\cap S=\emptyset$}\end{array}\Big\} \ar[r,shift left=3pt,"S^{-1}"] & \big\{\ \text{ prime $m$-ideals of $S^{-1}B$ } \ \big\} \ar[l,shift left=3pt,"\iota_S^{-1}"]
  \end{tikzcd}
 \]
 are mutually inverse bijections, and $\iota_S$ induces a canonical isomorphism $B_\fp\to(S^{-1}B)_{S^{-1}\fp}$ for every prime ideal $\fp$ of $B$ with $\fp\cap S=\emptyset$.
\end{prop}

\begin{proof}
 As a first step, we verify that the images of both maps are well-defined. Given a prime $m$-ideal $\fp$ of $B$ with $S\subset B-\fp$, the complement of $S^{-1}\fp$ in $S^{-1}B$ is the multiplicative set $\{\frac ts\mid t\in B-\fp,\ s\in S\}$, which shows that $S^{-1}\fp$ is prime and that $S^{-1}$ is well-defined. 
 
 The inverse image $\iota_S^{-1}(\fq)$ of a prime $m$-ideal $\fq$ of $S^{-1}B$ is prime by \autoref{lemma: inverse images of prime ideals}. Since $\iota_S(S)\subset(S^{-1}B)^\times$ and $\fq$ does not contain a unit, $\iota_S^{-1}(\fq)\cap S=\emptyset$. Thus $\iota_S^{-1}$ is well-defined.
 
 By \autoref{prop: ideals in localizations}, $S^{-1}(\iota_S^{-1})(\fq)=\fq$ for every prime $m$-ideal $\fq$ of $S^{-1}B$. Conversely, let $\fp$ be a prime $m$-ideal of $B$ with $\fp\cap S=\emptyset$. By \autoref{prop: ideals in localizations}, $\iota_S^{-1}(S^{-1}\fp)=\{a\in B\mid sa\in\fp\text{ for some }s\in S\}$. Since $\fp$ is prime and $\fp\cap S=\emptyset$, we have $a\in\fp$ if and only if $as\in\fp$. Thus $\iota_S^{-1}(S^{-1}\fp)=\fp$.
 
 We turn to the last claim. Let $\fp$ be a prime $m$-ideal of $B$ with $\fp\cap S=\emptyset$. Let $\fq=S^{-1}\fp$ and $T=B-\fp$. Then the identity map $\id:B\to B$ composed with $\iota_S:B\to S^{-1}B$ and $\iota_\fq:S^{-1}B\to (S^{-1}B)_\fq$ sends $T$ to $(S^{-1}B)_\fq^\times$, which induces a morphism $\iota:B_\fp\to (S^{-1}B)_\fq$ by \autoref{prop: universal property of the localization}. Conversely, the identity map $\id:B\to B$ composed with $\iota_\fp:B\to B_\fp$ sends $T=\iota_S^{-1}(S^{-1}T)$ to $B_\fp^\times$, which induces a morphism $(S^{-1}B)_\fq\to B_\fp$. Since both morphisms are induced by $\id:B\to B$ and since localizations are epimorphisms, these two morphisms are mutually inverse isomorphisms.
\end{proof}

\subsection{Maximal ideals}
\label{subsection: maximal ideals}

Since the notion of prime ideals does not make reference to axiom \ref{I3}, a $k$-ideal is prime if and only if it is prime as an $m$-ideal. However, the situation is different for maximal ideals. Every band $B$ has a unique maximal proper $m$-ideal, namely $\fm=B-B^\times$, which is a prime $m$-ideal since $B^\times$ is a multiplicative set. The $m$-ideal $\fm$ is in general not a $k$-ideal, however.

\begin{df}
 Let $B$ be a band. A \emph{maximal $k$-ideal of $B$} is a maximal proper $k$-ideal $\fm$ of $B$ with respect to inclusion.
\end{df}

\begin{prop}\label{proposition: maximal k-ideals are prime}
Let $B$ be a band, $S$ a multiplicative subset, and $I$ a $k$-ideal of $B$ such that $I\cap S=\emptyset$. Consider the collection $\cS$ of all $k$-ideals $J$ that contain $I$ and have empty intersection with $S$, ordered by inclusion. Then $\cS$ contains a maximal ideal and every maximal ideal in $\cS$ is prime. In particular, every maximal $k$-ideal of $B$ is prime.
\end{prop}

\begin{proof}
 Consider a chain $\{J_i\}$ of ideals in $\cS$. Then $\bigcup J_i$ is a $k$-ideal (which can be verified by fixing a suitably large $i$) that contains $I$ (which is contained in each $J_i$) with empty intersection with $S$ (since each $J_i$ has empty intersection with $S$). Thus $\bigcup J_i$ is an upper bound for $\{J_i\}$. This allows us to apply Zorn's lemma to deduce the existence of a maximal $k$-ideal in $\cS$, which establishes the first claim.
 
 In order to establish the second claim, consider a maximal $k$-ideal $\fm$ in $\cS$ and the localization $\iota_S:B\to S^{-1}B$. Let $\fn$ be a maximal $k$-ideal of $S^{-1}B$ that contains $S^{-1}\fm$. As a proper ideal, $\fn$ does not contain elements of the form $\frac s1$ with $s\in S$. Thus $\iota_S^{-1}(\fn)$ has empty intersection with $S$ and thus belongs to $\cS$. Since $\fm\subset\iota_S^{-1}(S^{-1}\fm)\subset\iota_S^{-1}(\fn)$ and $\fm$ is maximal in $\cS$, we conclude that $\fm=\iota_S^{-1}(\fn)$. If we can show that $\fn$ is prime, then it follows from \autoref{lemma: inverse images of prime ideals} that $\fm$ is prime.
 
 This reduces the proof to the case where $S\subset B^\times$ and $\fm$ is a maximal $k$-ideal. Consider $ab\in\fm$ with $a\notin\fm$ and let $T=\fm\cup\{a\}$. We prove by induction on $n>0$ that if $s\in\gen{T}_k^{(n)}$, then $sb\in\fm$. 
 
 For $n=0$, we have $s\in T=\fm\cup\{a\}$. Since $ab\in\fm$, we conclude that $bs\in\fm$. If $n>0$, then $N_B$ contains a term of the form $s-\sum d_jt_j$ with $t_j\in \gen{T}_k^{(n-1)}$. By the inductive hypothesis, we have $bt_j\in\fm$. Applying \ref{I3} to $sb-\sum d_jt_jb\in N_B$ shows that $sb\in\fm$.
 
 Since $\fm$ is maximal, $\gen{T}_k=B$ and therefore $1\in\gen{T}_k^{(n)}$ for some $n$. This implies that $b=1\cdot b\in\fm$, as desired.
\end{proof}

\subsection{Radical ideals}
\label{subsection: radical ideals}

\begin{df}
 Let $B$ be a band and $I$ an $m$-ideal of $B$. The \emph{radical of $I$} is 
 \[
  \sqrt{I} \ = \ \{a \in B \mid a^n \in I\text{ for some $n \geq1$}\}. 
 \]
 A \emph{radical $m$-ideal of $B$} is an $m$-ideal $I$ with $\sqrt{I}=I$.
\end{df}

Note that $I\subset\sqrt{I}$ by the definition of the radical. 
In the statement of the next result, we define the empty intersection of prime ideals to be the unit ideal.

\begin{prop}\label{prop: radical ideal}
 Let $B$ be a band and $I$ an $m$-ideal of $B$. Then $\sqrt{I}$ is a radical $m$-ideal equal to 
 \[
  \sqrt{I} \ = \ \bigcap_{\substack{\text{prime $m$-ideals $\fp$}\\ \text{that contain $I$}}} \fp.
 \]
 If $I$ is a $k$-ideal, then $\sqrt{I}$ is a $k$-ideal equal to 
 \[
  \sqrt{I} \ = \ \bigcap_{\substack{\text{prime $k$-ideals $\fp$}\\ \text{that contain $I$}}} \fp.
 \]
\end{prop}

\begin{proof}
 We begin by establishing that $\sqrt{I}=\bigcap\fp$. Given $a\in\sqrt{I}$ and a prime $m$-ideal $\fp$ that contains $I$, we have that $a^n\in I\subset \fp$ for some $n\geq1$, and thus $f\in\fp$ since $\fp$ is prime. Therefore $\sqrt{I}\subset\bigcap\fp$. Conversely, let $a\in B-\sqrt{I}$ and define $S=\{a^i\mid i\in\N\}$. Then $I\cap S=\emptyset$. Let $\iota_S:B\to S^{-1}B$ be the localization map. Since $a^ib\in I$ implies that $b\notin S$, $\iota_S^{-1}(S^{-1}I)=\{b\in B\mid a^ib\in I\text{ for some }i\geq0\}$ (cf.\ \autoref{prop: ideals in localizations}) also has empty intersection with $S$. Thus $I\subset\iota_S^{-1}(S^{-1}I)\subset\fp$ for the inverse image $\fp=\iota_S^{-1}(\fm)$ of the maximal $m$-ideal $\fm$ of $S^{-1}B$, which is a prime $m$-ideal of $B$ that contains $I$, but not $a$. Thus $\bigcap\fp\subset\sqrt{I}$, which establishes our first claim.
 
 As an intersection of $m$-ideals, $\sqrt{I}$ is thus an $m$-ideal. Let $\fp$ be a prime $m$-ideal that contains $I$ and let $a\in\sqrt{I}$. Then $a^i\in I\subset\fp$ for some $i\geq1$ and thus $a\in\fp$ since $\fp$ is prime. This shows that $\sqrt{I}\subset\fp$. 
We conclude that the radical of $\sqrt{I}$ is equal to $\sqrt{I}=\bigcap \fp$, which shows that $\sqrt{I}$ is a radical $m$-ideal.
 
 Assume that $I$ is a $k$-ideal. Then $\sqrt{I}$ is contained in the intersection $J=\bigcap\fp$ of all prime $k$-ideals $\fp$ that contain $I$. Consider $a\in B-\sqrt{I}$ and $S=\{a^i\mid i\in \N\}$. Then $I\cap S=\emptyset$. By \autoref{proposition: maximal k-ideals are prime}, $I$ is contained in a prime $k$-ideal $\fp$ that intersects $S$ trivially. Thus $a\notin J$, which shows that $\sqrt{I}=J$, which is a $k$-ideal since it is an intersection of $k$-ideals.
\end{proof}

\begin{cor}\label{cor: prime ideals are radical}
 Every prime $m$-ideal is radical.
\end{cor}

\begin{proof}
 Let $\fp$ be a prime $m$-ideal of a band $B$. Then $\fp$ is tautologically the unique minimal prime $m$-ideal $\fq$ of $B$ that contains itself. Thus $\sqrt{\fp}=\bigcap\fq=\fp$.
\end{proof}

\subsection{Limits and colimits}
The category of bands is complete and cocomplete, similar to other algebraic categories, such as (semi)rings, (pointed) monoids and (ordered) blueprints. In fact, limits and colimits commute with the forgetful functor $\Bands\to\Mon_0$, which allows us to compute a limit (or colimit) first as a pointed monoid and then endow this pointed monoid with the weakest (or strongest) possible null set. In this section, we describe the construction of products, equalizers and tensor products. 

\subsubsection{Initial and terminal objects}
\label{sec:initial-terminal} 

The \emph{trivial band} is the pointed monoid $\0=\{0\}$ with $1=0$ together with the null set $N_{\0}=\0^+=\{0\}$.

\begin{lemma}\label{lemma: initial and terminal object}
 The regular partial field $\Funpm$ is an initial object in $\Bands$ and $\0$ is a terminal object in $\Bands$.
\end{lemma}

\begin{proof}
 Let $B$ be a band. The requirement that a band morphism preserves $0$, $1$ and $-1$ singles out a unique candidate for a band morphism $f:\Funpm\to B$, which is evidently a morphism of pointed monoids. It sends the generating relation $1-1$ of $N_\Funpm$ to the corresponding element $1-1\in N_B$, thus it is indeed a band morphism by \autoref{lemma: morphisms can be tested on generators}. This shows that $\Funpm$ is initial in $\Bands$. 
 
 Similarly, there is a unique map $B\to\0$, which is evidently a morphism of pointed monoids. Since $N_{\0}=\0^+$, it is a band morphism. This shows that $\0$ is terminal in $\Bands$.
\end{proof}

\subsubsection{Products}
Consider a family $\{B_i\}_{i\in I}$ of bands. The \emph{product of $\{B_i\}$} is defined as the Cartesian product $\prod B_i$ of the underlying monoids, which is a pointed monoid with respect to the coordinatewise multiplication. Its null set is 
\[\textstyle
 N_{\prod B_i} \ = \ \{ \sum_j (a_{ji})_{i\in I} \in (\prod B_i)^+ \mid \sum a_{ji}\in N_{B_i}\text{ for every }i\in I\}.
\]
The coordinate projection $\pr_j:\prod B_i\to B_j$ is a band morphism for every $j\in I$. Note that the product of the empty family is $\0$.

\begin{prop}\label{prop: universal property of the product}
 The band $\prod B_i$ together with the projections $\pr_j$ is the product of the $B_i$ in $\Bands$.
\end{prop}

\begin{proof}
 Given a family of band morphisms $f_i:C\to B_i$, the map $f:C\to\prod B_i$ with $f(a)=(f_i(a))_{i\in I}$ is the unique map such that $f_j=\pr_j\circ f$. Since the multiplication in $\prod B_i$ is defined componentwise, $f:C\to \prod B_i$ is a monoid morphism. Let 
 $\sum a_j\in N_C$ and define $a_{ji}=f_i(a_j)$. Then $\sum a_{ji}\in N_{B_i}$ for every $i$ and thus $\sum f(a_j)=\sum (a_{ji})_{i\in I}\in N_{\prod B_i}$. This shows that $f:C\to\prod B_i$ is a band morphism and satisfies the universal property of the product.
\end{proof}

\subsubsection{Equalizers}
Let $f:B\to C$ and $g:B\to C$ be band morphisms . The \emph{equalizer of $f$ and $g$} is the pointed submonoid $\eq(f,g)=\{a\in B\mid f(a)=g(a)\}$ of $B$ together with the null set
\[\textstyle
 N_{\eq(f,g)} \ = \ \{ \sum a_i\in\eq(f,g)^+\mid \sum a_i\in N_B\}.
\]
The inclusion $\iota:\eq(f,g)\to B$ is a band morphism.

\begin{prop}\label{prop: universal property of the equalizer}
 The band $\eq(f,g)$ together with $\iota:\eq(f,g)\to B$ is the equalizer of $f$ and $g$ in $\Bands$.
\end{prop}

\begin{proof}
Let $h:D\to B$ be a morphism such that $f\circ h=g\circ h$. By definition, the band $\eq(f,g)$ contains the set-theoretic image of $h$, i.e.,\ $h$ factors uniquely through a map $h':D\to \eq(g,h)$. Since $\eq(f,g)$ inherits the band structure from $B$, the map $h'$ is evidently a band morphism. This establishes the universal property of the equalizer of $f$ and $g$.
\end{proof}

\subsubsection{Tensor products}
Consider a band $k$ and a non-empty family of $k$-algebras $\{B_i\}_{i\in I}$ with structure maps $\alpha_{B_i}:k\to B_i$. The \emph{tensor product of $\{B_i\}_{i \in I}$ over $k$} is the pointed monoid
\[\textstyle
 \bigotimes_k B_i \ = \ \{ a\in\prod B_i\mid a_i=1\text{ for all but finitely many }i\in I\} \ / \ \sim \; ,
\]
where $\sim$ is the equivalence generated by the relations of the type $a\sim b$ for which there are $i,j\in I$ and a $c\in k$ such that 
\[
 a_i \ = \ \alpha_{B_i}(c)\cdot b_i \qquad \text{and} \qquad \alpha_{B_j}(c)\cdot a_j \ = \ b_j
\]
and $a_k=b_k$ for all $k\neq i,j$. We write $\otimes a_i$ for the class of $a=(a_i)$ in $\bigotimes_kB_i$. 

If $I = \emptyset$, we define the (empty) tensor product to be $k$ itself.

The tensor product comes equipped with canonical maps $\iota_j:B_j\to\bigotimes_kB_i$ with $\iota_j(a)=\otimes b_i$, where $b_j=a$ and $b_k=1$ for $k\neq j$. The pointed monoid $\bigotimes_kB_i$ comes equipped with the null set
\[\textstyle
 N_{\bigotimes_kB_i} \ = \ \gen{\sum \iota_j(a_k) \mid  j\in I\text{ and }\sum a_k\in N_{B_j}}_{\bigotimes_k B_i^+},
\]
which turns $\bigotimes_kB_i$ into a band and $\iota_j:B_j\to\bigotimes_kB_i$ into a band morphism. 

\begin{prop}\label{prop: universal property of the tensor product}
 The tensor product $\bigotimes_kB_i$, together with the canonical maps $\iota_j:B_j\to\bigotimes_kB_i$, is the coproduct of $\{B_i\}$ in $\Alg_k$.
\end{prop}

\begin{proof}
 Given a family of $k$-linear band morphisms $f_i:B_i\to C$, the requirements $f_j=f\circ\iota_j$ single out a unique candidate for $f:\bigotimes_k B_i\to C$, defined by $f(\otimes a_i)=\prod f_i(a_i)$. The (possibly infinite) product $\prod f_i(a_i)$ has to be interpreted as a finite product over the finitely many nonzero coefficients $a_i$ of the tensor $\otimes a_i$. In order to verify that $f$ is well-defined on equivalence classes, consider $a_i=\alpha_{B_i}(c)b_i$ and $\alpha_{B_j}(c)a_j=b_j$. Then we have 
 \[
  f_i(a_i)\cdot f_j(a_j) \ = \ \alpha_C(c)f_i(b_i)f_j(a_j) \ = \ f_i(b_i)f_j(b_j),
 \]
 where we use $f_i\circ\alpha_{B_i}=\alpha_C=f_j\circ\alpha_{B_j}$, which shows that $f$ is well-defined as a map. Since the multiplication in $C$ is commutative and the $f_i$ are multiplicative, we have $f(\otimes a_i)\cdot f(\otimes b_i)=\big(\prod f_i(a_i)\big)\cdot\big(\prod f_i(b_i)\big)=\prod f_i(a_ib_i)=f(\otimes a_i\cdot\otimes b_i)$, 
 which shows that the map $f$ is a morphism of monoids. Consider a generator $\sum\iota_j(a_k)$ of $N_{\bigotimes_kB_i}$ with $\sum a_k\in N_{B_j}$. Then $\sum f(\iota_j(a_k))=\sum f_j(a_k)\in N_C$ since $f_j:B_j\to C$ is a band morphism. By \autoref{lemma: morphisms can be tested on generators}, $f:\bigotimes_kB_i\to C$ is a band morphism, which establishes the universal property of the tensor product.
\end{proof}

As instances of particular interest, we find the following cases: 
\begin{enumerate}
 \item Given two band morphisms $B\to C$ and $B\to D$ (i.e.,\ $C$ and $D$ are $B$-algebras), the tensor product $C\otimes_BD$ is the \emph{pushout} of $C\leftarrow B\to D$ in $\Bands$.
 \item Given two morphisms $f:B\to C$ and $g:B\to C$, the tensor product $\coeq(f,g)=C\otimes_BC$ is the \emph{coequalizer} of $f$ and $g$ in $\Bands$.
 \item Given a family of bands $\{B_i\}$, which are $\Funpm$-algebras with respect to the unique morphism $\Funpm\to B_i$, the tensor product $\bigotimes B_i=\bigotimes_{\Funpm}B_i$ is the \emph{coproduct} of the $B_i$ in $\Bands$.
\end{enumerate}

\begin{cor}\label{cor: Bands is complete and cocomplete}
 The category $\Bands$ is complete and cocomplete.
\end{cor}

\begin{proof}
 Since $\Bands$ has all products and equalizers, it is complete. Since it has all coproducts and coequalizers, it is cocomplete.
\end{proof}


\section{Band schemes}
\label{section: Band schemes}

In this section, we introduce band schemes and discuss some of their first properties. We provide an ample class of examples and compare band schemes to usual schemes, monoid schemes, and ordered blue schemes.

\subsection{Covering families}
\label{subsection: covering families}

We begin with a discussion of the right notion of spectrum, which dictates all further steps in this theory. Since bands come with various natural choices of ideals (such as $m$-ideals, $k$-ideals, null ideals) it is not so clear {\em a priori} which type of prime ideal leads to the most useful notion of the spectrum. This can be answered through topos-theoretic considerations, which we have attempted to present in as elementary way as possible below; it turns out that prime $m$-ideals provide the correct notion.

The reader who is interested in seeing an explicit definition of band schemes at once is invited to jump ahead to \autoref{subsection: the prime spectrum}.

There are different approaches to generalized scheme theory in the literature, and we can follow each of these different roads to develop a theory of band schemes, making additional choices in each approach. 
For example, we can consider a band scheme as:
\begin{itemize}
 \item a topological space with a structure sheaf in $\Bands$ that is locally isomorphic to the spectrum of a band (which relies on a suitable notion of spectrum);
 \item a sheaf on the site of affine band schemes (which relies on a suitable Grothendieck pretopology);
 \item a relative scheme in the sense of To\"en and Vaqui\'e (which relies on a suitable theory of band modules).
\end{itemize}
In the end, all three approaches result in the same theory in terms of natural equivalences of categories. We forgo a discussion of the third, more technical, viewpoint, and concentrate instead on the interplay between the first two perspectives. For more details on the interplay between these different approaches to a generalized scheme theory (in the case of semiring schemes and blue schemes), see \cite{Lorscheid17}.

As a starting point, we postulate an anti-equivalence $\Spec:\Bands\to\BAff$ of the category of bands with the category of affine band schemes, which determines $\BAff$ as the opposite category of $\Bands$ up to canonical equivalence. 

If we attempt to realize affine band schemes as topological spaces together with a structure sheaf in $\Bands$, then the topology of an affine band scheme $X=\Spec B$ is completely determined by the requirement that it is generated by the \emph{principal open subschemes} $U_h=\Spec B[h^{-1}]$ (with respect to the morphism $\iota_h^\ast:U_h\to X$ induced by the localization map $\iota_h:B\to B [h^{-1}]$), together with some requirements on the families $\{\varphi_i:U_i\to X\}_{i\in I}$ of open subschemes that cover $X$ topologically, which are natural from the perspective of topos theory. In detail, these are:
\begin{enumerate}[label={\rm(A\arabic*)}]
 \item \label{A1} The topology of $X$ is \emph{generated by principal opens}: for every $h\in B$, the map $\iota_h^\ast:U_h\to X$ is an open topological embedding, and every open affine subscheme $V$ of $X$ admits a covering family $\{\varphi_i:U_{h_i}\to V\}_{i\in I}$ by principal open subschemes $U_{h_i}=\Spec B[h_i^{-1}]$ of $X$, i.e.,\ $\iota_{h_i}^\ast$ is equal to the composition $U_{h_i}\to V\hookrightarrow X$.
 \item \label{A2} The topology is \emph{functorial in $X$}: morphisms of affine band schemes are continuous and given a morphism $\psi: X'\to X$ of affine band schemes and a covering family $\{U_i\to X\}$, the family $\{U_i\times_XX'\to X'\}$ is a covering family of $X'$. 
 \item \label{A3} Every affine band scheme $Y$ \emph{defines a sheaf} $\Hom(-,Y)$ on $\BAff$: given a covering family $\{\varphi_i:U_i\to X\}_{i\in I}$, the canonical map
 \[
  \begin{tikzcd}[column sep=25pt]
   \Hom(X,Y) \ar[r] 
   & \displaystyle \lim\bigg( \prod_{i\in I} \ \Hom(U_i,\ Y) \ar[r,shift left=4pt,"\res_1"] \ar[r,shift right=4pt,"\res_2"'] 
   & \displaystyle \prod_{k,l\in I} \ \Hom(U_k\times_X U_l,\ Y) \bigg) 
  \end{tikzcd}
 \]
 is a bijection. 
 \item \label{A4} Every affine band scheme $X$ is \emph{sober}: every irreducible closed subset $Z$ of $X$ is the topological closure of a unique point $z\in Z$.
\end{enumerate}
Note that property \ref{A2} guarantees that the collection of all covering families form a Grothendieck pretopology $\cT$ on $\BAff$. Property \ref{A3} means that $\cT$ is subcanonical, i.e.,\ the Yoneda functor $\BAff\to\Sh(\BAff,\cT)$ is fully faithful. Property \ref{A4} ensures that the underlying topological space of an affine band scheme is determined by the collection of open subsets.

Properties \ref{A1}--\ref{A3} determine $\cT$ uniquely as follows. We call any morphism $U\to X$ that appears in a covering family an \emph{open immersion}.

\begin{prop}\label{prop: covering families of affine band schemes}
 Assume that the collection $\cT$ of all topological covering families of affine band schemes satisfies \ref{A1}--\ref{A3}. Then every affine open subscheme is principal open and a family $\{\varphi_i:U_i\to X\}_{i\in I}$ of open immersions is in $\cT$ if and only if $\varphi_i$ is an isomorphism for some $i\in I$.
\end{prop}

\begin{proof}
 We begin with the proof of the second claim. By \ref{A2}, $\cT$ is a Grothendieck pretopology, so any isomorphism $U_i\to X$ covers $X$. Since covering families are defined in terms of topological covering conditions, adding open subsets does not change the covering property and thus $\{U_i\to X\}$ is in $\cT$ if $U_i\to X$ is an isomorphism for some $i\in I$. 
 
 Conversely, let $\{\varphi_i:U_i\to X\}_{i\in I}$ be a family in $\cT$. By \ref{A1}, we can assume that $U_i=\Spec B[h_i^{-1}]$ for all $i\in I$ where $h_i\in B$. Consider
 \[
  C \ = \ \bandgenquot{B[a_i\mid i\in I]}{a_ih_j-a_jh_i\mid i,j\in I}
 \]
 and let $\varphi:X'=\Spec C\to \Spec B=X$ be the morphism induced by the inclusion $f:B\to C$. By \ref{A2}, the family $\{U_i'\to X'\}$ is in $\cT$, where $U_i'=U_i\times_XX'=\Spec C[h_i^{-1}]$.
 
 Note that for any affine band scheme $\Spec D$, we have 
 \[
  \Hom(\Spec D,\ \Spec\Funpm[T]) \ = \ \Hom(\Funpm[T],\ D) \ = \ D
 \]
 as sets. Since $a_kh_l=a_lh_k$ in $C$, we have $\frac{a_k}{h_k}=\frac{a_l}{h_l}$ in $U_k'\times_{X'}U_l'=\Spec C[(h_kh_l)^{-1}]$. Thus the tuple $c=\big(\frac{a_i}{h_i}\big)_{i\in I}$ is an element of 
 \[
  \begin{tikzcd}[column sep=25pt]
   \displaystyle \lim\bigg( \prod_{i\in I} \ \Hom(U_i',\ \Spec\Funpm[T]) \ar[r,shift left=4pt,"\res_1"] \ar[r,shift right=4pt,"\res_2"'] 
   & \displaystyle \prod_{k,l\in I} \ \Hom(U'_k\times_{X'} U'_l,\ \Spec\Funpm[T]) \bigg),
  \end{tikzcd}
 \]
 which is naturally a subset of $\prod C[h_i^{-1}]$. By \ref{A3}, the element $c$ is in the image of $C\to \prod C[h_i^{-1}]$. This is the case if and only if $h_i\in C^\times=B^\times$ for some $i\in I$. Thus $U_i\to X$ is an isomorphism, which establishes the second claim of the proposition.
 
 We continue with the first claim. Let $V$ be an open affine subscheme of $X=\Spec B$. By \ref{A1}, there exists a family of the form $\{U_{h_i}\to V\}$ in $\cT$ with $U_{h_i}=\Spec B[h_i^{-1}]$. By what we have proven, $U_{h_i}\to V$ is an isomorphism for some $i$ and thus $V=\Spec B[h_i^{-1}]$ as subschemes of $X$, which concludes the proof.
\end{proof}

If we assume \ref{A4} in addition, then we can deduce a complete description of the underlying topological space of an affine band scheme as follows. By \ref{A1}, principal opens $U_h=\Spec B[h^{-1}]$ are open subsets of $X=\Spec B$, i.e.,\ the morphism $\iota_h^\ast:U_h\to X$ induced by the localization map $\iota_h:B\to B[h^{-1}]$ is an open topological embedding. This allows us to define a map
 \[
  \Phi_X: \ X \ \longrightarrow \ \big\{\text{prime $m$-ideals of $B$}\big\}
 \]
 as follows: given a point $x\in X$, we define $B_x=S_x^{-1}B$ for $S_x=\{h\in B\mid x\in\Spec B[h^{-1}]\}$ and denote the localization map by $\iota_x:B\to B_x$. Let $\fm_x=B_x-B_x^\times$ be the maximal $m$-ideal of $B_x$ and define $\Phi_X(x)=\iota_x^{-1}(\fm_x)$, which is a prime $m$-ideal of $B$.
 
\begin{thm}\label{thm: points of affine band schemes}
 If \ref{A1}--\ref{A4} hold, then the map $\Phi_X$ is a bijection. For $h\in B$, the set $\Phi_X(U_h)$ contains all prime $m$-ideals $\fp$ of $B$ with $h\notin \fp$.
\end{thm}

\begin{proof}
 In order to construct an inverse bijection to $\Phi_X$, we first prove that every non-empty affine band scheme $X=\Spec B$ has a unique closed point. 
 The points of $X$ are partially ordered by $x\leq y$ whenever $y\in\overline{\{x\}}$. Consider a chain $\{x_i\}_{i\in I}$ of points $x_i$ of $X$. We claim that $\bigcap_{i\in I}\overline{\{x_i\}}$ is non-empty. For if not, then the family $\{X-\overline{\{x_i\}}\}$ is an open covering of $X$ by proper open subsets, which violates \autoref{prop: covering families of affine band schemes}. We conclude that $\bigcap_{i\in I}\overline{\{x_i\}}$ contains a point $x$, which is an upper bound for the chain $\{x_i\}$. 
 
 This allows us to apply Zorn's lemma, which provides us with a maximal point $x\in X$, which means that $x$ is closed. Since $X$ is sober, $x$ is the unique closed point of $X$.
 
 We turn to the construction of the inverse of $\Phi_X$. Consider a prime $m$-ideal $\fp$ of $B$ and let $\iota_\fp:B\to B_\fp$ be the localization map. Then $U_\fp=\Spec B_\fp$ has a unique closed point $x_\fp$. We define $\Psi_X(\fp)$ as the image $\iota_\fp^\ast(x_\fp)$ of $x_\fp$ in $X$.

 Next we show that $\Psi_X$ is right inverse to $\Phi_X$. Let $\fp$ be a prime $m$-ideal of $B$ and let $x=\Psi_X(\fp)$. Then $x=\iota_\fp^\ast(x_\fp)$ for the unique closed point $x_\fp$ of $U_\fp=\Spec B_\fp$. Moreover, $x\in U_h$ for $h\in B$ if and only if $\iota_\fp:B\to B_\fp$ factors through $B[h^{-1}]$, which is the case if and only if $\iota_\fp(h)\in B_\fp^\times$. Thus
 \[ 
  S_x \ = \ \{h\in B\mid x\in U_h\} \ = \ \iota_\fp^{-1}(B_\fp^\times) \ = \ B-\fp,
 \] 
 which shows that $B_\fp=S_x^{-1}B$ and $\iota_x=\iota_\fp$. We conclude that $\Phi_X(x)=\iota_x^{-1}(\fm_x)=\fp$, as desired.
 
 We continue with the proof that $\Psi_X$ is left inverse to $\Phi_X$. Consider $x\in X$ and $S_x=\{h\in B\mid x\in U_h\}$. Let $\iota_x:B\to B_x=S_x^{-1}B$ be the localization map, $\fm_x=B_x-B_x^\times$ the maximal $m$-ideal of $B_x$, and $\fp=\Phi_X(x)=\iota_x^{-1}(\fm_x)$. Let $x_\fp$ be the unique closed point of $U_\fp=\Spec B_\fp$.
 
 Since $B-\fp=\iota_x^{-1}(B_x^\times)=S_x$, we have $B_\fp=B_x$ and $\iota_\fp=\iota_x$. Thus $\Psi_X(\fp)=\iota_\fp^\ast(x_\fp)=\iota_x^\ast(x_\fp)$ and $U_\fp=\Spec S_x^{-1}B$. Since
 \[
  \Spec S_x^{-1}B = \Spec \big( \underset{h\in S_x}\colim\{B\to B[h^{-1}]\} \big) = \underset{h\in S_x}\lim \big\{ \Spec B[h^{-1}] \to X\big\} = \bigcap_{h\in S_x} U_h
 \]
 as an intersection inside $X$, the morphism $\iota_x^\ast:U_S=\Spec S^{-1}B\to\Spec B=X$ is an inclusion whose image is $\bigcap_{h\in S_x} U_h$. Thus $\Psi_X(\fp)=\iota_x^\ast(x_\fp)$ is a closed point of $\bigcap_{h\in S_x} U_h$.
 
 The family $\{U_h\mid h\in S_x\}$ is a neighbourhood filter of $x$ by the definition of $S_x$, and it is completely prime thanks to \autoref{prop: covering families of affine band schemes} (for background on filter, cf.\ \cite{Johnstone82}; in particular, see section II.1.3). Since $X$ is sober, this means that $x$ is the unique point for which $\{U_h\mid h\in S_x\}$ is a neighbourhood filter and thus $x$ is the only closed point in $\bigcap_{h\in S_x} U_h$. Therefore we have $\Psi_X(\fp)=x$, as desired.
 
 In order to prove the last claim of the theorem, consider $h\in B$, $x\in X$, and $\fp=\Phi_X(x)$. Then $x\in U_h$ if and only if $h\in S_x=B-\fp$, which means that $h\notin\fp$. This shows that $\Phi_X(U_h)$ is the set of all prime $m$-ideals $\fp$ of $B$ that do not contain $h$, as claimed.
\end{proof}

\subsection{The prime spectrum and band schemes}
\label{subsection: the prime spectrum}

The considerations from \autoref{subsection: covering families} lead us to the following definition of the prime spectrum of a band. 

\subsubsection{The spectrum}

\begin{df}
 Let $B$ be a band. Its \emph{prime spectrum $\Spec B$} is the set of all prime $m$-ideals $\fp$ of $B$ together with the topology generated by the \emph{principal open subsets} 
 \[
  U_h \ = \ \{ \fp\in\Spec B \mid h\notin\fp\}
 \]
 for $h\in B$ and with the structure sheaf $\cO_X$ on $X=\Spec B$ in $\Bands$ that is characterized by $\cO_X(U_h)=B[h^{-1}]$ for $h\in B$. 
\end{df}

\begin{rem}
 Note that the rule $\cO_X(U_h)=B[h^{-1}]$ does indeed describe a unique sheaf on $X$. The sheaf exists since every covering family $\{U_i\to V\}$ of an affine open subset $V$ of $X=\Spec B$ contains an isomorphism by \autoref{prop: covering families of affine band schemes}, so the sheaf axiom is vacuous. And it is uniquely determined by its values on $U_h$, since the family $\{U_h\mid h\in B\}$ is a basis for the topology of $X$.
 
 A more systematic explanation can be derived from property \ref{A3} in \autoref{subsection: covering families}: the structure sheaf $\cO_X$ is nothing other than the restriction of the sheaf $\Hom(-,\ \Spec\Funpm[T])$ to the subsite of $\BAff$ that consists of all open subschemes of $X$.
\end{rem}

\begin{ex}\label{ex: affine band schemes}
 An idyll $F$ has a unique prime $m$-ideal, which is $\gen 0=\{0\}$. Thus $\Spec F=\{\gen 0\}$ is the one-point space. 
 
 The affine $n$-space over an idyll $F$ is $\A^n_F=\Spec F[T_1,\dotsc,T_n]$. The prime $m$-ideals of $F[T_1,\dotsc,T_n]$ are of the form $\fp_I=\gen{T_i\mid i\in I}_m$ for subsets $I$ of $\{1,\dotsc,n\}$. The closed point of $\A^n_F$ is $\gen{T_1,\dotsc,T_n}_m$ and its generic point is $\gen0_m$.
 
 The $n$-dimensional torus over an idyll $F$ is $\G_{m,F}^n=\Spec F[T_1^{\pm1},\dotsc,T_n^{\pm1}]$. Since $F[T_1^{\pm1},\dotsc,T_n^{\pm1}]$ is an idyll, the torus consists of a single point, which is the prime $m$-ideal $\gen0_m$. Note that the more involved structure of $\G_{m,F}^n$, in contrast to $\Spec F$, is visible in the larger variety of closed subschemes, or quotients of $F[T_1^{\pm1},\dotsc,T_n^{\pm1}]$.
\end{ex}

\subsubsection{Band spaces and band schemes}

\begin{df}
 A \emph{band space} is a topological space $X$ together with a sheaf $\cO_X$ in $\Bands$. The \emph{stalk at $x\in X$} is the band
 \[
  \cO_{X,x} \ = \ \underset{x\in U\subset X\text{ open}}{\colim} \cO_X(U).
 \]
 A \emph{morphism of band spaces} is a continuous map $\varphi:X\to Y$ between band spaces together with a sheaf morphism $\varphi^\#:\cO_Y\to\varphi_\ast\cO_X$ that is \emph{local} in the sense that for every $x\in X$ and $y=\varphi(x)$, the induced morphism of stalks $\varphi_x^\#:\cO_{Y,y}\to \cO_{X,x}$ sends non-units to non-units. 
\end{df}

Often we say that $X$ is a band space and that $\varphi:X\to Y$ is a morphism of band spaces where we suppress an explicit mentioning of the structure $\cO_X$ and the sheaf morphism $\varphi^\#:\cO_Y\to\varphi_\ast\cO_X$.

\begin{rem}
 Every band space $X$ is \emph{local} in the sense that $\cO_{X,x}$ has a unique maximal $m$-ideal $\fm_x=\cO_{X,x}-\cO_{X,x}^\times$ for every $x\in X$. The condition that morphisms $\varphi:X\to Y$ of band spaces are local amounts to the condition that $\varphi_x^{\#}(\fm_y)=\fm_x$ for every $x\in X$ and $y=\varphi(x)$.
\end{rem}

The primary example of a band space is the prime spectrum $\Spec B$ of a band $B$. An open subset $U$ of a band space $X$ is naturally a band space with respect to the restriction of the structure sheaf $\cO_X$ to $U$.

\begin{df}
 An \emph{affine band scheme} is a band space that is isomorphic to the spectrum of a band. A \emph{band scheme} is a band space in which every point has an open neighborhood isomorphic to an affine band scheme. A \emph{morphism of band schemes} is a morphism of band spaces. This defines the categories $\BAff$ of affine band schemes and $\BSch$ of band schemes.
\end{df}

\subsubsection{On the local nature of band schemes}
\label{subsubsection: local nature of band schemes}
 
 The fact that every band $B$ has a unique maximal $m$-ideal, namely $\fm_B=B-B^\times$, has strong implications and leads to certain simplifications when compared with usual scheme theory.
 
 To start with, $\fm_B$ is the unique closed point of $\Spec B$. Thus every affine open of a band scheme $X$ has a unique closed point, which means that a closed point $x$ of a band scheme $X$ is contained in a unique affine open $U_x$ of $X$. If the closure of every point of $X$ contains a closed point (which is the case for a band scheme of finite type over an idyll since it has only finitely many points), then $X$ has a \emph{minimal} affine open covering, namely $\cU_{\min}=\{U_x\mid x\in X \text{ closed}\}$. 
 
 If $X$ is affine, then every open covering $\{U_i\}$ of $X$ must contain $U_i=X$ (for some $i$). In consequence, every affine open $U$ of $X$ is a principal open since it is covered by principal opens and thus equal to one of them.
 
 Another consequence is that the image $\varphi(U)$ of an affine open $U$ of $X$ under a morphism $\varphi:X\to Y$ is contained in an affine open $V$ of $Y$. Even better: for any two chosen affine open coverings $X=\bigcup_{i\in I} U_i$ and $Y=\bigcup_{j\in J} V_j$, there is a map $\varphi_0:I\to J$ such that $\varphi(U_i)\subset V_{\varphi_0(i)}$. This spares us the annoyance of having to pass to suitable refinements in many situations.

\subsubsection{Functoriality of \texorpdfstring{$\Spec$}{Spec}} 

The construction of the spectrum is functorial: a band morphism $f:B\to C$ defines a map $\varphi=f^\ast:\Spec C\to\Spec B$ that sends a prime $m$-ideal $\fq$ of $C$ to $\varphi(\fq)=f^{-1}(\fq)$. This map is continuous since 
\[
 \varphi^{-1}(U_h) \ = \ \{\fq\in\Spec C\mid h\notin f^{-1}(\fq)\} \ = \ \{\fq\in\Spec C\mid f(h)\notin \fq\} \ = \ U_{f(h)}.
\]
The morphism $\varphi^\#:\cO_{\Spec B}\to\varphi_\ast\cO_{\Spec C}$ is defined on principal opens $U_h$ of $\Spec B$ as the unique band morphism $\varphi^\#(U_h)$ that makes the diagram
\[
 \begin{tikzcd}[column sep=-5]
  B \ar[d,"\iota_h"'] \ar[rrrrrr,"f"] &&&&&& C \ar[d,"\iota_{f(h)}"] \\
  B[h^{-1}] & = & \cO_{\Spec B}(U_h) \ar[rr,dashed,"\varphi^\#(U_h)"] & \hspace{50pt} & \cO_{\Spec C}(U_{f(h)}) & = & C[f(h)^{-1}]
 \end{tikzcd}
\]
commute, which exists by \autoref{prop: universal property of the localization} since $f(h)$ is invertible in $C[f(h)^{-1}]$. The morphism of sheaves $\varphi^\#$ is local since for every $\fq\in\Spec C$ and $\fp=f^{-1}(\fq)$, the induced morphism of stalks $\varphi_\fq:B_\fp\to C_\fq$ sends the maximal ideal of $B_\fp$ to the maximal ideal of $C_\fq$. Thus $\varphi=f^\ast:\Spec C\to\Spec B$ is a morphism of band spaces. This defines a contravariant functor $\Spec:\Bands\to\BSch$ whose image is contained in $\BAff$.

We write $\Gamma X=\cO_X(X)$ for the band of global sections of $X$. Note that pulling back global sections along morphisms of band schemes turns $\Gamma$ into a functor. 

\begin{thm}\label{thm: adjunction between Spec and Gamma}
 The functor $\Spec:\Bands\to\BSch$ is fully faithful and $\Gamma: \BSch\to\Bands$ is its right adjoint, i.e.,\ pulling back global sections yields a bijection
 \[
  \Phi: \ \Hom(X,\ \Spec B) \ \longrightarrow \ \Hom(B,\ \Gamma X)
 \]
 for every band scheme $X$ and every band $B$ that is functorial in $X$ and $B$. In other words, $\Spec:\Bands\to\BAff$ is an anti-equivalence of categories with inverse $\Gamma$ and $\BAff$ is a reflective subcategory of $\BSch$ with reflection $\Spec\circ\Gamma:\BSch\to\BAff$. 
\end{thm}

\begin{proof}
 As a first step, we show that $\Spec:\Bands\to\BAff$ is an anti-equivalence of categories with inverse $\Gamma$. For $X=\Spec B$, the canonical map $B\to \cO_X(X)=\Gamma X$ is an isomorphism. Given a band morphism $f:B\to C$ with induced morphism $\varphi=f^\ast:\Spec C\to\Spec B$, we have $\Gamma\varphi=f$ by the definition of $\Gamma\varphi=\varphi^\#(\Spec B)$. This shows that $\Spec $ is faithful with retraction $\Gamma$. By definition $\BAff$ is the essential image of $\Spec$. 
 
 We are left with showing that $\Spec$ is full. Let $\varphi:Y\to X$ be a morphism of affine band schemes $Y=\Spec C$ and $X=\Spec B$ and let $f=\Gamma\varphi:B\to C$ be the induced band morphism. Consider $\fq\in Y$ and its image points $\fp=\varphi(\fq)$ and $\tilde\fp=f^\ast(\fp)=f^{-1}(\fp)$ in $X$. Then we get induced maps $\varphi_\fq:B_\fp\to C_\fq$ and $f^\ast_\fq:B_{\tilde\fp}\to C_\fq$ between stalks. Since $\varphi$ is local, $\varphi_\fq^{-1}(\fm_\fq)=\fm_\fp$ and thus $B_{\fp}=S^{-1}_\fq B$ for 
 \[\textstyle
  S_\fq \ = \ \{a\in B\mid \frac{f(a)}1\in C_\fq^\times\} \ = \ \{a\in B\mid f(a)\notin\fq\}. 
 \]
Thus $B_\fp=B_{\tilde\fp}$ and $\fp=\tilde\fp$, which shows that $f^\ast=\varphi$ as maps. Since $B\to B[h^{-1}]$ is an epimorphism of bands, the morphism $\varphi^\#:\cO_X\to \varphi_\ast\cO_Y$ of sheaves is determined by $\Gamma\varphi=f=\Gamma f^\ast$. This shows that $\Spec:\Bands\to\BAff$ is an anti-equivalence of categories with inverse $\Gamma$.

 Let $\{U_i\}$ be an open covering of $X$ by affine band schemes $U_i=\Spec C_i$, which come with restriction maps $\res_{U_i}:\Gamma X\to\Gamma U_i=C_i$. Composing these with a given band morphism $f:B\to \Gamma X$ yields band morphisms $f_i:B\to C_i$, which define morphisms $\varphi_i=f_i^\ast:U_i\to \Spec B$ of band schemes. Since the composition of restriction maps are again restriction maps, the morphisms $\varphi_i$ coincide on the intersections of the $U_i$ and glue to a morphism $\varphi:X\to \Spec B$, whose associated map $\Gamma\varphi:B\to\Gamma X$ is equal to $f$ by construction. This shows that $\Phi$ is surjective. Conversely, a morphism $\varphi:X\to \Spec B$ is determined by its restrictions $\varphi_i:U_i\to\Spec B$, which implies that $\Phi$ is injective.
 
 In the case of an affine scheme $X=\Spec C$, this shows that $\Phi:\Hom(\Spec C,\ \Spec B)\to\Hom(B, \ C)$ is a bijection, which shows that $\Spec:\Bands\to\BAff$ and $\Gamma:\BAff\to\Bands$ are mutually inverse equivalences of categories. Since $\Bands\simeq\BAff$ is a full subcategory, the adjunction $\Phi$ establishes the reflection $\Spec\circ\Gamma:\BSch\to\Aff$. 
\end{proof}

There is yet another way to express the adjunction from \autoref{thm: adjunction between Spec and Gamma}, which we find worthwhile to mention.

\begin{cor}\label{cor: canonical morphism from X to Spec Gamma X}
 Every band scheme $X$ comes with a canonical morphism $\eta_X:X\to\Spec\Gamma X$ which satisfies the property that every morphism from $X$ into an affine band scheme factors uniquely through $\eta_X$.
\end{cor}

\begin{proof}
 The canonical morphism $\eta_X$ is the morphism that corresponds to the identity map $\id:\Gamma X\to\Gamma X$ under the adjunction $\Phi$ from \autoref{thm: adjunction between Spec and Gamma}. As the unit of the adjunction $\Phi$, it satisfies the asserted property of the corollary.
\end{proof}

\subsubsection{The radical ideal--closed subset correspondence}

Let $X=\Spec B$ be an affine band scheme. 
For a subset $S$ of $B$, we define the \emph{vanishing set of $S$} as
\[
 \cV(S) \ = \ \{\fp\in X\mid S\subset \fp\},
\]
which is closed since it is the complement of the open set $\bigcup_{h\in S}U_h$. Given a subset $Z$ of $X$, we define the \emph{vanishing ideal of $Z$} as 
\[
 \cI(Z) \ = \ \bigcap_{\fp\in Z} \ \fp,
\]
which is a radical $m$-ideal of $B$ by \autoref{prop: radical ideal}. Just as in usual scheme theory, we have:

\begin{thm}\label{thm: Nullstellensatz}
 The maps $\cV$ and $\cI$ define mutually inverse inclusion-reversing bijections 
 \[
  \begin{tikzcd}[column sep=60]
   \big\{\text{radical $m$-ideals of $B$}\big\} \ar[r,shift left=3pt,"\cV"] & \big\{\text{closed subsets of $X$}\big\}, \ar[l,shift left=3pt,"\cI"]
  \end{tikzcd}
 \]
 which restrict to bijections between prime $m$-ideals and irreducible closed subsets.
\end{thm}

\begin{proof}
 By \autoref{prop: radical ideal}, we have for every radical $m$-ideal $I$ that
 \[
  \cI(\cV(I)) \ = \ \bigcap_{\substack{\text{prime $m$-ideals $\fp$}\\ \text{ that contain $I$}}} \fp \ = \ I,
 \]
 so we are left with showing that $\cV(\cI(Z))=Z$ for every closed subset $Z$ of $X$. We have $Z\subset\cV(\cI(Z))$ tautologically. Since the vanishing sets of the form $\cV(h)$ generate the topology of closed subsets of $X$, it suffices to show that $Z\subset\cV(h)$ implies $\cV(\cI(Z))\subset\cV(h)$ in order to establish $\cV(\cI(Z))\subset Z$. Assume $Z\subset\cV(h)$, i.e.,\ $h\in\fp$ for all $\fp\in Z$. Then $h\in\cap_{\fp\in Z}\fp=\cI(Z)\subset\fq$ for every $\fq\in\cV(\cI(Z))$, which shows that $\cV(\cI(Z))\subset\cV(h)$, as desired. Thus $\cV(\cI(Z))=Z$, which shows that $\cV$ and $\cI$ are mutually inverse bijections. It is evident that $\cV$ and $\cI$ reverse inclusions, which establishes the first claim of the theorem.
 
 We turn to the second claim. A prime $m$-ideal $\fp$ is contained in every other prime $m$-ideal $\fq\in\cV(\fp)$. Thus whenever $\fp\in V$ for a closed subset $V$ of $\cV(\fp)$, we have $V=\cV(\fp)$, which shows that $\cV(\fp)$ is irreducible. Conversely, let $Z$ be an irreducible closed subset of $X$ and let $I=\cI(Z)=\bigcap_{\fp\in Z}\fp$. Consider $g,h\in B$ such that $g\cdot h\in I$. This means that $V(I)\subset V(g)\cup V(h)$. Since $Z=V(I)$ is irreducible, $V(I)\subset V(g)$ or $V(I)\subset V(h)$, which means that $g\in I$ or $h\in I$. This shows that $I$ is a prime $m$-ideal and establishes the second claim  of the theorem.
\end{proof}

\subsubsection{Properties \ref{A1}--\ref{A4}}

To conclude this section, we check that the theory of band schemes does indeed satisfy properties \ref{A1}--\ref{A4} from \autoref{subsection: covering families}. Note that \autoref{thm: adjunction between Spec and Gamma} assures us that the functor $\Spec:\Bands\to\BAff$ is an anti-equivalence of categories, which is the basic requirement for the formulation of \ref{A1}--\ref{A4}.

Properties \ref{A1} and \ref{A2} follow immediately from the definitions of $\Spec B$ and a morphism of band schemes, respectively. Property \ref{A3} follows directly from the fact that every affine open covering of an affine band scheme $X=\Spec B$ must contain $X$ itself, which contains every other open subset in the covering. Property \ref{A4} follows from \autoref{thm: Nullstellensatz}: every irreducible subset of $\Spec B$ is the closure of a uniquely determined prime $m$-ideal of $B$.

\subsection{Properties of band schemes and their morphisms}
\label{subsection: properties}

We introduce some important concepts for band schemes in this section, such as residue fields, $k$-schemes (of locally finite type), open and closed immersions, separated morphisms, and fibre products.

\subsubsection{Residue fields}

\begin{df}
 Let $X$ be a band scheme and $x\in X$. Let $\cO_{X,x}$ be the stalk at $x$ and $\fm_x=\cO_{X,x}-\cO_{X,x}^\times$ its maximal $m$-ideal. The \emph{residue field at $x$} is the quotient $k(x)=\bandquot{\cO_{X,x}}{\gen{\fm_x}}$. It comes with a canonical morphism $\kappa_x:\Spec k(x)\to X$.
\end{df}

The term ``residue field'' is borrowed from usual algebraic geometry. The residue field of a point in a band scheme is in general not a field, but rather an idyll if not $\0$; cf.\ \autoref{subsection: the kernel space} for a detailed discussion.

\begin{prop}\label{prop: universal property of the residue field}
 Let $X$ be a band scheme, $x\in X$, and $\kappa_x:\Spec k(x)\to X$ the canonical morphism. Then for every idyll $F$, every morphism $\Spec F\to X$ with image $\{x\}$ factors uniquely through $\kappa_x$.
\end{prop}

\begin{proof}
 Since the image of $\Spec F\to X$ is contained in an affine open of $X$, we can assume that $X=\Spec B$ is affine. Thus $\Spec F\to X$ corresponds to a band morphism $f:B\to F$ with kernel $\fp$, where $\fp$ is the prime $m$-ideal of $B$ that corresponds to $x$. Let $S=B-\fp$. Since $F$ is an idyll, $f(S)\subset F^\times$, and thus $f$ factors uniquely through $f_\fp:\cO_{X,x}=B_\fp\to F$ by \autoref{prop: universal property of the localization}. Since $f(\fp)\subset\{0\}$, this map factors uniquely through $\bar f_\fp:k(x)=\bandgenquot{B_\fp}{S^{-1}\fp}\to F$ by \autoref{prop: universal property of the quotient}, as claimed.
\end{proof}

\subsubsection{\texorpdfstring{$k$}{k}-Schemes}

\begin{df}
 Let $k$ be a band. A \emph{(band) $k$-scheme} is a band scheme $X$ together with a morphism $\omega_X:X\to\Spec k$. A ($k$-linear) \emph{morphism of $k$-schemes} is a morphism $\varphi:X\to Y$ of $k$-schemes such that $\omega_X=\omega_Y\circ\varphi$. We denote the category of $k$-schemes by $\BSch_k$.
\end{df}

Note that in the case of a ring $k$, the category $\BSch_k$ of band $k$-schemes differs from the category $\Sch_k$ of (usual) $k$-schemes since not every prime $m$-ideal of a ring is a prime ideal in the usual sense. See \autoref{subsection: Functors into other scheme theories} for the description of a functor $\BSch_k\to\Sch_k$.

\begin{df}
 A $k$-scheme $X$ is \emph{locally of finite type} if for every affine open $U$ of $X$, the band morphism $\omega_U^\#:k\to \Gamma U$ factors through a $k$-linear surjection $k[T_1,\dotsc,T_n]\to \Gamma U$ for some $n$. It is \emph{of finite type} if, in addition, $X$ is quasi-compact.
\end{df}

\subsubsection{Fibre product and disjoint unions}

Let $\psi_X:X\to Z$ and $\psi_Y:Y\to Z$ be morphisms of band schemes. The \emph{fibre product of $X$ and $Y$ over $Z$} is defined as the topological fibre product $X\times_ZY$ together with the structure sheaf that sends an affine open of the form $U\times_WV$ (with $U$, $V$ and $W$ affine) to $\Gamma U\otimes_{\Gamma W}\Gamma V$, together with the obvious restriction maps. The disjoint union $\coprod X_i$ of band schemes $X_i$ comes with the structure sheaf $\cO_{\coprod X_i}(U)=\cO_{X_j}(U)$ for a basic open subset $U\subset X_j$.

The fact that the formation of fibre products commutes with the functor $\BSch\to\Top$ which sends a band scheme $X$ to its underlying topological space is analogous to the corresponding fact for monoid schemes; cf.\ \cite[Prop.~3.1]{CHWW15}. This relationship is clarified in terms of the functor $\BSch\to\MSch$ from band schemes to monoid schemes, which is the identity on underlying topological spaces; cf.\ \autoref{subsubsection: the underlying monoid schemes}.

\begin{thm}\label{thm: limits and colimits for band schemes}
 The category $\BSch$ is closed under finite limits and arbitrary coproducts. In particular, $\Spec\Funpm$ is terminal and $\Spec \0 =\emptyset$ is initial, $X\times_ZY$ is the fibre product of $X$ and $Y$ over $Z$, and the disjoint union $\coprod X_i$ is the coproduct of a family of band schemes $X_i$.
\end{thm}

\begin{proof}
 Since $\Funpm$ is initial in $\Bands$, its spectrum is terminal in $\BAff$. A morphism $\varphi:X\to\Spec\Funpm$ is uniquely determined as a map since $\Spec\Funpm$ contains only one point, and $\varphi^\#$ is uniquely determined in terms of the unique band morphism $\Funpm\to\Gamma X$. Evidently $\Spec\0=\emptyset$ is initial.

 Given a family of morphisms $\varphi_i:X_i\to Z$ of band schemes, there is a unique continuous map $\varphi:\coprod X_i\to Z$ so that the restriction to $X_j$ is $\varphi_j$. The morphism of sheaves is given as the morphism $\varphi_U^\#:\cO_Z(U)\to\prod\cO_{X_i}(\varphi_i^{-1}(U))$ into the product whose $j$-th coordinate projection is $\varphi_{j,U}^\#:\cO_Z(U)\to\cO_{X_j}(\varphi_j^{-1}(U))$. This defines (in a unique way) a morphism $\varphi:\coprod X_i\to Z$ whose restriction to $X_j$ is $\varphi_j$.
  
 Consider arbitrary $X$, $Y$ and $Z$ and a commutative diagram
 \[
  \begin{tikzcd}[column sep=40]
                                                                                 &&                                            & X \ar[dr,"\psi_X"] \\
   T \ar[urrr,bend left=10pt,"\varphi_X"] \ar[drrr,bend right=10pt,"\varphi_Y"'] && X\times_ZY \ar[ur,"\pi_X"'] \ar[dr,"\pi_Y"] &                    & Z \\
                                                                                 &&                                            & Y \ar[ur,"\psi_Y"'] 
  \end{tikzcd}
 \]
 of band schemes. Since a morphism $\varphi:T\to X\times_ZY$ is uniquely determined by its restriction to the affine opens of $T$ and since every compatible choice of morphisms from the affine opens of $T$ into $X\times_ZY$ glues to a such a morphism, we can assume that $T$ is affine. Since the image of an affine band scheme is contained in an affine open (cf.\ \autoref{subsubsection: local nature of band schemes}), we can assume that $X$, $Y$ and $Z$ are affine. Then $X\times_ZY=\Spec(\Gamma X\otimes_{\Gamma Z}\Gamma_Y)$. Since $\Spec:\Bands\to\BAff$ is an anti-equivalence by \autoref{thm: adjunction between Spec and Gamma}, $X\times_ZY$ is the fibre product of $X$ and $Y$ over $Z$ in $\BAff$, which establishes the unique existence of a morphism $\varphi:T\to X\times_ZY$ that makes the above diagram commute.
\end{proof}

\subsubsection{Open and closed immersions}
\label{subsection: open and closed immersions}

\begin{df}
 Let $X$ be a band scheme. An \emph{open subscheme of $X$} is an open subset $U$ of $X$ together with the restriction $\cO_U=\cO_X\vert_U$ of the structure sheaf of $X$ to $U$. An \emph{open immersion} is an morphism $\iota: Y\to X$ that restricts to an isomorphism of $Y$ with an open subscheme of $X$.
\end{df}

We call an (affine) open subscheme of $X$ also an \emph{(affine) open} for short.

\begin{df}
  A morphism $\varphi:Y\to X$ of band schemes if \emph{affine} if for every affine open subscheme $U$ of $X$, the inverse image $\varphi^{-1}(U)=U\times_XY$ is an affine open subscheme of $Y$.
  A \emph{closed immersion} is an affine morphism $\iota:Y\to X$ such that $\iota^\#(U):\cO_X(U)\to\cO_Y(V)$ is surjective for all affine opens $U$ of $X$ and $V=\iota^{-1}(U)$. A \emph{closed subscheme of $X$} is an isomorphism class of closed immersions into $X$.
\end{df}
 
\begin{rem}
 Similar to the case of monoid schemes (cf.\ \cite{Lorscheid-Ray23}), a closed immersion $\iota:Y\to X$ of band schemes is a topological embedding that is in general not closed; e.g.\ the image of the diagonal embedding $\A^1\to\A^2$ is not closed. In the case of monoid schemes, which behave similarly to band schemes in many respects, the failure of their underlying topological spaces to reflect topological properties has been rectified using {\em congruence spaces} in \cite{Lorscheid-Ray23}. We refer to \autoref{subsection: the null space} for a potential generalization to band schemes in terms of null spaces.
\end{rem}

\subsubsection{Separated morphisms}
\label{subsection: separated morphisms}

The existence of fibre products allows us to define separated morphisms in terms of the \emph{diagonal} $\Delta: X\to X\times_YX$ of a morphism $\varphi:X\to Y$, which is the morphism induced by the identity maps $\id:X\to X$ to each copy of $X$ in $X\times_YX$. 

\begin{df}
 A morphism $\varphi:X\to Y$ of band schemes is \emph{separated} if its {diagonal} $\Delta:X\to X\times_YX$ is a closed immersion. A $k$-scheme $X$ is \emph{separated} if the structure map $X\to\Spec k$ is separated.
\end{df}

\begin{lemma}\label{lemma: affine morphisms are separated}
 Every affine morphism of band schemes is separated.
\end{lemma}

\begin{proof}
Let $\varphi:X\to Y$ be an affine morphism with diagonal $\Delta:X\to X\times_YX$. By definition of the fibre product, $X\times_YX$ is covered by affine opens of the form $U_1\times_VU_2$ where $U_1$ and $U_2$ are affine opens of $X$ and $V$ is an affine open of $Y$ that contains $\varphi(U_1)$ and $\varphi(U_2)$. Since $\varphi$ is affine, $U=\varphi^{-1}(V)$ is affine. Since $U$ contains $U_1$ and $U_2$, the fibre product $X\times_YX$ is, in fact, covered by affine opens of the form $U\times_VU$ with $U=\varphi^{-1}(V)$. The inverse image $\Delta^{-1}(U\times_VU)=U$ is affine and the induced band morphism $\Gamma U\otimes_{\Gamma V}\Gamma U\to\Gamma U$ is evidently surjective. Thus $\Delta$ is a closed immersion and $\varphi$ is separated, as claimed.
\end{proof}

\begin{lemma}\label{lemma: intersections affines in a separated band scheme are affine}
 Let $k$ be a band and $X$ a separated $k$-scheme. Then the intersection of affine opens of $X$ is affine.
\end{lemma}

\begin{proof}
 Let $U_1$ and $U_2$ be affine open in $X$. Then $U_1\times_k U_2$ is affine open in $X\times_kX$. Since $X\to\Spec k$ is separated, $\Delta:X\to X\times_kX$ is a closed immersion, and therefore $U_1\cap U_2=\Delta^{-1}(U_1\times_k U_2)$ is affine, as claimed.
\end{proof}

\begin{ex}
 An example of a non-separated $k$-scheme is the affine line with double point, which is the band scheme $X$ that is covered by two copies $U_1$ and $U_2$ of $\A^1_k=\Spec k[T]$ that intersect in $U_0=U_1\cap U_2\simeq \G_{m,k}=\Spec k[T^{\pm1}]$, together with the restriction maps $\res_{U_i,U_0}: k[T]\to k[T^{\pm1}]$ that send $T$ to $T$ for both $i=1$ and $i=2$. 
 
 In order to see that $X\to\Spec k$ is not separated, consider the diagonal $\Delta:X\to X\times_kX$ and the affine open $U_1\times_k U_2$ of $X\times_k X$, whose coordinate band is $k[T_1]\otimes_kk[T_2]=k[T_1,T_2]$. Then $\Delta^{-1}(U_1\times_kU_2)=U_1\cap U_2=U_0$, and the induced band morphism $k[T_1,T_2]\to k[T^{\pm1}]$ maps both $T_1\otimes 1$ and $1\otimes T_2$ to $T$. Thus $T^{-1}$ is not in its image, which shows that $\Delta$ is not a closed immersion and that $X$ is not separated.
\end{ex}

\subsection{Examples}
\label{subsection: examples}

\subsubsection{Projective space}\label{subsubsection: projective space}
 In the following, we define projective $n$ space over an idyll $F$. The points of projective $n$-space
 over $F$ correspond bijectively to ordered $n+1$-tuples $(x_0,\dotsc,x_n)$ with $x_i\in\{0,1\}$ not all equal to $0$. 
Indeed, such tuples correspond bijectively to homogeneous prime $m$-ideals of $F[T_0,\dotsc,T_n]$ that do not contain all of $T_0,\dotsc,T_n$. The generic point of $\P_F^n$ is $(1,\dotsc,1)$, and $\P^n_F$ has $n+1$ closed points $(1,0,\dotsc,0),\dotsc,(0,\dotsc,0,1)$. 
  
The topology of $\P_F^n$ is generated by open subsets of the form
\[
 U_J \ = \ \big\{ (x_0,\dotsc,x_n) \,\big| \, x_j=1\text{ for }j\in J \big\}
\]
for subsets $J$ of $\{0,\dotsc,n\}$. The value of the structure sheaf $\cO_{\P^n_F}$ in $U_J$ is the degree $0$ part of the $F$-algebra $F[T_i,T_j^{\pm1}\mid j\in J, i\notin J]$. Note that the underlying topological space of $\P^n_F$ is independent of the idyll $F$, in contrast to usual scheme theory.

\subsubsection{Toric band schemes}\label{ex: toric band schemes}
 Projective spaces are specific cases of \emph{toric band schemes}, which are defined as follows. Let $k$ be a band and $\Delta$ a (rational polyhedral) fan in $\R^n$. For a cone $\sigma$ in $\Delta$, we denote by $A_\sigma=\sigma^\vee\cap\Z^n$ the lattice points of the dual cone $\sigma$, considered as a multiplicative monoid. Let $k[A_\sigma]$ be the monoid band of $A_\sigma$ over $k$ (cf.\ \autoref{subsubsection: monoid algebras}), and define $U_\sigma=\Spec k[A_\sigma]$. The \emph{toric $k$-scheme with fan $\Delta$} is the union $X(\Delta)=\colim U_\sigma$ of all the spaces $U_\sigma$ with $\sigma\in\Delta$, and its structure sheaf $\cO_{X(\Delta)}$ takes the value $k[A_\sigma]$ on $U_\sigma$.
 
 If $k$ is an idyll, then the points of $X(\Delta)$ correspond bijectively to the cones $\sigma\in\Delta$, with points of the open affine subset $U_\sigma$ of $X(\Delta)$ corresponding to the faces of $\sigma$. 
 
 We mention without proof that $X(\Delta)$ is a separated $k$-scheme and that its base extension $X(\Delta)^+_K$ to a field $K$ (as defined in \autoref{prop: base extension of k-schemes to usual schemes}) is the toric $K$-variety of $\Delta$ in the usual sense.

\subsubsection{Projective band schemes}
 A \emph{projective $B$-scheme} is a $B$-scheme that is isomorphic to a closed subscheme of a projective space $\P^n_B$. Closed subschemes of projective space can be characterized in terms of homogeneous null ideals of $B[T_0,\dotsc,T_n]^+$, which are null ideals of $B[T_0,\dotsc,T_n]^+$ that are generated by homogeneous elements (i.e., $B$-linear combinations of monomials in $T_0,\dotsc, T_n$ of the same degree).
  
 An affine morphism $\varphi:Y\to \P^n_B$ is a closed immersion if there is a homogeneous null ideal $I$ of $B[T_0,\dotsc,T_n]$ such that on the affine open subset $U_J$ of $\P^n$ (as defined in \autoref{subsubsection: projective space}), the map between coordinate bands $\cO_{\P^n_B}(U_J)\to \cO_Y\big(\varphi^{-1}(U_J)\big)$ corresponds to the degree $0$ part of the quotient map 
 \[
  S_J^{-1}B[T_0,\dotsc,T_n] \ \onto \ \overline{S}_J^{-1}\bandquot{ B[T_0,\dotsc,T_n]}{I},
 \]
 where $S_J$ is the multiplicative subset of $B[T_0,\dotsc,T_n]$ generated by $\{T_i\mid i\in J\}$ and $\overline{S}_J$ is the image of $S_J$ in $\bandquot{ B[T_0,\dotsc,T_n]}{I}$. We also write $Y=\Proj \big(\bandquot{B[T_0,\dotsc,T_n]}{I}\big)$ for the closed subscheme $Y$ of $\P^n$, in analogy to usual projective schemes.

\begin{ex}\label{ex: Grassmannians}
 As an example of particular interest, we describe the Grassmannian $\Gr(r,n)$: it is the closed subscheme 
 \[\textstyle
  \Gr(r,n) \ = \ \Proj \big(\bandquot{B[T_I\mid I\in\binom Er]}{J}\big)
 \]
 of $\P^{\binom nr-1}=\Proj \big(B[T_I\mid I\in\binom Er])$, where $E=\{1,\dotsc,n\}$ and $\binom Er$ is the collection of $r$-subsets of $E$, whose homogeneous ideal $J$ is generated by the Pl\"ucker relations
 \[
  \sum_{k=0}^r (-1)^{\epsilon(k,I,J)} T_{I-i_k} T_{Ji_k}
 \]
 for all subsets $I=\{i_0,\dotsc,i_r\}$ and $J=\{j_2,\dotsc, j_r\}$ of $E$. In this formula, we use the shorthand notations $I-i_k=I_k-\{i_k\}$ and $Ji_k=J\cup\{i_k\}$, the convention that $T_{I'}=0$ if $\# I'<r$, and the function $\epsilon(k,I,J) \ = \ \#\{e\in I\cup J\mid e<i_k\}$.
 
If $F$ is an idyll, the set $\Gr(r,n)(F) := \Hom(\Spec(F), \Gr(r,n))$ coincides with the set of strong $F$-matroids of rank $r$ on $\{1,\ldots,n \}$ in the sense of \cite{Baker-Bowler19}, cf.~\cite{Baker-Lorscheid21b}.
In particular, when $F=\K$ is the Krasner hyperfield, $\Gr(r,n)(\K)$ coincides with the set of matroids (in the usual sense) of rank $r$ on $\{1,\ldots,n \}$. 

This characterization of $F$-matroids generalizes to flag $F$-matroids: they correspond to the $F$-rational points of a suitably defined flag variety $\Fl(\br,n)$ over $\Funpm$; cf.\ \cite{Jarra-Lorscheid24} for details.
\end{ex}
 
 \subsubsection{Subvarieties of toric varieties}
 Let $K$ be a field, and let $X(\Delta)$ be the toric $K$-variety associated to the fan $\Delta$, as above.
To every $K$-subscheme $Y$ of $X(\Delta)$, we can associate in a natural way a band scheme $Y(\Delta)$ over $K$ as follows.

First, consider the case where $X(\Delta) = U_\sigma =\Spec K[A_\sigma]$ is the affine open subset associated to a cone $\sigma$ of $\Delta$.
Write $Y \cap U_\sigma = \Spec R$, and let $\pi : K[A_\sigma] \to R$ be the corresponding morphism of $K$-algebras.
We define $B_\sigma$ to be the band whose underlying monoid is $\pi(A_\sigma)$ and whose null set is $N_R \cap \pi(A_\sigma)$.

For the general case, we glue together the various $\Spec B_\sigma$ to get a band scheme $Y(\Delta)$ over $K$ with the property that $Y(\Delta) \cap U_\sigma = \Spec B_\sigma$ for every cone $\sigma$ in $\Delta$.

In  \autoref{ex: tropicalization}, we explain how to see that if $v : K \to \T$ is a valuation on $K$, the set of $\T$-points of the band scheme $Y(\Delta)_\T := Y(\Delta) \times_{\Spec K} \Spec \T$ coincides with the Kajiwara--Payne tropicalization of $Y$.
In this way, band schemes provide a scheme-theoretic enhancement of the usual notion of tropicalization.

\subsection{Functors into other scheme theories}
\label{subsection: Functors into other scheme theories}

Band schemes map to several other types of schemes in a functorial way. We begin with a description of the general principle that enables us to establish such functors; for more details, cf.\ \cite[Section 1]{Lorscheid17}.

We start with a functor $\cG:\Bands\to\cC$ from bands into some other category $\cC$ of algebraic objects, e.g.\ rings, monoids, or ordered blueprints, as considered below. By passing to the opposite category, this yields a functor $\cG:\BAff\to\Aff_\cC$ from the category $\BAff$ of affine band schemes to the category of affine schemes for $\cC$. Applying $\cG$ to every object and every morphism of a diagram $\cU$ of affine band schemes yields a diagram $\cG(\cU)$ in $\Aff_\cC$.

\begin{df}
 Let $X$ be a band scheme. An \emph{affine presentation of $X$} is an isomorphism $\colim\cU\simeq X$, where $\cU$ is a commutative diagram of affine band schemes and open immersions such that for every $U\in\cU$, the canonical inclusion $\iota_U:U\to\colim\cU\simeq X$ is an open immersion. We also say for short that $X=\colim\cU$ is an affine presentation.
\end{df}

The notion of $\cC$-schemes relies on the choice of a Grothendieck pretopology for $\Aff_\cC$, which we assume as given. The following result \cite[Lemma 1.3]{Lorscheid17} extends $\cG$ to a functor $\cG:\BSch\to\Sch_\cC$ from band schemes to $\cC$-schemes. 

\begin{lemma}\label{lemma: extension of a functor to schemes} 
 Assume that $\cG:\BAff\to\Aff_\cC$ commutes with fibre products and preserves covering families. Then $\cG$ extends uniquely to a functor $\cG:\BSch\to \Sch_{\cC}$ with $\cG(X)=\colim\cG(\cU)$ for every affine presentation $X=\colim\cU$.
\end{lemma}

\begin{rem}\label{rem: reduced hypotheses for extensions of functors}
 Since every open covering $X=\bigcup U_i$ of an affine band scheme $X$ is trivial, in the sense that one of the $U_i$ equals $X$, and since the Grothendieck pretopologies in all of our examples are defined in terms of finite localizations, it suffices to verify that $\cG:\Bands\to\cC$ preserves tensor products $C\otimes_BD$ and finite localizations $B\to B[h^{-1}]$ in order for \autoref{lemma: extension of a functor to schemes} to hold.

 Another example of particular interest is $\Aff_\cC=\Sch_\cC=\Top$, which comes equipped with the Grothendieck pretopology of open topological coverings. In this case, the hypotheses of \autoref{lemma: extension of a functor to schemes} turn into the assumptions that $\cG:\BSch\to\Top$ commutes with fibre products and turns open immersions into open topological embeddings.
\end{rem}

\subsubsection{Base extension to usual schemes}

A band $B$ comes with the universal ring
\[
 B^+_\Z \ = \ \Z[B]/\gen{N_B}
\]
(cf.\ \autoref{subsection: universal ring}). We show in the following that this construction is functorial and extends to a \emph{base extension functor}
\[
 (-)^+_\Z: \ \BSch \ \longrightarrow \ \Sch
\]
from the category $\BSch$ of band schemes into the category $\Sch$ of usual schemes. To begin with, the functoriality of $(-)^+_\Z$ follows from its universal property: 

\begin{lemma}\label{prop: universal property of the base extension to rings}
 The map $\rho_B:B\to B^+_\Z$ with $\rho(a)=[a]$ is a band morphism and it is initial for all band morphisms from $B$ into a ring, i.e.,\ the canonical map
 \[
  \rho_B^\ast: \ \Hom_{\Rings}(B^+_\Z,\ R) \ \longrightarrow \ \Hom_{\Bands}(B, \ R)
 \]
 is a bijection for every ring $R$. 
\end{lemma}

\begin{proof}
 The map $\rho_B$ is multiplicative by the definition of $B^+_\Z$. Given $\sum a_i\in N_B$, we have $\sum \rho_B(a_i)=0$ in $B^+_\Z$ and thus $\sum \rho_B(a_i)\in N_{B^+_\Z}$ by the definition of $B^+_\Z$ as a band. Thus $\rho_B$ is a band morphism.
 
 The map $\rho_B^\ast$ is injective since $\rho_B(B)$ generates $B^+_\Z$ as a ring, so any ring homomorphism $f:B^+_\Z\to R$ is determined uniquely by its restriction to $B$. In order to show surjectivity, let $f:B\to R$ be a band morphism. It extends to $\hat f:\Z[B]\to R$ by the universal property of the monoid algebra $\Z[B]$. Since $f(0)=0$, $\hat f$ factors through $\bar f:\Z[B]/\gen 0\to R$. Consider an element $\sum\rho_B(a_i)$ of $\gen{N_B}$ with $\sum a_i\in N_B$. Then $\sum f(a_i)\in N_R$, i.e.,\ $\bar f\big(\sum \rho_B(a_i)\big)=\sum f(a_i)=0$ in $R$. Thus $\bar f$ factors through $f^+_\Z:B^+_\Z\to R$, as claimed.
\end{proof}

\begin{prop}\label{prop: base extension to schemes}
 The base extension $(-)^+_\Z:\Bands\to\Rings$ extends uniquely to a functor $(-)^+_\Z:\BSch\to\Sch$ such that $X^+_\Z=\colim\cU^+_\Z$ for every affine presentation $X=\colim\cU$ of a band scheme $X$.
\end{prop}

\begin{proof}
 By \autoref{prop: universal property of the base extension to rings}, $(-)^+_\Z:\Bands\to\Rings$ is right adjoint to the embedding $\Rings\to\Bands$ and therefore preserves cofibre products. A localization $B\to B[h^{-1}]$ of bands is sent to the localization $B^+_\Z\to B^+_\Z[h^{-1}]$ of rings. Thus the result follows from \autoref{rem: reduced hypotheses for extensions of functors}.
\end{proof}

These results have a relative version, which we mention without proof. 

\begin{prop}\label{prop: base extension of k-schemes to usual schemes}
 Let $k$ be a band with universal ring $K=k^+_\Z$ and $B$ a $k$-algebra with $B^+_K=B^+_\Z\otimes_\Z K$. Then the canonical map
 \[
  \Hom_K(B^+_K, \ R) \ \longrightarrow \ \Hom_k(B,\ R)
 \]
 is a bijection for every ring $R$ with an $K$-algebra structure. The functor $(-)^+_K:\Alg_k\to\Alg_K^+$ extends uniquely to a functor $(-)^+_K:\BSch_k\to\Sch_K$ such that $X^+_K=\colim\cU^+_K$ for every $k$-linear affine presentation $X=\colim\cU$ of a $k$-scheme $X$.
\end{prop}

\subsubsection{The underlying monoid scheme}
\label{subsubsection: the underlying monoid schemes}

The following result extends the forgetful functor $\cF:\Bands\to\Mon_0$ from bands to pointed monoids to a functor $\cF:\BSch\to\MSch$ from band schemes to monoid schemes.

\begin{prop}\label{prop: functor to the underlying monoid scheme}
 The functor $\cF:\Bands\to\Mon_0$ extends uniquely to a functor $\cF:\BSch\to\MSch$ such that $\cF(X)=\colim\cF(\cU)$ for every affine presentation $X=\colim\cU$ of a band scheme $X$.
\end{prop}

\begin{proof}
 The underling monoid of the tensor product $C\otimes_BD$ of bands is the tensor product of the respective underlying monoids by construction of the tensor product. The same holds for finite localizations. Thus the claim follows from \autoref{rem: reduced hypotheses for extensions of functors}.
\end{proof}

\subsubsection{Band schemes as ordered blue schemes}

The embedding $(-)^\oblpr:\Bands\to\OBlpr$ from \autoref{subsubsection: bands to oblpr} extends to an embedding $(-)^\oblpr:\BSch\to\OBSch$ from band schemes to ordered blue schemes due to the following result (cf.\ \cite{Lorscheid23} for a definition of ordered blue schemes):

\begin{prop}\label{prop: extension of Bands to OBlpr to schemes}
 The functor $(-)^\oblpr:\Bands\to\OBlpr$ extends uniquely to an fully faithful embedding of categories $(-)^\oblpr:\BSch\to\OBSch$ such that $\iota(X)=\colim\iota(\cU)$ for every affine presentation $X=\colim\cU$ of a band scheme $X$.
\end{prop}

\begin{proof}
 Since $\Bands$ is coreflective in $\OBlpr$ by \autoref{subsubsection: bands to oblpr}, the embedding $(-)^\oblpr:\Bands\to\OBlpr$ preserves cofibre products. Finite localizations are preserved essentially by definition. Thus $(-)^\oblpr$ extends to a functor $(-)^\oblpr:\BSch\to\OBSch$ by \autoref{rem: reduced hypotheses for extensions of functors}. Since morphisms of both band schemes and ordered blue schemes are determined by their restrictions to affine opens, the full faithfulness of $(-)^\oblpr:\Bands\to\OBlpr$ implies that also $(-)^\oblpr:\BSch\to\OBSch$ is fully faithful.
\end{proof}

\section{Visualizations}
\label{section: visualizations}

Besides its underlying topological space, a band scheme comes with several further subordinate topological spaces, which we call \emph{visualizations}\footnote{The term \textit{visualization} is borrowed from the introduction of \cite{Fujiwara-Kato18}.} because they exhibit certain properties of band schemes. More precisely, we describe functors $\cV:\BSch\to\Top$ from the category $\BSch$ of band schemes into the category $\Top$ of topological spaces together with a natural transformation $\cV\to\utop$ to the forgetful functor $\utop:\BSch\to\Top$ that sends a band scheme to its underlying topological space.

Examples of such visualizations of a band scheme $X$ are rational point sets $X(B)$, which come equipped with both a weak and a strong Zariski topology, which we denote by $X(B)^\weak$ and $X(B)^\str$, respectively. Furthermore, a topology on $B$ induces the fine topology on $X(B)$, which we indicate by $X(B)^\fine$. Other examples of visualizations are the kernel space $X^\kernel$ and the null space $X^\nul$ of $X$.

Under suitable conditions ($B=F$ is a topological idyll and $N_F$ is a prime ideal), these visualizations are connected by a commutative diagram of continuous maps
\[
 \begin{tikzcd}[column sep=40]
  X(F)^\fine \ar[r] & X(F)^\str \ar[r] \ar[d] & X^\nul \ar[->>,d] \\ 
                          & X(F)^\weak \ar[r] & X^\kernel \ar[right hook->,r] & \underline{X}
 \end{tikzcd}
\]
where $\underline X$ is the underlying topological space of $X$; cf.\ \autoref{thm: comparison of visualizations}.

\subsection{The Zariski topology}
\label{section: the Zariski topology}

In classical algebraic geometry, rational point sets $X(B)$ come equipped with a Zariski topology. Since there are more subsets supporting closed subschemes than complements of open subsets in the generality of band schemes, rational point sets $X(B)$ come in fact with two Zariski topologies, a weak and a strong version.

Since it is often required in applications, we fix a base band $k$ and a $k$-algebra $B$. As usual in algebraic geometry, we suppress $k$ from the notation and denote by $X(B)=\Hom_k(\Spec B,X)$ the set of \emph{$k$-linear} $B$-points of a $k$-scheme $X$.
 Corrected.
Since localizations and quotient maps of bands are epimorphisms of bands, open and closed immersions of band schemes are monomorphisms, in perfect analogy to usual scheme theory. As a consequence, every open (or closed) immersion $\iota:Y\to X$ induces an inclusion $\iota_B:Y(B)\to X(B)$ and thus the $B$-points $Y(B)$ of an open (or closed) subscheme $Y$ of $X$ form a subset of $X(B)$.

The \emph{weak Zariski topology for $X(B)$} is the topology of open subsets of the form $Y(B)$ where $Y$ is an open subscheme of $X$. The \emph{strong Zariski topology for $X(B)$} is the topology of closed subsets of the form $Y(B)$ where $Y$ is a closed subscheme of $X$. 

\begin{lemma}\label{lemma: Zariksi topology}
 Both the weak and the strong Zariski topology are indeed topologies.
\end{lemma}

\begin{proof}
 By the local nature of affine band schemes (cf.\ \autoref{subsubsection: local nature of band schemes}), $X(B)=\bigcup U_i(B)$ for any chosen open covering $X=\bigcup U_i$ of $X$. Since also open and closed subsets are local in nature, we can assume that $X=\Spec C$ is affine.
 
 We begin to verify the topology axioms for the weak Zariski topology, which is generated by open subsets of the form $U_h(B)$ for principal open subschemes $U_h=\{\fp\in\Spec C\mid h\notin\fp\}$ of $X=\Spec C$ with $h\in C$. Thus $\emptyset=U_0(B)$ and $X(B)=U_1(B)$ are weak Zariski opens of $X$. By definition, arbitrary unions $\bigcup U_{h_i}(B)=\big(\bigcup U_{h_i} \big)(B)$ of principal opens are open, as is the intersection $U_g(B)\cap U_h(B)=U_{gh}(B)$. This shows that the weak Zariski topology is indeed a topology.
 
 We continue with the strong Zariski topology. A closed subscheme $Z$ of $X=\Spec C$ is of the form $Z=\Spec (\bandquot CI)$ for some null ideal $I$ of $C$. In particular, we have that $X(B)=\big(\Spec (\bandquot C{N_C})\big)(B)$ and $\emptyset=\big(\Spec (\bandquot C{C^+})\big)(B)$ are strong Zariski closed. The union of $\big(\Spec (\bandquot CI)\big)(B)$ and $\big(\Spec (\bandquot CJ)\big)(B)$ is the strong Zariski closed subset $\big(\Spec (\bandquot C{(I\cdot J)})\big)(B)$, where $I\cdot J=\gen{a\cdot b\mid a\in I,\ b\in J}$. The intersection of a family of strong Zariski closed subsets $\big\{\big(\Spec (\bandquot C{J_i})\big)(B))\big\}$ is the strong Zariski closed subset $\big(\Spec (\bandquot CI))\big)(B)$ for the ideal $I=\gen{J_i}$ generated by all $J_i$. This shows that the strong Zariski topology is indeed a topology. 
\end{proof}

For an open subscheme $U$ of $X$ and $h\in \Gamma U^+$, we define 
\[\textstyle
 U(B)_h \ = \ \big\{ \alpha:\Spec B\to U \, \big| \, \Gamma\alpha^+(h)\notin N_B \big\},
\]
where $\Gamma\alpha^+(h)=\sum\Gamma\alpha(a_i)$ if $h=\sum a_i$. 

\begin{prop}\label{prop: Zariski topology}
 Assume that $B$ is an idyll. Then the weak Zariski topology is generated by subsets of the form $U(B)_h$ where $U$ is an affine open subscheme $U$ of $X$ and $h\in \Gamma U$, and the strong Zariski topology is generated by (open) subsets of the form $U(B)_h$ where $U$ is an affine open subscheme $U$ of $X$ and $h\in \Gamma U^+$. In particular, every weak Zariski open is strong Zariski open.
\end{prop}

\begin{proof}
 We begin with the claim regarding the weak Zariski topology. Since the weak Zariski topology is defined in terms of open subschemes, we can assume thanks to \autoref{subsubsection: local nature of band schemes} that $X=\Spec C$ is affine. Consider a subset $U(B)_h$ with $U$ affine open in $X$, which is a principal open by \autoref{subsubsection: local nature of band schemes}, i.e.,\ of the form $U=U_g$ for some $g\in C$. For $h\in\Gamma U=C[g^{-1}]$, the condition $\Gamma\alpha(h)\notin N_B$ means that $\Gamma\alpha(h)\neq0$. Since $B$ is an idyll, this is equivalent to $\Gamma\alpha(h)\in B^\times$, i.e.,\ the image of $\alpha$ is contained in $U_h$. Thus
 \[
  U_g(B)_h \ = \ \{\alpha:\Spec B\to U_g\mid \im\alpha\subset U_h\} \ = \ \{\alpha:\Spec B\to U_{gh}\} \ = \ U_{gh}(B)
 \]
 is an open subset of $X(B)$. For $g=1$, this yields $U_h(B)=U(B)_h$. Since the principal opens $U_h$ of $X$ generate the topology of $X$, this shows that the collection of opens $U(B)_h$ (for $h\in C$) generates the weak Zariski topology of $X(B)$.
 
 We continue with the claim for the strong Zariski topology. As a first step, we note that since $B$ is the disjoint union of its unit group $B^\times$ and $0$, a morphism $\Spec B\to\Spec C$ factors through either $U_h$ or $Z_h=\Spec(\bandgenquot Ch)$, and not through both, for every fixed $h\in C$. By definition of the strong Zariski topology, $Z_h(B)$ is strong Zariski closed in $(\Spec C)(B)$, and thus its complement $U(B)_h=U_h(B)$ is strong Zariski open.
 
 Let $X$ be an arbitrary band scheme and $W$ an open subscheme. Then for every affine open $U$ of $X$ with $C=\Gamma U$, the intersection $W\cap U$ is covered by the principal opens $U_{h}$ of $U$, where $h$ ranges over the null ideal $I=\{h\in C\mid U_h\subset W\}$. By \autoref{subsubsection: local nature of band schemes}, $W\cap U(B)=\bigcup_{h\in I}U_h(B)$, and thus
 \[
  (W\cap U)(B) \ = \ \bigcup_{h\in I} U_{h}(B) \ = \ \bigcup_{h\in I} \big(U(B)-Z_h(B)\big) \ = \ U(B) - \bigcap_{h\in I} Z_h(B)
 \]
 is the complement of the $B$-points of the closed subscheme $Z_I=\Spec(\bandquot BI)$ of $U$. Thanks to their intrinsic definition, the closed subschemes $Z_I$ of the various affine opens $U$ of $X$ glue together to a closed subscheme $Z$ of $X$. Our local verification that $W\cap U(B)$ is the complement of $Z_I(B)$ in $U(B)$ shows that $W(B)$ is the complement of $Z(B)$ in $X(B)$, and as such is open in the strong Zariski topology.
 
Consequently, the claim for the strong Zariski topology follows if we can establish it for affine open subschemes $U$ of $X$. Let $C=\Gamma U$ and $h\in C^+$. Then  
 \[
  U(B)-U(B)_h \ = \ \{\alpha:\Spec B\to X\mid \Gamma\alpha^+(h)\in N_B\} \ = \ \{f:C\to B\mid f^+(h)\in N_B\}
 \]
 consists of all band morphisms $f:C\to B$ that factor through the quotient map $C\to\bandquot{C}{\gen h}$. Thus $U(B)-U(B)_h=Z_h(B)$
 for the closed subscheme $Z_h=\Spec(\bandquot{C}{\gen h})$ of $U$, which shows that $U(B)_h$ is strong Zariski open as the complement of the strong Zariski closed subset $Z_h(B)$. Conversely, let $Z_I=\Spec(\bandquot CI)$ be a closed subscheme of $U$ for some null ideal $I$ of $C$. Then $Z_I(B)=\bigcap_{h\in I}Z_h(B)$, which shows that the open complements $U(B)_h$ of the closed subsets $Z_h(B)$ (for varying $h\in C$) generate the strong Zariski topology of $U(B)$, as claimed.
\end{proof}

\begin{cor}\label{cor: weak=strong Zariski topology for fields}
 Let $B$ be a field. Then $X^+_\Z(B)=X(B)$ and the strong Zariski topology on $X(B)$ agrees with the usual Zariski topology on $X^+_\Z(B)$. If in addition $X$ is covered by the spectra of rings, then the weak Zariski topology agrees with the other two topologies.
\end{cor}

\begin{proof}
 The equality $X^+_\Z(B)=X(B)$ follows from \autoref{prop: base extension of k-schemes to usual schemes}. For the comparison of the different Zariski topologies, we can assume that $X=\Spec C$ is affine. Given a null ideal $I$ of $C$, we define $I^+_\Z$ as the kernel of the quotient map $C^+_\Z\to(\bandquot CI)^+_\Z$. We conclude from \autoref{prop: base extension of k-schemes to usual schemes} that every band morphism $\bandquot CI$ factors uniquely through $C^+_\Z/I^+_\Z$. Thus every strong Zariski closed subset (defined by $I$) is Zariski closed in the usual sense (defined by $I^+_\Z$). Let $\rho_C:C\to C^+_\Z$ be the canonical map. Since $J=\rho^{-1}(J)^+_\Z$ for every ideal $J$ of $C^+_\Z$, we conclude that every usual Zariski closed subset is strong Zariski closed.
 
 For the second assertion, assume that $X$ is covered by the spectra of rings. Since every weak Zariski open is strong Zariski open by \autoref{prop: Zariski topology}, all we need to show is that for an affine open $U$ of $X$ and $h\in\Gamma U^+$, the set $U(B)_h$ is a weak Zariski open.
 
 Let $C=\Gamma U$ and $h=\sum a_i\in C^+$. Then $U(B)_h$ consists of all $k$-morphisms $\alpha:B\to U$ such that $\sum\Gamma\alpha(a_i)\notin N_B$, which means that the sum $\tilde h=\sum\Gamma\alpha(a_i)$ in $B$ is not $0$. If $\bar h=\sum a_i$ is the corresponding sum in $C$ (which is a ring by assumption), then $U(B)_h=U(B)_{\bar h}$, which is a weak Zariski open, as desired.
\end{proof}

\begin{rem}\label{rem: Zariski topology for affine line}
 The weak and the strong Zariski topologies are in general not comparable, as the following example shows. It also suggests that the strong Zariski topology is more useful than the weak Zariski topology.
 
 Consider $\A_k^1=\Spec k[T]$. Then $\A^1_k(B)=\Hom_k(k[T],B)$ is in bijective correspondence with $B$, by identifying a morphism $f:k[T]\to B$ with $f(T)\in B$. For $a\in B$, we write $f_a:k[T]\to B$ for the unique morphism with $f_a(T)=a$.
 
 The points of $\A_k^1$ are of the form $\fp_0=\gen{\fp}$ and $\fp_T=\gen{\fp,T}$, where $\fp$ is a prime $m$-ideal of $k$. The principal opens of $\A^1_k$ are of the form
 \[
  U_h \ = \ \{\fp_0,\ \fp_T \mid h\notin\fp\} \qquad \text{and} \qquad U_{hT} \ = \ \{\fp_0\mid h\notin\fp\}
 \]
 for $h\in k$. Let $\bar h$ be the image of $h$ in the $k$-algebra $B$. Note that $f_a:k[T]\to B$ factors through $k[T,c^{-1}]$ for some $c\in k[T]$ if and only if $f_a(c)\in B^\times$. Thus 
 \[
  U_h(B) \ = \ \{f_a:k[T]\to B\mid f_a(h)\in B^\times\} \simeq \{a\in B\mid \bar h\in B^\times\}
 \]
 is either $\emptyset$ (if $\bar h\notin B^\times$) or $B$ (if $\bar h\in B^\times$). Similarly,
 \[
  U_{hT}(B) \ = \ \{f_a:k[T]\to B\mid f_a(hT)\in B^\times\} \simeq \{a\in B\mid \bar h\cdot a\in B^\times\}
 \]
 is either $\emptyset$ (if $\bar h\notin B^\times$) or $B^\times$ (if $\bar h\in B^\times$). Consequently, $\A_k^1(B)=B$ has three weak Zariski open sets, namely $\emptyset$, $B^\times$, and $B$.

 The strong Zariski topology for $\A_k^1(B)$ is in general more interesting. In particular, we find the following strong Zariski closed subsets. Let us denote by $\bar c$ the image of $c\in k$ in $B$. Since $a-b\in N_B$ if and only if $a=b$, the singletons
 \[
  Z_{T-c} \ = \ \{f_a:k[T]\to B\mid f_a(T)-f_a(c)\in N_B\} \ = \ \{a\in B\mid a=\bar c\} \ = \ \{\bar c\}
 \]
 are closed in $\A^1_k(B)=B$ for every $c\in k$. Depending on the null set of $B$, there might be more complicated strong Zariski closed subsets (which are infinite and proper), but in the case of an idyll $B=k$ there are not, i.e.,\ $\A^1_k(k)=k$ carries the cofinite topology, in perfect analogy to the Zariski topology in usual algebraic geometry.
\end{rem}

\subsection{The fine topology}
\label{subsection: the fine topology}

Let $k$ be a band and $B$ a $k$-algebra. The choice of a topology for a band $B$ endows the set $X(B)=\Hom_k(\Spec B,X)$ of $B$-rational points of a $k$-scheme $X$ with a topology, as described in the following. For more details, cf.\ \cite{Lorscheid23} and \cite{Lorscheid-Salgado16}.

In the case of an affine $k$-scheme $X=\Spec C$, we define the \emph{affine topology} for $X(B)=\Hom_k(C,B)$ as the compact open topology with respect to the discrete topology on $C$ and the given topology on $B$. For an arbitrary band scheme $X$, the \emph{fine topology} for $X(B)=\Hom_k(\Spec B,X)$ is the finest topology such that for every morphism $\alpha:Y\to X$ from an affine band scheme $Y$ to $X$, the induced map $\alpha_B:Y(B)\to X(B)$ is continuous with respect to the affine topology for $Y(B)$. 

A \emph{topological band} is a band $B$ together with a topology such that the multiplication $m:B\times B\to B$ is continuous. A topological band $B$ is \emph{with open unit group} if $B^\times$ is an open subset of $B$ and if the multiplicative inversion $B\to B$ is a continuous map. A topological ring is an example of a topological band, and a topological field is a topological band with open unit group. Note that a topological band that is a ring is not necessarily a topological ring, since its addition might fail to be continuous.

The following is a specialization of \cite[Lemma 6.1, Thm.\ 6.4]{Lorscheid23} from ordered blueprints to bands. We omit the proof, which is analogous.

\begin{thm}\label{thm: fine topology}
 Let $k$ be a band and $B$ a $k$-algebra with topology. Then the following holds:
 \begin{enumerate}
  \item\label{fine1} For affine $X=\Spec C$, the affine and the fine topology for $X(B)$ agree.
  \item\label{fine2} The fine topology for $X(B)$ is functorial in both $X$ and $B$, i.e.,\ a morphism $\varphi:Y\to X$ of band schemes induces a continuous map $\varphi_B:Y(B)\to X(B)$ and a continuous $k$-linear morphism $f:B\to C$ (where $C$ is a $k$-algebra with topology) induces a continuous map $f_X: X(B)\to X(C)$.
  \item\label{fine3} A closed immersion $\iota:Y\to X$ of band schemes yields a topological embedding $\iota_B:Y(B)\to X(B)$.
  \item\label{fine4} If $X=\bigcup U_i$ for open subschemes $U_i$ of $X$, then $X(B)=\bigcup U_i(B)$ as sets.
 \end{enumerate}
 If $B$ is a topological $k$-algebra, then the following holds:
 \begin{enumerate}\setcounter{enumi}{4}
  \item\label{fine5} The functor $\Hom(\Spec B,-):\BSch\to\Top$ commutes with finite limits; in particular, the canonical map $\left( X\times_kY\right)(B)\to X(B)\times_{B}Y(B)$ is a homeomorphism.
  \item\label{fine6} The canonical bijection $\A^1(B)\to B$ is a homeomorphism.
 \end{enumerate}
 If $B$ is a topological $k$-algebra with open unit group, then the following holds:
 \begin{enumerate}\setcounter{enumi}{6}
  \item\label{fine7} An open immersion $\iota:Y\to X$ of band schemes yields an open topological embedding $\iota_B:Y(B)\to X(B)$.
  \item\label{fine8} If $X=\colim\cU$ is an affine presentation, then the canonical map $\colim\cU(B)\to X(B)$ is a homeomorphism.
 \end{enumerate}
 If $k$ is a ring and $B$ is a topological ring, then the following holds:
 \begin{enumerate}\setcounter{enumi}{8}
  \item\label{fine9} The canonical bijection $X^+(B)\to X(B)$ is a homeomorphism.
  \end{enumerate}
 If, in addition, the topological ring $B$ is Hausdorff, then the following strengthening of \eqref{fine3} holds:
 \begin{enumerate}\setcounter{enumi}{9}
  \item\label{fine10} A closed immersion $\iota:Y\to X$ of band schemes yields a closed topological embedding $\iota_B:Y(B)\to X(B)$.
  \end{enumerate}
\end{thm}

\begin{ex}
 Let $B=k$ be a topological field. Then $X(k)$ satisfies all of \eqref{fine1}--\eqref{fine10}, thus in particular $X(k)$ is homeomorphic to $X^+(k)$, which is the set of $k$-rational points of the $k$-scheme $X^+$, equipped with the topology induced by the topological field $k$.
\end{ex}

\begin{ex}
 The tropical hyperfield $\T=\R_{\geq0}$, endowed with the natural order topology, is a topological band with open unit group. Thus \eqref{fine1}--\eqref{fine8} hold for $B=k=\T$. It is shown in \cite[Thm.\ 2.7]{Lorscheid22} that \eqref{fine10} also holds in this case, even though $\T$ is not a ring. Thus in particular a closed subscheme $X$ of $\A^n_\T$ defines a closed subspace $X(\T)$ of $\A^n_\T(\T)=\T^n$.
\end{ex}

\begin{ex}\label{ex: tropicalization}
 Combining the previous two examples, we consider a field $k$ together with a band morphism (or, equivalently, non-archimedean absolute value) $v:k\to\T$ and let $B=\T$. Let $X=\Spec R$ be an affine $k$-scheme of finite type, which we can consider as a band scheme by considering the $k$-algebra $R$ as a band. Then $X(\T)=\Hom_k(R,\T)$ is canonically homeomorphic to the Berkovich analytification of $X$ (cf.\ \cite[Thm.~3.5]{Lorscheid22}).
 
 Choosing generators $a_1,\dotsc,a_n$ for $R$ as a $k$-algebra yields a representation $R=k[a_1,\dotsc,a_n]/I$ for some ideal $I$. Interpreting $N_C=I$ as the null set for the pointed submonoid $A=\{ca_1^{e_1}\dotsb a_n^{e_n}\mid c\in k,e_1,\dotsc,e_n\in\N\}$ of $R$ generated by $k$ and $a_1,\dotsc,a_n$ defines the subband $C=\bandquot{A}{N_C}$ of $R$ and the band $k$-scheme $Y=\Spec C$. The embedding $C\hookrightarrow R$ yields a morphism $\pi:X\to Y$.
 
 Then $Y(\T)$ is canonically homeomorphic to the tropicalization $X^\trop$ of $X$ with respect to the embedding into $\A_k^n$ given by $a_1,\dotsc,a_n$, and the diagram
 \[
  \begin{tikzcd}[column sep=60]
   X^\an \ar[r,"\cong"] \ar[d,"\substack{\text{Payne}\\ \text{tropicalization}}"'] & X(\T) \ar[d,"\pi_\T"] \\
   X^\trop \ar[r,"\cong"] & Y(\T)
  \end{tikzcd}
 \]
 commutes; cf.\ \cite[Thm.~3.5]{Lorscheid22} for details.
\end{ex}

\begin{ex}\label{ex: K-points}
 The \emph{natural topology for $\K$} consists of the open subsets $\emptyset$, $\{1\}$ and $\K$, and it is natural in the sense that it is the quotient topology on $\K=F/F^\times$ where $F$ is a non-discrete topological field. This endows $X(\K)$ with a topology. In the case of the Grassmannian $X=\Gr(r,E)$, this topology characterizes weak maps of matroids: given two rank $r$ matroids $M$ and $N$ on $E$ with corresponding rational points $x_M$ and $x_N$ in $\Gr(r,E)(\K)$ (cf.\ \cite{Baker-Lorscheid21b}), there is a weak map $M\to N$ if and only if the $x_N$ is contained in the topological closure of $x_M$.
\end{ex}

\begin{ex}\label{ex: S-points}
 The \emph{natural topology for $\S$} is the quotient topology with respect to $\S=\R/\R_{\geq0}$, which consists of the open subsets $\emptyset$, $\{1\}$, $\{-1\}$, $\{\pm1\}$ and $\S$. This topology endows rational point sets $X(\S)$ of band schemes $X$ with a fine topology. If $X$ is an $\R$-scheme, the continuous map $\sign:\R\to\S$ induces a continuous map $X(\R)\to X(\S)$.
 
 As an application to matroid theory, we note that the topological space $\Gr(r,E)(\S)$ is homeomorphic to the \emph{MacPhersonian} $\MacPh(r,n)$ with its poset topology, cf.~\cite[Section 6]{Anderson-Davis19}. 
\end{ex}

We conclude this section with a comparison of the fine topology with the strong Zariski topology that extends the usual comparison of these two topologies for varieties over topological fields. A \emph{topological idyll} is a topological band $F$ with open unit group that is an idyll, and such that
\[
 N_F^{(n)} \ = \ \big\{ (a_1,\dotsc,a_n)\in (F^\times)^n \, \big| \, a_1+\dotsb+a_n\in N_F \big\}
\]
is a closed subset of $(F^\times)^n$ for all $n\geq1$. Note that $\{0\}=F-F^\times$ is closed in a topological idyll $F$.

Examples of topological idylls are topological fields $F$ (since $N_F^{(n)}$ is the inverse image of $\{0\}$ under the $n$-fold addition map), $\K$ and $\S$ (since both $\K^\times$ and $\S^\times$ are discrete), and $\T$ (since $N_\T^{(n)}$ is a semi-algebraic subset of $(\T^\times)^n\cong\R^n$) with respect to the natural topologies for $\K$, $\S$ and $\T$. Every idyll is trivially a topological idyll with regard to the discrete topology.

\begin{prop}\label{prop: the fine topology is finer than the Zariski topology}
 Let $F$ be a topological idyll with a $k$-algebra structure $k\to F$ and $X$ a band $k$-scheme. Then the fine topology is finer than the strong Zariski topology for $X(F)$.
\end{prop}

\begin{proof}
 The fine topology is local in nature, and by \autoref{prop: Zariski topology} so is the strong Zariski topology for idylls, which allows us to assume that $X=\Spec C$ is affine. The strong Zariski topology is generated by the closed subsets of the form $Z_h(F)$, where $Z_h=\Spec(\bandgenquot Ch)$ for $h=a_1+\dotsb+a_n\in C$, so all we need to show is that $Z_h(F)$ is closed in the fine topology for $X(F)$.
 
 Since $N_F^{(n)}$ is closed in $(F^\times)^n$, it is of the form
 \[
  N_F^{(n)} \quad = \quad \bigcap_{i\in I} \ \ \Big( \ \ \bigcup_{j=1}^{r_i} \ \ A_{i,j,1}\times\dotsb\times A_{i,j,n} \ \ \bigg)
 \]
 for closed subsets $A_{i,j,k}$ of $F^\times$, where $I$ is some index set and the $r_i$ are suitable integers. Using \autoref{prop: Zariski topology} and this presentation of $N_F^{(n)}$, we find that 
 \begin{align*}
  Z_h(F) \quad &= \quad \{ f:C\to F \mid f(a_1)+\dotsb+f(a_n)\in N_F\} \\ 
               &= \quad \bigcap_{i\in I} \ \ \bigcup_{j=1}^{r_i} \ \ \bigcap_{k=1}^n \{f:C\to F\mid f(a_k)\in A_{i,j,k}\} 
 \end{align*}
 is closed in the fine topology for $X(F)$, as claimed.
\end{proof}

\subsection{The kernel space}
\label{subsection: the kernel space}

Let $X$ be an affine band scheme. The \emph{kernel space of $X$} is the subspace $X^\kernel$ of $X$ that consists of all prime $k$-ideals of $\Gamma X$. Note that \autoref{prop: ideals in localizations} guarantees that $U^\kernel$ is an open subspace of $X^\kernel$ for every affine open subscheme $U$ of $X$. This allows us to extend the \emph{kernel space of $X$} to arbitrary band schemes $X$ as its subspace $X^\kernel=\bigcup U^\kernel$ where $U$ varies over all affine open subschemes $U$ of $X$. 

Sending a $\K$-point $\alpha:\Spec\K\to X$ to its unique image point $\alpha(\gen 0)$ defines the \emph{characteristic map} $\chi:X(\K)\to X$. Recall from \autoref{ex: K-points} that the natural topology for $\K$ consists of the open subsets $\emptyset$, $\{1\}$, and $\K$. 

\begin{lemma}\label{lemma: k-space as K-rational points}
 The characteristic map $\chi:X(\K)\to X$ restricts to a homeomorphism $\chi^\kernel:X(\K)\to X^\kernel$ with respect to the natural topology for $\K$.
\end{lemma}

\begin{proof}
 Since both topologies are defined in terms of affine opens, we can assume that $X=\Spec B$ is affine. The image point of a morphism $\alpha:\Spec\K\to X$ is a $k$-ideal $\fp$ of $B$. Since $\fp=\Gamma\alpha^{-1}(\gen0)$ and $\K$ has a unique nonzero element, $\fp$ determines $\Gamma\alpha:B\to\K$ as the band morphism with $\Gamma\alpha(a)=0$ if $a\in\fp$ and $\Gamma\alpha(a)=1$ otherwise. This shows that $\chi$ restricts to an injection $\chi^\kernel:X(\K)\to X^\kernel$.
 
 For a given prime $k$-ideal $\fp$ of $B$, the map $\pi_\fp:B\to \K$ with $\pi_\fp(a)=0$ for $a\in\fp$ and $\pi_\fp(a)=1$ for $a\notin\fp$ is a morphism of bands. Since $\chi^\kernel(\pi_\fp^\ast)=\fp$, this shows that $\chi^\kernel:X(\K)\to X$ is surjective, and thus a bijection.
 
Let $U_h=\{\fp\in X\mid h\notin\fp\}$ be a principal open affine of $X$ with $h\in B$. Then
 \[
  \chi^{-1}(U_h) \ = \ \{\alpha:\Spec\K\to X\mid \Gamma\alpha(h)\in\{1\}\}
 \]
 is an open subset of $X(\K)$, since $\{1\}$ is open in $\K$. This shows that $\chi^\kernel:X(\K)\to X^\kernel$ is continuous. Since the topology of $X(\K)$ is generated by open subsets of this form, we conclude that $\chi^\kernel$ is a homeomorphism, as claimed.
\end{proof}

The restriction of the structure sheaf $\cO_X$ from $X$ to $X^\kernel$ fails to be a sheaf, since the kernel space of affine band schemes has non-trivial coverings in general (\autoref{ex: kernel space with nontrivial coverings}), which violates \autoref{prop: covering families of affine band schemes}. But the notions of stalks and residue fields carry over, and provide yet another characterization of the kernel space. 

\begin{prop}\label{prop: characterization of the kernel space in terms of residue fields}
 Let $X$ be a band scheme with kernel space $X^\kernel$ and $x\in X$. Then the following are equivalent:
 \begin{enumerate}
  \item\label{kernel1} $x\in X^\kernel$;
  \item\label{kernel2} $\{x\}$ is the image of a morphism $Y\to X$ for some band scheme $Y$;
    \item\label{kernel3} $k(x)$ is an idyll;
  \item\label{kernel4} $k(x)\neq\0$.
 \end{enumerate}
\end{prop}

\begin{proof}
 Since $x$ is contained in an affine open of $X$, we can assume throughout that $X=\Spec B$ is affine and that $x$ corresponds to a prime $m$-ideal of $B$. Replacing $B$ by $B_\fp$ lets us assume that $\fp$ is the maximal $m$-ideal of $B$. Thus $k(x)=\bandgenquot B\fp$. We establish a circle of implications \eqref{kernel1}$\Rightarrow$\eqref{kernel2}$\Rightarrow$\eqref{kernel3}$\Rightarrow$\eqref{kernel4}$\Rightarrow$\eqref{kernel1}. 
  
 Assume \eqref{kernel1}. Then $\fp$ is the kernel of $B\to \bandgenquot B\fp=k(x)$ and $\{x\}$ is the image of $\kappa_x:\Spec k(x)\to X$, which implies \eqref{kernel2}.
 
 Assume \eqref{kernel2}. If $\{x\}$ is the image of $Y\to X$, then it is the image of an affine open of $Y$. Thus we can assume that $Y=\Spec C$ is affine and that $Y\to X$ corresponds to a band morphism $f:B\to C$ into a nonzero band $C$ with $f^{-1}(\fq)=\fp$ for all prime $m$-ideals $\fq$ of $C$. By \autoref{proposition: maximal k-ideals are prime}, $C$ has a maximal $k$-ideal $\fm$, which shows that $\fp=f^{-1}(\fm)$ is a $k$-ideal by \autoref{lemma: inverse images of ideals and k-ideals as kernels}. Thus $\fp$ is the kernel of $B\to \bandgenquot B\fp=k(x)$, which shows that $k(x)$ is nonzero. We conclude that $k(x)$ is an idyll, which establishes \eqref{kernel3}.

 The implication \eqref{kernel3}$\Rightarrow$\eqref{kernel4} is evident. Assume \eqref{kernel4}. If $\bandgenquot B\fp=k(x)\neq0$, then $\fp=\ker(B\to k(x))$ is a $k$-ideal and belongs to $X^\kernel$. This establishes \eqref{kernel1}.
\end{proof}

\begin{ex}\label{ex: kernel space with nontrivial coverings}
 We provide two examples of kernel spaces $X^\kernel$ with non-trivial coverings which illustrate the fact that the restriction of the structure sheaf $\cO_X$ can fail to be a sheaf.
 
 The first example is $B=\bandgennquot{\Funpm[a,b,s,t]}{ta-sb,\ s+t-1}$ and $X=\Spec B$. Since $s+t-1\in N_B$, there is no $k$-ideal that contains both $s$ and $t$. Therefore $X^\kernel=U_s^\kernel\cup U_t^\kernel$ is an open covering. We have local sections $\frac as\in\cO_X(U_s)$ and $\frac bt\in\cO_X(U_t)$ that agree on the intersection $U_{st}=U_s\cap U_t$, since $ta=sb$ in $B$, but there is no global section in $\cO_X(X)=B$ that restricts to the pair of local sections $\frac as$ and $\frac bt$. This shows that the restriction of $\cO_X$ to $X^\kernel$ is not a sheaf.

 The second example, adapted from \cite[Example 4.24]{Jun18}, is the hyperring quotient $R=(\Q \times \Q)/\{\pm1\}$. Its multiplication is defined coordinate-wise, and its null set is
 \[\textstyle
  N_R \ = \ \{ \sum a_i \in R^+ \mid 0 = \sum \epsilon_ia_i \text{ for some }\epsilon_i\in\{\pm1\}\}. 
 \]
 The band $R$ has two prime $k$-ideals: $\fp_1=\langle \overline{(1, 0)} \rangle_k$ and $\fp_2 = \langle \overline{(0, 1)} \rangle_k$. Thus the kernel space of $X=\Spec R$ consists of the two closed points $\fp_1$ and $\fp_2$ with stalks $R_{\fp_1} \simeq R_{\fp_2} \simeq \Q/\{\pm1\}$. Since $(\Q/\{\pm1\}) \oplus (\Q/\{\pm1\})$ is not isomorphic to $R$, this shows that the restriction of the structure sheaf of $X$ to $X^\kernel$ is not a sheaf.
\end{ex}

\subsection{The Tits space}
\label{subsection: the Tits space}

The \emph{Tits space $X^\Tits$} of a band scheme $X$ is defined as the set of closed points of $X^\kernel=X(\K)$. In many examples, it realizes hoped-for analogies in $\Fun$-geometry. 

For example, points of the Tits space of the Grassmannian $\Gr(r,n)$ correspond to $r$-element subsets of an $n$-element set, the Tits space of a building of type $A$ is its Coxeter complex, and the Tits space of $\SL_n$ is its Weyl group. In this section, we expand on material from \cite{Lorscheid-Thas23}.

\subsubsection{Vector spaces over \texorpdfstring{$\Fun$}{F1} and their subspaces}
The naive approach to linear algebra over $\Fun$ is as follows: an $\Fun$-vector space is a set; its dimension is its cardinality; an $\Fun$-linear subspace is a subset. In particular, the $\Fun$-rational points of the Grassmannian $\Gr(r,n)$ of $r$-dimensional subspaces of $n$-space should correspond to the $r$-subsets of $E=\{1,\dotsc,n\}$.

From the point of view of band schemes, this naive approach becomes correct if one replaces $\Fun$-rational points by (closed) $\K$-rational points. Indeed, as explained in \autoref{ex: Grassmannians}, the $\K$-points of the band scheme $\Gr(r,n)$ correspond to rank $r$ matroids $M$ on $E$. By \autoref{ex: K-points}, the closure of a $\K$-point $x_M$ contains another $\K$-point $x_N$ if and only if there is a weak map $M\to N$. Therefore the closed points of $\Gr(r,n)(\K)$ are precisely the rank $r$ matroids on $E$ with a unique basis, and such matroids correspond bijectively to $r$-subsets of $E$. This establishes the postulated bijection between $\Gr(r,n)^\Tits$ and the family of all $r$-subsets of $E$.

Note that this also agrees with the general expectation in $\Fun$-geometry that the number of $\Fun$-points of a ``nice'' space $X$ should equal the Euler characteristic of the associated complex variety $X_{\C}$:
\[
 \chi\big(\Gr(r,n)(\C)\big) \ = \ \binom rn \ = \ \#\,\big\{\text{$r$-subsets of $E$}\big\} \ = \ \#\, \Gr(r,n)^\Tits.
\]

\subsubsection{Flags and buildings}
More generally,  in $\Fun$-geometry a flag of $\Fun$-linear subspaces of $E$ is thought of as a flag $F_1\subset\dotsc\subset F_s$ of subsets of $E$, and its type is $\br=(r_1,\dotsc,r_s)$, where $r_i=\# F_i$ for $i=1,\dotsc,s$. Such flags correspond bijectively to  the closed points of $\Fl(\br,n)(\K)$, which are flag matroids $(M_1,\dotsc,M_s)$ of type $\br$ on $E$ for which each $M_i$ has a unique basis $F_i$.

There are canonical projections $\pi_{\br,I}:\Fl(\br,n)\to\Fl(\br_I,n)$ for every subset $I$ of $\{1,\dotsc,s\}$ and type $\br_I=(r_i)_{i\in I}$, sending a flag $(F_i)_{i=1,\dotsc,s}$ to $(F_i)_{i\in I}$. The projection $\pi_{\br,I}$ corresponds to a face map when we consider a flag $F_1\subset\dotsc\subset F_s$ as an $(s-1)$-dimensional simplex. Thus the system of flag varieties $\Fl(\br,n)$ (of various types $\br$), together with the canonical projections $\pi_{\br,I}$ (for varying $\br$ and $I$), define a functor $\cB_n$ from $\Bands$ to simplicial complexes.

For a field $K$, the simplicial complex $\cB_n(K)$ is the spherical building of type $A_n$ whose simplices correspond to flags of subspaces of $K^n$. The simplicial complex $\cB_n(\K)$ is the combinatorial flag variety $\Omega_n$ of type $A_n$, as introduced by Borovik, Gelfand, and White in \cite{Borovic-Gelfand-White00}. The projections $\pi_{\br,I}$ induce maps between the Tits spaces of $\Fl(\br,n)$ and $\Fl(\br_I,n)$, which corresponds to omitting entries of a flag $F_1\subset\dotsc\subset F_s$ of subsets of $E$. The set of closed points of $\cB_n(\K)$ can be naturally identified with the Coxeter complex $\cC_n$ of $S_n$, and the embedding of $\Fl(\br,n)^\Tits$ into $\Fl(\br,n)^\kernel=\Fl(\br,n)(\K)$ (for various $\br$) corresponds to the embedding of $\cC_n$ into $\Omega_n$.

The case $n=3$ is illustrated in \autoref{fig: Coxeter complex and combinatorial flag variety for S3}, where we label the vertices of $\Omega_3$ by a binary matrix that represents the corresponding matroid. The Coxeter complex $\cC_3$ and its image in $\Omega_3$ are drawn in orange. 

\begin{figure}[htb]
 \[
  \newcounter{tikz-counter}
  \begin{tikzpicture}
   \node (origin) {};
   \foreach \a in {1,...,14}{\draw (\a*360/7+0/7: 2cm) node [draw,circle,inner sep=2pt] (l\a) {};}
   \foreach \a in {1,...,7}{\draw (\a*360/7+180/7: 2cm) node [draw,circle,inner sep=1.8pt,fill] (p\a) {};
                              \draw [-] (p\a) -- (l\a);
                              \setcounter{tikz-counter}{\a};
                              \addtocounter{tikz-counter}{1};
                              \draw [-] (p\a) -- (l\arabic{tikz-counter});
                              \addtocounter{tikz-counter}{2};
                              \draw [-] (p\a) -- (l\arabic{tikz-counter});
                             }
   \draw [-] (p4) -- (l6);                          
   \node at      (0:2.8cm) {$\tinymatrix 110001$};                          
   \node at (1440/7:2.8cm) {$\tinymatrix 100011$};                          
   \node at (1800/7:2.8cm) {$\tinymatrix 101010$};                          
   \node at (2160/7:2.8cm) {$\tinymatrix 110011$};                          
   \node at (1260/7:2.8cm) {$\tinyvector 011$};                          
   \node at (1620/7:2.8cm) {$\tinyvector 111$};                          
   \node at (1980/7:2.8cm) {$\tinyvector 101$};                          
   \node at (2340/7:2.8cm) {$\tinyvector 110$};   
   \node[color=orange!90!black] at  (360/7:2.8cm) {$\tinymatrix 100001$};                          
   \node[color=orange!90!black] at  (720/7:2.8cm) {$\tinymatrix 100010$};                          
   \node[color=orange!90!black] at (1080/7:2.8cm) {$\tinymatrix 001010$};                          
   \node[color=orange!90!black] at  (180/7:2.8cm) {$\tinyvector 001$};                          
   \node[color=orange!90!black] at  (540/7:2.8cm) {$\tinyvector 100$};                          
   \node[color=orange!90!black] at  (900/7:2.8cm) {$\tinyvector 010$};                          
   \foreach \a in {1,...,3}{\draw (\a*360/7+0/7: 2cm) node [draw,thick,color=orange!90!black,circle,inner sep=2pt] (l\a) {};}
   \foreach \a in {0,...,2}{\draw (\a*360/7+180/7: 2cm) node [draw,thick,color=orange!90!black,circle,inner sep=1.8pt,fill=orange!90!black] (p\a) {};}
   \draw [-,color=orange!90!black,thick] (p0) -- (l1) -- (p1) -- (l2) -- (p2) -- (l3) -- (p0);
   \foreach \a in {1,...,6}{\draw [xshift=-6cm,yshift=1cm] (\a*360/3+30: 1.2cm) node [draw,thick,color=orange!90!black,circle,inner sep=2pt] (l\a) {};}
   \foreach \a in {1,...,3}{\draw [xshift=-6cm,yshift=1cm] (\a*360/3+90: 1.2cm) node [draw,thick,color=orange!90!black,circle,inner sep=1.8pt,fill=orange!90!black] (p\a) {};
                              \draw [-,color=orange!90!black,thick] (p\a) -- (l\a);
                              \setcounter{tikz-counter}{\a};
                              \addtocounter{tikz-counter}{1};
                              \draw [-,color=orange!90!black,thick] (p\a) -- (l\arabic{tikz-counter});
                              \addtocounter{tikz-counter}{2};
                              \draw [-,color=orange!90!black,thick] (p\a) -- (l\arabic{tikz-counter});
                             }
   \draw[->,color=orange!90!black,thick] (-4.2,1) -- (-3.5,1);                          
  \end{tikzpicture}
 \] 
 \caption{Embedding of the Coxeter complex $\cC_3$ into $\Omega_{S_3}$.}
 \label{fig: Coxeter complex and combinatorial flag variety for S3}
\end{figure}
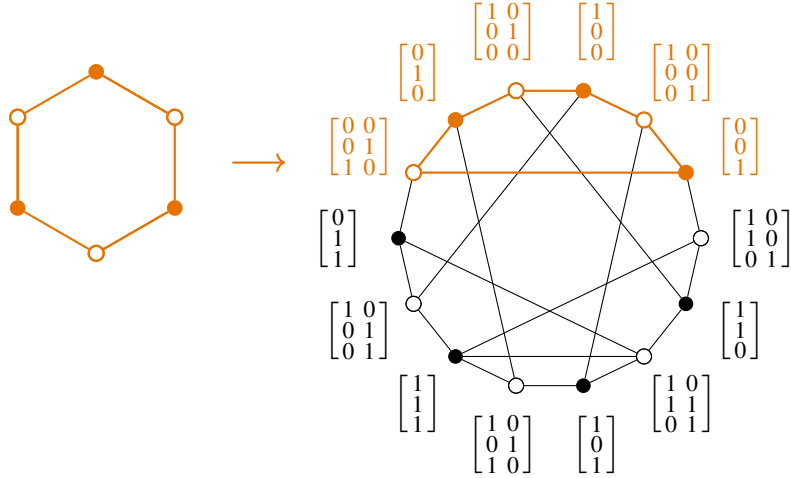

\subsubsection{Chevalley groups and their Weyl groups}
One of the best known expectations in $\Fun$-geometry is that the $\Fun$-points of (a suitable model of) a Chevalley group $G$ should correspond to its Weyl group $W$ in a canonical way. When $\Fun$-points are interpreted as morphisms $\Spec\Fun\to G$, then the base extension from $\Fun$ to $\Z$ would yield a map 
\[
 W \ \simeq \ G(\Fun) \ = \ \Hom(\Spec\Fun,G) \ \longrightarrow \ \Hom(\Spec\Z,G) \ = \ G(\Z).
\]
A natural expectation on this map  (cf.\ \cite[Prologue]{Lorscheid16} for details) would be that it factors through a section $\sigma:W\to N(\Z)$ to the canonical short exact sequence
\[
 \begin{tikzcd}[column sep=50pt]
  0 \ar[r] & T(\Z) \ar[r] & N(\Z) \ar[r] & W \ar[r] \ar[l,dotted,bend right=20pt,"\sigma"'] & 0,
 \end{tikzcd}
\]
where $T$ is the diagonal torus and $N$ is its normalizer in $G$. For $G=\SL_2$, this sequence is isomorphic to
\[
 \begin{tikzcd}[column sep=50pt]
  0 \ar[r] & \Z/2\Z \ar[r] & \Z/4\Z \ar[r] & \Z/2\Z \ar[r] & 0,
 \end{tikzcd}
\]
which does not split. The Tits space avoids this pitfall while nevertheless achieving a satisfying outcome.

Indeed, the Weyl group appears as the Tits space of the obvious model of $\SL_n$ as a band scheme, which is
\[
 \SL_n \ = \ \Spec \ \bandgenquot{\Funpm[T_{i,j}\mid i,j=1,\dotsc,n]}{\det(T_{i,j})-1},
\]
where
\[
 \det(T_{i,j}) \ = \ \sum_{\sigma\in S_n} \, \sign(\sigma) \ \cdot \ \prod_{i=1}^n \ T_{i,\sigma(i)}.
\]
The closed points of $\SL_n^\kernel$ are the prime $k$-ideals of the form $\fp_\sigma=\gen{T_{i,j}\mid j\neq\sigma(i)}_k$ with $\sigma\in S_n$. The fibre of $\fp_\sigma$ under the canonical map $\SL_n(\C)\to\SL_n(\K)=\SL_n^\kernel$ is the coset $\sigma T(\C)$ of the diagonal torus, which establishes a bijection between $\SL_n^\Tits$ and the Weyl group $N(\C)/T(\C)\simeq S_n$ of $\SL_n$ (where $N$ is the normalizer of $T$ in $\SL_n$). 

As a final remark, we mention that $\SL_n$, as a band scheme, carries a natural structure of a ``crowd'', cf.\ \cite{Lorscheid-Thas23}. Roughly speaking, crowds generalize groups in the same way that bands generalize rings; the group law, which is a binary operation, is replaced by a ternary relation consisting of triples of elements whose product is $1$. This endows $\SL_n^\Tits$ with a crowd structure, which coincides with that of the Weyl group $S_n$ under the above identification. Moreover, there is a natural extension of the group action of $\SL_n$ on buildings to a crowd action, which specializes to the usual action of $S_n=\SL_n^\Tits$ on the Coxeter complex $\cC_n=\cB_n^\Tits$.

\subsection{The null space}
\label{subsection: the null space}

A null ideal $I$ of a band $B$ is \emph{prime} if $S=B^+-I$ is a multiplicative set in $B^+$. The \emph{null space of $B$} is the set $\Null(B)$ of prime null ideals of $B$ together with the topology generated by \emph{basic opens} of the form
\[
 U_h \ = \ \big\{\fp\in\Null(B) \, \big| \, h\notin\fp \big\}
\]
for $h\in B^+$. 

Since the intersection $I\cap B$ of a null ideal $I$ with $B$ is an $m$-ideal of $B$ and since the intersection $S\cap B$ of a multiplicative set $S$ in $B^+$ with $B$ is a multiplicative set in $B$, the association $\fp\mapsto\fp\cap B$ defines a map $\bar\pi_B:\Null(B)\to\Spec B$. The inverse image of a basic open $U_h$ of $\Spec B$ (with $h\in B$) is the basic open $U_h$ of $\Null(B)$, which shows that $\bar\pi_B$ is continuous.

A band morphism $f:B\to C$ induces a continuous map $f^\ast:\Null(C)\to\Null(B)$ that pulls back a prime null ideal $\fp$ of $C$ to the prime null ideal $f^\ast(\fp)=(f^+)^{-1}(\fp)$ of $B$. The pullback is compatible with $\bar\pi$ in the sense that $f^\ast\circ\bar\pi_B=\bar\pi_C\circ f^\ast$. This upgrades $\Null$ to a functor from $\Bands$ to $\Top$, and $\bar\pi$ to a natural transformation from $\Null$ to $\utop\circ\Spec$, where $\utop:\BSch\to\Top$ sends a band scheme to its underlying topological space.

\begin{thm}\label{thm: null space}
Up to unique isomorphism, there is a unique functor $(-)^\nul:\BSch\to\Top$, together with a natural transformation $\pi:(-)^\nul\to\utop$ and a natural isomorphism $\eta:\Null\to(-)^\nul\circ\Spec$, with the following properties: 
 \begin{enumerate}
  \item \label{null1} $\bar\pi=\pi\circ\eta$, i.e.,\ 
  \[
   \begin{tikzcd}[column sep=30]
    \Null(B) \ar[rr,"\simeq","\eta_B"'] \ar[dr,"\bar\pi_B"'] && (\Spec B)^\nul \ar[dl,"\pi_{\Spec B}"] \\
    & \Spec(B) 
   \end{tikzcd}
  \]
  commutes for every band $B$.
  \item \label{null2} An affine presentation $\colim\cU\to X$ induces a homeomorphism $\colim\cU^\nul\to X^\nul$.
  \item \label{null3} An open (resp.\ closed) immersion $\iota:Y\to X$ induces an open (resp.\ closed) topological embedding $\iota_\ast:Y^\nul\to X^\nul$.
 \end{enumerate}
\end{thm}

The proof of this theorem is analogous to that of \cite[Thm.~A, Prop.~4.2]{Lorscheid-Ray23}, and we refer to the latter text for more details.

\begin{proof}
 Properties \eqref{null1} and \eqref{null2} lead to the following definition for $(-)^\nul$, $\pi$, and $\eta$: for an affine band scheme $X$, we define 
 \[
  X^\nul \ = \ \Null(\Gamma X), \qquad \pi_X \ = \ \bar\pi_{\Gamma X}, \qquad \text{and} \qquad \eta_{\Gamma X} \ = \ \id_{\Null(\Gamma X)}. 
 \]
 For a morphism $\varphi:X\to Y$ of affine band schemes, we define $\varphi_\ast=(\Gamma\varphi)^\ast:X^\nul\to Y^\nul$. For an arbitrary band scheme $X$ with affine presentation $X=\colim\cU$, we define $X^\nul=\colim\cU^\nul$. The map $\pi_X:X^\nul\to X$ is determined by its restrictions to affine open subschemes of $X$. Since every morphism of band schemes restricts to a morphism between affine presentations, by \autoref{subsubsection: local nature of band schemes}, the functoriality of $(-)^\nul$ dictates its extension to morphisms between arbitrary band schemes.
 
 The uniqueness of $(-)^\nul$, $\pi$, and $\eta$ follows from their construction. In the following, we verify that this definition is independent of our choices and that \eqref{null3} holds. By \autoref{rem: reduced hypotheses for extensions of functors}, we need to verify that $\Null:\Bands\to\Top$ preserves cofibre products and sends finite localizations to open topological embeddings.
 
 \subsubsection*{\texorpdfstring{$\Null$}{Null} preserves cofibre products}
 Consider band morphisms $C\leftarrow B\to D$ and the canonical inclusions $\iota_C:C\to C\otimes_BD$ and $\iota_D:D\to C\otimes_BD$. The corresponding continuous maps $\Null(C)\to\Null(B)\leftarrow\Null(D)$ induce a canonical map $\Phi:\Null(C\otimes_BD)\to\Null(C)\times_{\Null(B)}\Null(D)$, which sends a prime null ideal $\fp$ of $C\otimes_ B D$ to the element $(\iota_C^\ast(\fp),\iota^\ast(\fp)\big)$ of $\Null(C)\times_{\Null(B)}\Null(D)$. 
 
 We construct an inverse to $\Phi$ as follows. Given $(\fq_C,\fq_D)\in \Null(C)\times_{\Null(B)}\Null(D)$, we define $\Psi(\fq_C,\fq_D)=\gen{\iota_C^+(\fp),\iota_D^+(\fq)}$, which is a prime null ideal of $C\otimes_BD$ since whenever $(x\otimes y)\cdot(x'\otimes y')=xx'\otimes yy'\in\Psi(\fq_C,\fq_D)$, we have $xx'\in\fq_C$ or $yy'\in\fq_D$ and thus one of
 \[
  x\in\fq_C, \quad x'\in\fq_C, \quad  y\in\fq_D, \quad \text{or} \quad  y'\in\fq_D
 \]
 holds, since $\fq_C$ and $\fq_D$ are prime. Consequently, either $x\otimes y$ or $x'\otimes y'$ is in $\Psi(\fq_C,\fq_D)$, which shows that $\Psi(\fq_C,\fq_D)$ is indeed a prime null ideal.
 
 In order to show that $\Phi$ and $\Psi$ are mutually inverse bijections, we need to show that the tautological inclusions $\Psi(\Phi(\fp))\subset \fp$ and $(\fq_C,\fq_D)\subset\Phi(\Psi(\fq_C,\fq_D))$ are equalities. Consider a null ideal $\fp$ of $C\otimes_BD$ and $x\otimes y\in\fp$. Then either $x\otimes 1\in\fp$ or $1\otimes y\in\fp$, since $\fp$ is prime, and thus $\fp$ is generated by $\iota_C^\ast(\fp)$ and $\iota_D^\ast(\fp)$.  This shows that $\Psi(\Phi(\fp))=\fp$. Next, consider $(\fq_C,\fq_D)\in\Null(C)\times_{\Null(B)}\Null(D)$ and elements $x\in\iota_C^\ast(\fp)$ and $y\in\iota_D^\ast(\fp)$, where $\fp=\Psi(\fq_C,\fq_D)$. Then $x\otimes1\in\fp$ and $1\otimes y\in\fp$, which implies $x\in\fp_C$ and $y\in\fp_D$. This shows that $(\fq_C,\fq_D)=\Phi(\Psi(\fq_C,\fq_D))$. 
 
 We have therefore shown that the map $\Phi:\Null(C\otimes_BD)\to\Null(C)\times_{\Null(B)}\Null(D)$ is a bijection. It is a homeomorphism since basic opens agree:
 \[
  \Phi\big(U_{g\otimes h}\big) \ = \ \Phi\big(U_{g\otimes1} \cap U_{1\otimes h}\big) \ = \ \big(U_g\times\Null(D)\big) \cap \big(\Null(C)\times U_h\big) \ = \ U_g\times U_h
 \]
 for $g\in C^+$ and $h\in D^+$. This verifies that $\Null$ preserves cofibre products.
 
 \subsubsection*{\texorpdfstring{$\Null$}{Null} sends finite localizations to open topological embeddings}
 Consider a finite localization $\iota_a: B\to B[a^{-1}]$ for $a\in B$ and the induced continuous map $\iota_a^\ast:\Null(B[a^{-1}])\to \Null(B)$ between the respective null spaces. 
 
 As a first step, we show that $\iota_a^\ast$ injects onto the subset of prime null ideals $\fp$ of $B$ that do not contain $a$. Let $\fp$ be such a null ideal of $B$, and let $\fq=\gen{\iota_a^+(\fp)}$. If $\frac b{a^i}\in\fq$, then $\frac b{a^i}=\frac c{a^j}\cdot\frac {d}1$ for some $d\in\fp$ by the definition of $\fq$. Since $a^{j+k}b=a^{i+k}cd$ for some $k$, and since $\fp$ is prime, $b\in\fp$. This shows that
 \[\textstyle
  \fq \ = \ \gen{\iota_a^+(\fp)} \ = \ \big\{ \frac b{a^i}\in B[a^{-1}] \, \big| \, b\in\fp \big\}.
 \]
 We conclude that $\iota_a^\ast(\fq)=\fp$. Moreover, $\frac 11\notin\fq$ since $1\notin\fp$. If $\frac{bc}{a^{i+j}}=\frac{b}{a^i}\cdot\frac{c}{a^j}\in\fq$, then $bc\in\fp$ and thus (since $\fp$ is prime) either $b\in\fp$ or $c\in\fp$. Therefore either $\frac{b}{a^i}\in\fq$ or $\frac{c}{a^j}\in\fq$. This shows that $\fq$ is prime, and that the image of $\iota_a^\ast$ consists of all null ideals of $B$ that do not contain $a$.
 
 Next we consider a prime $m$-ideal $\fq$ of $B[a^{-1}]$ and $\frac b{a^i}\in\fq$. Then $\frac b1=\frac {a^i}1\cdot\frac b{a^i}\in\fq$ and thus $b\in\fp$ for $\fp=\iota_a^\ast(\fq)$. Therefore $\frac b{a^i}\in\gen{\iota_a^+(\fp)}$, which shows that $\iota_a$ is injective and completes the first step.
 
 As our next step, we show that the inclusion $\iota_a^\ast:\Null(B[a^{-1}])\to \Null(B)$ is an open topological embedding. The inverse image of $U_h=\{\fp\in\Null(B)\mid h\notin\fp\}$ under $\iota_a^\ast$ is $U_{\frac h1}$. Since $\frac{a^i}1$ is invertible in $B[a^{-1}]$, we have 
 \[\textstyle
  U_{\frac h{a^i}} \ = \ \{ \fq\in\Null(B[a^{-1}] \mid \frac h{a^i}\notin\fq\} \ = \ \{ \fq\in\Null(B[a^{-1}] \mid \frac h{1}\notin\fq\} \ = \ U_{\frac h{1}}
 \]
 for all $h\in B^+$. This shows that $\iota_a^\ast$ is a topological embedding. Since its image is $U_{\frac a1}$, it is open.
 
 In conclusion, $\Null$ preserves cofibre products and sends open immersions to open topological embeddings. This concludes the proof of the existence of $(-)^\nul$, $\pi$, and $\eta$. In particular, this establishes \eqref{null1} and \eqref{null2}. Since open immersions are local in nature, this also shows that an open immersion $\iota:Y\to X$ induces an open topological embedding $Y^\nul\to X^\nul$, as claimed in \eqref{null3}.

\subsubsection*{Closed immersions induce closed topological embeddings}
 We are left with the second claim of \eqref{null3}. Since closed immersions are local in nature, it suffices to consider maps of the form $\pi^\ast:\Null(\bandquot BI)\to\Null(B)$, where $\pi$ is the quotient map $\pi:B\to \bandquot BI$. Since $\pi$ is surjective, $\pi^\ast$ is injective. The image of a basic closed subset $Z_{x}=\{\fq\in\Null(\bandquot BI) \mid x\in\fq\}$ of $\Null(\bandquot BI)$, with $x\in (\bandquot BI)^+$, is the closed subset 
 \[
  \pi^\ast(Z_x) \ = \ \{\fp\in\Null(B)\mid (\pi^+)^{-1}(x)\subset\fp\} \ = \ \bigcap_{h\in (\pi^+)^{-1}(x)} \{\fp\in\Null(B)\mid h\in\fp\}
 \]
 of $\Null(B)$. This shows that $\pi^\ast$ is a closed topological embedding, and completes the proof.
\end{proof}
 
\begin{rem}
 As shown in \cite{Lorscheid-Ray23}, there is a useful notion of closedness for morphisms of monoid schemes, defined in terms of their so-called congruence spaces. It seems to us that the right device to replace the congruence space of a monoid scheme in the case of band schemes should be the null space. This point of view is of importance for topological characterizations of separated and proper morphisms. 
\end{rem}

\subsubsection{Residue fields}
Just as with the kernel space, the null space of an affine band scheme $X=\Spec B$ does not carry a natural sheaf whose global section are equal to $B$. But the notion of stalks and residue fields extends to null spaces as follows.

\begin{df}
 Let $X$ be a band scheme with null space $X^\nul$ and $x\in X^\nul$. Let $\pi_X:X^\nul\to X$ be the canonical map and $\bar x=\pi_X(x)$. The \emph{stalk at $x$} is the stalk $\cO_{X,\bar x}$ of $X$ at $\bar x$.
 
 Let $\iota_x:\Spec \cO_{X,\bar x}\to X$ be the canonical inclusion, and let $P_x$ be the null ideal of $\cO_{X,x}$ that corresponds to $x$ as a point of the subspace $\Null(\cO_{X,x})$ of $X$. The \emph{residue field at $x$} is the band $k_x=\bandquot{\cO_{X,x}}{P_x}$. It comes with a canonical morphism $\kappa_x:\Spec k_x\to X$.
\end{df}

The residue field $k_x=\bandquot{\cO_{X,x}}{P_x}$ of a point $x\in X^\nul$ is an idyll since $\bar x=\pi_X(x)$ is a $k$-ideal and the canonical map $k(\bar x)\to k_x$ is an isomorphism of pointed monoids.

\begin{prop}\label{prop: universal property of the residue fields of the null space}
 Let $X$ be a band scheme with null space $X^\nul$ and $x\in X^\nul$. Let $\pi_X:X^\nul\to X$ be the canonical map and $\bar x=\pi_X(x)$. Let $F$ be an idyll such that $N_F$ is a prime ideal of $F^+$ and $\varphi:\Spec F\to X$ a morphism with image $\{\bar x\}$ such that the image of $\varphi^\nul$ is contained in the topological closure of $x$ in $X^\nul$. Then $\varphi$ factors uniquely through $\kappa_x:\Spec k_x\to X$.
\end{prop}

\begin{proof}
 By \autoref{prop: universal property of the residue field}, the morphism $\varphi:\Spec F\to X$ factorizes uniquely through $\kappa_{\bar x}:\Spec k(\bar x)\to X$. Thus we can assume that $X=k(\bar x)$ and that $\varphi$ corresponds to a band morphism $k(\bar x)\to F$ such that $P_x\subset f^{-1}(Q)$ for every prime null ideal $Q$ of $F$. In particular $P_x\subset f^{-1}(N_F)$, thus by \autoref{prop: universal property of the quotient}, the band morphism $k(\bar x)\to F$ factors uniquely through $k(\bar x)\to \bandquot{k(\bar x)}{P_x}=k_x$, as claimed.
\end{proof}

\begin{rem}
 Since every prime null ideal of $B$ is a prime ideal of the semiring $B^+$, the null space of a band $B$ embeds naturally as a subspace of the semiring spectrum $\Spec B^+$ of $B^+$. This subspace consists precisely of the semiring ideals in the closed subscheme of $\Spec B^+$ defined by $N_B$ that satisfy the substitution rule \ref{SR}. Since every semiring ideal of $B^+$ can be completed to a null ideal, the open coverings of $\Null(B)$ are the restrictions of open coverings of $\Spec B^+$, which shows that $\Null(B)$ is naturally equipped with a sheaf in semirings.
\end{rem}

\subsubsection{The kernel space as the image of the null space}
We conclude this section with the characterization of the kernel space as the image of the null space in $X$.

\begin{prop}\label{prop: k-space as image of the null space}
 Let $X$ be a band scheme with associated null space $X^\nul$. Then its kernel space $X^\kernel$ is the image of $\pi_X:X^\nul\to X$.
\end{prop}

\begin{proof}
 Since the assertion is local in $X$, we can assume that $X=\Spec B$ is affine. We first show that the image of $\pi_X:X^\nul\to X$ lies in $X^\kernel$. Let $\fp$ bea prime null ideal of $B$. By \autoref{prop: band quotients}, $\fp$ is the null kernel of the quotient map $\pi_\fp:B\to \bandgenquot B\fp$, and its kernel is $\fp\cap B=\pi_X(\fp)$. This shows that $\pi(\fp)$ is a $k$-ideal and thus $\im(\pi_X)\subset X^\kernel$.
 
 Conversely, consider a prime $k$-ideal $\fq$ of $B$ and the unique morphism $f:B\to\K$ with kernel $\fq$. The null kernel $\fp$ of $f$ is a null ideal by \autoref{prop: band quotients}. Since $N_\K=\N-\{1\}$ is a prime ideal of $\K^+=\N$, also $\fp$ is prime as the inverse image of $N_\K$ under the semiring morphism $B^+\to\K^+$. By definition, we have $\fq=\fp\cap B=\pi_X(\fp)$, which establishes the surjectivity of $\pi$.
\end{proof}

\subsection{Comparison of visualizations}
\label{subsection: comparison of visualizations}

In this final section, we collate the comparison results between the different visualizations and put them into a big picture. To get there, we introduce some notations and additional maps,

Let $k$ be a band, $X$ a $k$-scheme, and $F$ a topological idyll with a $k$-algebra structure. We denote the topology that we consider for $X(F)=\Hom_k(\Spec F,X)$ by a superscript: $X(F)^\weak$ for the weak Zariski topology, $X(F)^\str$ for the strong Zariski topology, and $X(F)^\fine$ for the fine topology. By \autoref{prop: Zariski topology} and \autoref{prop: the fine topology is finer than the Zariski topology}, the identity maps
\[
 \begin{tikzcd}[column sep=40]
  X(F)^\fine \ar[r,"\id"] & X(F)^\str \ar[r,"\id"] & X(F)^\weak
 \end{tikzcd}
\]
are continuous. 

Sending an $F$-point $\alpha:\Spec F\to X$ to its unique image point $\alpha(\gen 0)$ defines a map $\chi_{X,F}^\kernel:X(F)\to X^\kernel$. This map lifts to a map into the null space if we assume that the null set of $F$ is a prime ideal of $F^+$. This assumption is satisfied by fields, $\K$, $\S$ and $\T$. It is satisfied by a partial field if and only if its universal ring is an integral domain (cf.\ Section 1.1 in \cite{Baker-Lorscheid21} for a definition and Remark 3.12 for examples). In order to define the map into the null space, consider an $F$-point $\alpha:\Spec F\to X$ and an affine open $U$ of $X$ that contains the image of $\alpha$. Since $N_F$ is prime, the null kernel $(\Gamma\alpha^+)^{-1}(N_F)$ defines a point $\chi_{X,F}^\nul(\alpha)$ in $U^\nul\subset X^\nul$. Since inverse images commute with compositions, which applies, in particular, to finite localizations, the point $\chi_{X,F}^\nul(\alpha)$ of $X^\nul$ is independent of the choice of affine open $U$, which yields an intrinsic map $\chi_{X,F}^\nul:X(F)\to X^\nul$.

The inclusion $X^\Tits\hookrightarrow X^\kernel$ lifts to $X^\nul$ by sending a point $x\in X^\Tits$ to the unique closed point $\hat x$ in the fibre $\pi_X^{-1}(x)$. This allows us to consider $X^\Tits$ as a subspace of $X^\nul$.

\begin{thm}\label{thm: comparison of visualizations}
 Let $k$ be a band, $X$ a $k$-scheme, and $F$ a topological idyll with a $k$-algebra structure whose null set is a prime ideal of $F^+$. Then all maps in the diagram
\[
 \begin{tikzcd}[column sep=40]
  X(F)^\fine \ar[r,"\id"] & X(F)^\str \ar[r,"\chi_{X,F}^\nul"] \ar[d,"\id"'] & X^\nul \ar[->>,d,"\pi_X"]     & X^\Tits \ar[>->,l]  \\ 
                          & X(F)^\weak \ar[r,"\chi_{X,F}^\kernel"]           & X^\kernel \ar[>->,r] & \underline{X} 
 \end{tikzcd}
\]
are continuous and functorial in $X$, and the diagram commutes. 
\end{thm}

\begin{proof}
 The functoriality follows from the local construction of the various spaces and maps, which also allows us to assume that $X=\Spec B$ is affine. The diagram commutes since the kernel of a morphism $f:B\to F$ is the intersection of the null kernel of $f$ with $B$, as subsets of $B^+$.
 
 We know already that all maps of the diagram are continuous, except possibly $\chi_{X,F}^\kernel$ and $\chi_{X,F}^\nul$. The map $\chi_{X,F}^\kernel$ is continuous by the definitions of the weak topology for $X(F)$ and the topology of $X^\kernel$ as a subspace of $\underline{X}$.
 
 In order to prove that $\chi_{X,F}^\nul$ is continuous, consider a basic open $U_h=\{\fp\in X^\nul\mid h\notin h\}$ of $X^\nul$ where $h\in B$. Then
 \[
  (\chi_{X,F}^\nul)^{-1}(U_h) \ = \ \{f:B\to F\mid h\notin\nullker f\} \ = \ \{f:B\to F\mid f^+(h)\notin N_F\}
 \]
 is strong Zariski open by \autoref{prop: Zariski topology}, and thus $\chi_{X,F}^\nul$ is continuous as claimed.
\end{proof}


\begin{small}
 \bibliographystyle{plain}
 \bibliography{refences}
\end{small}


\end{document}